\documentclass[a4paper,11pt]{amsart}
\usepackage[utf8]{inputenc}

\usepackage[T1]{fontenc}
\usepackage[english]{babel}
\usepackage{color}
\usepackage{amsmath}
\usepackage{amssymb}
\usepackage{amsthm}
\usepackage{graphicx}
\usepackage[colorlinks=true, linkcolor=blue]{hyperref}

\usepackage[margin=1in]{geometry}

\usepackage{enumitem}   

\usepackage{lmodern}
\usepackage{microtype}

\usepackage{subfiles}

\usepackage{soul}

\usepackage[capitalise]{cleveref}

\usepackage[colorinlistoftodos,prependcaption,textsize=tiny]{todonotes}

\usepackage{cancel}

\newcommand {\R} {\mathbb{R}}
\newcommand {\Z} {\mathbb{Z}}
\newcommand {\T} {\mathbb{T}}

\newcommand {\p} {\partial}

\newcommand\dd{{\rm d}}
\newcommand\e{{\rm e}}

\newcommand{\calF}{\mathcal{F}}

\newcommand{\calI}{\mathcal{I}}

\newcommand{\calN}{\mathcal{N}}

\newcommand{\eps}{\varepsilon}

\newcommand{\ndiamond}{%
  \diamond\llap{\raisebox{0.1ex}{\scalebox{0.7}{$\cancel{\phantom{\diamond}}$}}}%
}

\DeclareMathOperator{\dive }{div}

\theoremstyle{plain}
\newtheorem{thm}{Theorem}
\newtheorem{prop}{Proposition}
\newtheorem{cor}{Corollary}

\newtheorem{lemma}{Lemma}[section]
\newtheorem{rem}{Remark}[section]

\newtheorem*{thm*}{Theorem}

\date{\today}

\begin{document}
\title[3D MHD Equations Around Couette Flow]{Transition and Stability of 3D MHD Around Couette Flow} 
\author{Niklas Knobel}
\begin{abstract}

We study the three-dimensional incompressible magnetohydrodynamic (MHD) equations near Couette flow with a constant magnetic field perpendicular to the shear plane. Couette flow induces mixing and generates magnetic induction, while the constant magnetic field stabilizes $z$-dependent modes. In contrast, the $z$-averaged magnetic field exhibits algebraic growth. Letting $\mu$ denote the inverse fluid and magnetic Reynolds numbers, we analyze how $\mu$ governs stability thresholds in Sobolev spaces.

We identify a nonlinear transient-growth regime characterized by the sharp threshold $5/6\le \gamma\le 1 $. For $x$-average–free initial data of size $\mu^\gamma$, solutions are nonlinearly stable; however, for certain initial data, the solution departs from the linear dynamics at rate $\mu^{\gamma-1}$ due to a first-order nonlinear instability. The exponent $\gamma = 5/6$ is optimal for the associated energy functional and cannot be improved in Sobolev spaces without secondary transient-growth mechanisms. Below this threshold, solutions necessarily transition away from the linear dynamics at a minimal rate.

As a consequence, the $3d$ results yield sharp Sobolev stability thresholds for the $2d$ MHD equations around Couette flow without a constant magnetic field. In particular, the threshold is strictly larger than in prior $2d$ results with a constant field, revealing destabilizing effects normally suppressed by a constant magnetic field. Crucially, this stabilization is restricted to the direction of the magnetic field.

\end{abstract}
\maketitle
\setcounter{tocdepth}{1}
\tableofcontents

\section{Introduction}
The goal of this article is to characterize the stability, instability, and long-time behavior of a three-dimensional conducting fluid near the planar Couette flow coupled with a uniform vertical magnetic field:
\begin{equation}\label{eq:equilibrium}
      V_s = (y,0,0), \qquad B_s =(0,0,\alpha),  
\end{equation}  
where $\alpha \in \mathbb{R} \setminus \{0\}$. The governing equations are the incompressible magnetohydrodynamic (MHD) equations, written for perturbations of the velocity field 
\[
  V=(V^x,V^y,V^z): \mathbb{R}_+ \times \mathbb{T} \times \mathbb{R} \times \mathbb{T} \to \mathbb{R}^3
\]
and magnetic field 
\[
  B=(B^x,B^y,B^z): \mathbb{R}_+ \times \mathbb{T} \times \mathbb{R} \times \mathbb{T} \to \mathbb{R}^3,
\]
as
\begin{equation}
\begin{aligned}
\label{MHD2}
    \partial_t V + y \partial_x V + V^y e_1 + (V \cdot \nabla) V + \nabla \Pi &= \mu \Delta V + \alpha \partial_z B + (B \cdot \nabla) B, \\
    \partial_t B + y \partial_x B - B^y e_1 + (V \cdot \nabla) B &= \mu \Delta B + \alpha \partial_z V + (B \cdot \nabla) V, \\
    \dive(V) = \dive(B) &= 0.
\end{aligned}
\end{equation}
Here, $\mu > 0$ is a dimensionless parameter representing both viscosity and resistivity, under the assumption that the magnetic Prandtl number equals one.

This system exhibits a competition between stabilizing mechanisms---such as inviscid mixing, enhanced dissipation, and magnetic tension---and destabilizing mechanisms, including linear and nonlinear transient growth. %Linear transient growth amplifies small disturbances. Beyond a first threshold, nonlinear effects trigger primary instabilities and nonlinear transient growth. Crossing a second threshold introduces secondary nonlinear interactions, leading to full system instability. 
Our main interest lies in the \emph{inviscid limit} $\mu \to 0$, where these competing effects are most pronounced.

We aim to describe the precise threshold between stability and instability, in terms of the Sobolev size of the initial perturbations $(V_{in}, B_{in})$ with which \eqref{MHD2} is initialized. We identify three qualitatively different regimes depending on rates of nonlinear transient growth:

%Linear transient growth amplifies perturbations through the non-normality of the linearized operator. Once the perturbation amplitude exceeds a first critical threshold, first-order nonlinear instabilities and associated nonlinear transient growth arise. If the system reaches a second critical threshold, higher-order nonlinear interactions emerge, leading to secondary instabilities and ultimately to the breakdown of stability in the system.

%Linear transient growth amplifies perturbations. At fist critical point, this triggers first order nonlinear instabilities leading to nonlinear transient growth. At a second critical point, this leads to second order nonliear stabilities leading to instability of the system. 

\medskip

\noindent $\diamond$ \emph{Linear dynamics %Linear transient growth%, Inviscid damping and enhanced dissipation
.}  
If $\|(V_{in}, B_{in})\|_{H^N} \leq \delta_1 \mu^{\gamma_1}$, then the solution is enslaved to the linearized dynamics of \eqref{MHD2}, which is characterized by linear transient growth, inviscid damping, and enhanced dissipation.

\medskip

\noindent $\diamond$ \emph{Nonlinear yet stable transient growth.}  
If $\delta_1 \mu^{\gamma_1} \leq \|(V_{in}, B_{in})\|_{H^N} \leq \delta_2 \mu^{\gamma_2}$, then due to a first order nonlinear instability the solution exhibits a nonlinear departure from the linearized dynamics---finite and quantifiable in terms of an inverse power of $\mu$. Nevertheless, asymptotic stability still holds. 

\medskip

\noindent $\diamond$ \emph{Secondary nonlinear instability and nonlinear regime.}  
If $\delta_2 \mu^{\gamma_2} \leq \|(V_{in}, B_{in})\|_{H^N}$, then nonlinear growth triggers further instabilities. The solution experiences dramatic growth away from the linear dynamics and enters a nonlinear regime.

\medskip

The constants $\delta_i > 0$ are not essential for our analysis, while the exponents $\gamma_i \geq 0$ play a crucial role. Much of the recent mathematical literature has focused on the first regime, aiming to determine the smallest exponent $\gamma_1 \geq 0$ for which the linear dynamics---dominated by linear transient growth, inviscid damping, and enhanced dissipation---also controls the full nonlinear evolution. This value $\gamma_1$ defines the so-called \emph{nonlinear stability threshold}. In the three-dimensional homogeneous Navier–Stokes equations with Sobolev regularity, this line of work was initiated in \cite{bedrossian2017stability} and has since developed into a rich and active research area (see Section~\ref{sub:litover} for further discussion). 

In contrast, the second and third regimes---those involving unstable or partially unstable dynamics---have received significantly less attention. The nonlinear yet stable transient growth regime may be viewed as a refinement of the stability threshold framework, capturing initial data that evade classical stability mechanisms yet remain asymptotically stable. To the best of our knowledge, this is the first instance of such behavior being rigorously identified in a fluid model.

The third regime, involving norm inflation and potential secondary instabilities, highlights the sharpness of the exponents $\gamma_1,\gamma_2$, by producing solutions that undergo growth sufficient to trigger secondary instabilities. Crucially, such transient growth does not necessarily imply long-time instability: the solution may still converge to equilibrium as $t \to \infty$. However, the growth takes it beyond the reach of current perturbative techniques, limiting our ability to describe its detailed evolution.

\subsection{Main results}
To state our main results, it is instructive to first examine the linearized dynamics associated with \eqref{MHD2}, given by the system
\begin{equation}
\begin{aligned}
\label{MHD2lin}
    \partial_t V + y \partial_x V + V^y e_1 + \nabla \Pi &= \mu \Delta V + \alpha \partial_z B , \\
    \partial_t B + y \partial_x B - B^y e_1   &= \mu \Delta B + \alpha \partial_z V , \\
    \dive (V) = \dive (B) &= 0.
\end{aligned}
\end{equation}
It is evident from the structure of \eqref{MHD2lin} that the behavior of general solutions differs significantly from that of solutions which are independent of the periodic $x$ or $z$ variables. Specifically, modes independent of $x$ do not experience the shearing induced by the term $y \partial_x$, while modes independent of $z$ evolve according to effectively decoupled dynamics. To reflect this distinction, we introduce the following notation for functions $K:\T \times \R \times \T \to \mathcal{V}$, where $\mathcal{V} \in \{\R, \R^3\}$: 
\begin{align}
    %K_{\diamond}(x,y)&= \frac{1}{2\pi}\int_{\T}K(x,y,z)\dd z,\\
    K_{\neq =}(x,y)&:= \frac{1}{2\pi}\int_{\T}K(x,y,z)\dd z-\frac{1}{4\pi^2}\int_{\T^2}K(x,y,z)\dd x\,\dd z, \label{eq:nox-zave}\\
    K_{\neq \neq}(x,y,z)&:=K(x,y,z)-\frac{1}{2\pi}\int_{\T}K(x,y,z)\dd x-K_{\neq =}(x,y).\label{eq:nox-noz}
\end{align}
In words, $ K_{\neq =} $ is the $z$-average of the nonzero $x$-modes (i.e., the projection onto $x$-nonzero, $z$-zero average), while $ K_{\neq \neq} $ denotes the projection onto nonzero modes in both $x$ and $z$. Additional notation is introduced in Section~\ref{sub:notation}. 

For the $z$-dependent modes, the interaction induced by the constant magnetic field stabilizes the equation. Taking the $z$-average of \eqref{MHD2lin} reduces the system to effectively two dimensions. In this projection:

\medskip

\noindent $\bullet$ $(V^z_{\neq =}, B^y_{\neq =}, B^z_{\neq =})$ obey linear advection-diffusion equations with Couette drift;

\medskip

\noindent $\bullet$ $B^x_{\neq =}$ is linearly forced by $B^y_{\neq =}$ and therefore exhibits linear transient growth;

\medskip

\noindent $\bullet$ $(V^x_{\neq =}, V^y_{\neq =})$ follow the $2d$ linearized Navier–Stokes equations.

\medskip

To analyze these two-dimensional components, we introduce the following auxiliary variables:
\begin{equation}\label{eq:adapted_unknown}
    F:=\frac{1}{2\pi}\int_{\T}\left(\p_y V^x-\p_xV^y\right)  \dd z,\qquad G:=-\Delta^{-1}\frac{1}{2\pi}\int_{\T}\left(\p_y B^x-\p_xB^y\right)  \dd z,  
\end{equation}
where $F$ is the vorticity associated with the horizontal $(x,y)$ components after taking the $z$-average and $G$ is the corresponding magnetic potential associated to the magnetic current\footnote{In this article we use the letters $F$ and $G$ instead of $\Omega$ and $\Phi$ since the $3d$ vorticity and magnetic potential are different if the $z$ component of the $z$-average is nonzero.}. We now state the following linear stability result:

\begin{prop}[Linearized stability]\label{prop:lin}
Let $(V, B) \in C([0,\infty); L^2(\T \times \R \times \T))$ solve \eqref{MHD2lin} with $x$-average free initial data $(V_{\mathrm{in}}, B_{\mathrm{in}}) \in H^2$. Then, for all $t > 0$: 
\medskip

\noindent $\bullet$ Doubly nonzero modes undergo inviscid damping and enhanced dissipation:
 \begin{equation}
        \Vert   (V_{\neq\neq}(t),  B_{\neq\neq}(t))\Vert_{L^2} +\langle t \rangle \Vert   (V^y_{\neq\neq}(t),  B^y_{\neq\neq}(t))\Vert_{L^2} \lesssim \e^{-\frac 1{12} \mu t^3 }\Vert (V_{in},B_{in})\Vert_{H^2} ,
    \end{equation} 

\noindent $\bullet$ The $z$-averaged, $x$-nonzero modes of $V$ undergo inviscid damping and  enhanced dissipation:
        \begin{equation}\label{eq:lin_zaverage}
        \Vert F_{\neq=}(t)\Vert_{L^2}+
         \langle t \rangle  \Vert   V_{\neq=}^x(t)\Vert_{L^2} + \langle t \rangle^2  \Vert   V_{\neq=}^y(t)\Vert_{L^2} + \Vert   V_{\neq=}^z(t)\Vert_{L^2} \lesssim   \e^{-\frac{1}{12} \mu t^3 }\Vert (V_{in},B_{in})\Vert_{H^2},
    \end{equation} 
 \noindent $\bullet$  Corresponding magnetic components undergo linear transient growth and  enhanced dissipation:
        \begin{equation}
       \Vert G_{\neq=}(t)\Vert_{L^2}+\langle t \rangle^{-1}  \Vert   B_{\neq=}^x(t)\Vert_{L^2} +  \Vert   (B^y_{\neq=}(t),B^z_{\neq=}(t))\Vert_{L^2}\lesssim    \e^{-\frac{1}{12} \mu t^3}\Vert  (V_{in},B_{in})\Vert_{H^2}.
    \end{equation} 
\end{prop}
Our first main result identifies a nonlinear stability threshold with an exponent
\begin{equation*}
  5/6 \le \gamma \le 1,
\end{equation*}
corresponding to the second regime described in the introduction, in which solutions are globally stable but exhibit nonlinear transient growth. In this range, the transient amplification is characterized by 
\begin{equation*}
   1-\gamma,
\end{equation*}
so that the adapted unknown $F$ undergoes nonlinear growth of order $\mu^{-(1-\gamma)}$ before relaxing under the dissipative dynamics.

The following theorem provides a quantitative description of this regime. In particular, it shows that solutions remain globally stable, while exhibiting a transient nonlinear growth quantified by $1-\gamma$.

\begin{thm}[Stable nonlinear transient growth]\label{thm:stab}
    Let $N\ge 8$ and $\alpha\neq 0$. Then there exist $c, \delta >0$ such that for all $5/6 \le \gamma\le 1$, $0<\mu \le 1$, and $x$-average free initial data $V_{in}, B_{in}$ such that
    \begin{align*}
         \mu^{1-\gamma} \|F_{in}\|_{H^{N}} + \|G_{in}\|_{H^{N+1}} + \|(V_{in},B_{in})_{\neq \neq }\|_{H^{N+1}} = \varepsilon \le \delta \mu^{\gamma},
    \end{align*}
    the corresponding solution $(V,B)$ to \eqref{MHD2} lies in $C([0,\infty); H^N(\T\times\R\times\T))$. Moreover, there exists a constant $c(\alpha)>0$ such that for all $t\ge0$:
    
    \medskip
    
\noindent $\bullet$ Doubly nonzero modes undergo inviscid damping and enhanced dissipation:
    \begin{align*}
        \|(V_{\neq\neq}(t),B_{\neq\neq}(t))\|_{L^2} + \langle t \rangle \|(V^y_{\neq\neq}(t),B^y_{\neq\neq}(t))\|_{L^2}
        \lesssim \varepsilon \e^{-c\mu^{\frac13}t}.
    \end{align*}

\noindent  $\bullet$ The  $z$-averaged, $x$-nonzero modes of $V$ undergo nonlinear transient growth,  inviscid damping, and  enhanced dissipation:
\begin{align*}
       \mu^{1-\gamma} \Big(\|F_{\neq =}(t)\|_{L^2} + \langle t \rangle \|V^x_{\neq =}(t)\|_{L^2} + \langle t \rangle^2 \|V^y_{\neq =}(t)\|_{L^2}\Big) 
       + \|V^z_{\neq =}(t)\|_{L^2} \lesssim \varepsilon \e^{-c\mu^{\frac13}t}.
\end{align*}

\noindent $\bullet$ Corresponding magnetic components undergo linear transient growth and  enhanced dissipation:
\begin{align*}
       \|G_{\neq =}(t)\|_{L^2} + \langle t \rangle^{-1}\|B^x_{\neq =}(t)\|_{L^2} 
       + \|(B^y_{\neq =}(t),B^z_{\neq =}(t))\|_{L^2} \lesssim \varepsilon \e^{-c\mu^{\frac13}t}.
\end{align*}     

\noindent $\bullet$ The critical energy functional stays bounded: 
\begin{align*}
        \mu^{1-\gamma}\|F(t,x+yt,y)\|_{H^N} + \|G(t,x+yt,y)\|_{H^{N+1}} 
        + \mu^{\frac12}\|\nabla G(\tau,x+y\tau,y)\|_{L^2_t H^{N+1}} \lesssim \varepsilon.
\end{align*}

\noindent $\bullet$ Nonlinear transient growth: there exist data with $F_{in}=0$ such that at some time $t_1\ge0$,
\begin{align*}
        \|F(t_1)\|_{L^2}\gtrsim \delta \mu^{-(1-\gamma)}\varepsilon.
\end{align*}
\end{thm}

Thus, Theorem~\ref{thm:stab} rigorously establishes a regime of stable nonlinear transient growth, with sharp dependence on $\gamma$. The extremal cases $\gamma=1$ and $\gamma=5/6$ interpolate between purely linear dynamics and the onset of nonlinear amplification of order $\mu^{-1/6}$. In particular:  
\medskip

\noindent $\bullet$ For $\gamma=1$, no nonlinear transient growth occurs, and the solution is determined by the linear dynamics with only a perturbative correction.  

\medskip

\noindent $\bullet$ For $\gamma=5/6$, the solution exhibits nonlinear transient growth of order $\mu^{-1/6}$, marking the maximal deviation from the purely linear regime while remaining asymptotically stable.  

\medskip

Prior works on the $3d$ MHD equations near Couette flow and tilted magnetic fields of the form $B_s=\alpha(\sigma,0,1)$---with $\alpha$ large and $\sigma$ rational or Diophantine---include~\cite{liss2020sobolev,rao2025stability,wang2025stabilitythreshold3dmhd}, and establish threshold exponents $\gamma_1 \le 1$. Our setting corresponds to $\sigma=0$ and yields the sharper threshold exponent $\gamma=5/6$. 

A natural question is whether the lower threshold $\gamma=5/6$ can be improved. Our second result shows this is not the case: the threshold is optimal for controlling the critical energy functional.%\red{[still very hard to read the statement below]}

\begin{thm}\label{thm:opts}
    The minimal $\gamma\ge 0$ for which Theorem \ref{thm:stab} holds is $\gamma =5/6$.
\end{thm}

Hence, the stability threshold of Theorem~\ref{thm:stab} is sharp: below $\gamma=5/6$, additional instability mechanisms must appear. Our final result identifies the minimal rate of transition away from the threshold of $5/6$.

\begin{prop}\label{thm:belows}
    For $N\ge8$ and $0<\gamma \le 5/6$ there exists initial data $\|(V_{in},B_{in})\|_{H^N} =\eps \le \delta \mu^\gamma$ such that 
    \begin{align*}
            \mu^{\frac16}\|F(t,x+yt,y)\|_{H^7} + \|G(t,x+yt,y)\|_{H^8} \gtrsim  \mu^{-\beta}\varepsilon
        \end{align*}
        with 
    \begin{align*}
        \beta \ge \min\!\Big(\tfrac{5}{6}-\gamma,\tfrac{1}{6}\Big).
    \end{align*}

\end{prop}

Theorem~\ref{thm:stab}-\ref{thm:opts} and Proposition~\ref{thm:belows} are simplified versions of the Theorems~\ref{thm:main}-\ref{thm:opt} and Proposition~\ref{thm:below} which we outline in Section \ref{sec:Thm}. 
\begin{rem}
We provide explicit initial data producing nonlinear transient growth or contradicting the assumed threshold (see Sections \ref{sec:nlgrow} and \ref{sec:nlgrow2}). The mechanism is driven by low-frequency perturbations of $G_{low}$ of the form
\begin{align}
    G_{low}(t,x,y) \approx \delta \mu^\gamma \cos(k(x+yt)) \int_{|\eta|\le C}\cos(\eta y)\,\dd\eta,\label{Glow}
\end{align}
for bounded frequency $k\approx 1$ on times $t\approx \mu^{-\frac13}$.  
\medskip

\noindent $\diamond$ If $\gamma\ge1$, the interactions of $G_{low}$ remain perturbative.  

\medskip

\noindent $\diamond$ For $\gamma<1$, the interaction of $G_{low}$ with another perturbation in $G$ induces transient growth in $F$.  

\medskip

\noindent $\diamond$ Below the threshold $\gamma<5/6$, a additional instability appears: the interaction of $G_{low}$ with a perturbation of $F$ produces transient growth in $G$. 

\medskip

In particular, for low-frequency $G$ of the form \eqref{Glow} with $\gamma=5/6$, the nonlinear $F$–$G$ coupling becomes crucial and the corresponding energy functional is critical. %Conversely, if $G$ satisfies the stronger condition $\gamma\ge 1$, we expect that the $z$-average evolves similarly to the 2D Navier–Stokes equations around Couette flow with a perturbation.

\end{rem}

\subsection{Change of Coordinates}
Theorem~\ref{thm:stab} is the simplified version of our main stability result, Theorems~\ref{thm:main}. 
The latter are formulated in variables adapted to the linear dynamics and measured in Sobolev spaces defined in the moving frame associated with the characteristics of the Couette flow:
\begin{equation}\label{eq:changelin}
    (x,y,z) \mapsto (x-yt,y,z).
\end{equation}
This coordinate change is standard in shear flow analysis, as it eliminates the linear advection term and makes the mixing mechanism explicit in the evolution of perturbations.
In this frame, the perturbations are given by
\begin{equation}
    v(t,x,y,z) = V(t,x+yt,y,z), \qquad b(t,x,y,z) = B(t,x+yt,y,z), \label{eq:defvb}
\end{equation}
and satisfy
\begin{align}\tag{3d-MHD}
\begin{split}
    \partial_t v + e_1 v^y - 2 \partial_x \nabla_t \Delta_t^{-1} v^y 
    &= \mu \Delta_t v + \alpha \partial_z b + (b \cdot \nabla_t) b - (v \cdot \nabla_t) v - \nabla_t \pi, \\
    \partial_t b - e_1 b^y 
    &= \mu \Delta_t b + \alpha \partial_z v + (b \cdot \nabla_t) v - (v \cdot \nabla_t) b, \label{MHD3} \\
    \nabla_t \cdot v&= \nabla_t \cdot b= 0,\\
     \Delta_t \pi &=\p_i^tb^j\p_j^tb^i -\p_i^tv^j\p_j^tv^i.
\end{split}\end{align}
The change of variables modifies the differential operators as
\[
\partial_y^t := \partial_y - t\partial_x, 
\quad \nabla_t := (\partial_x,\, \partial_y^t,\, \partial_z), 
\quad \Delta_t := \partial_x^2 + (\partial_y^t)^2 + \partial_z^2, 
\quad \Lambda_t := (-\Delta_t)^{1/2}.
\]
Throughout the paper, lowercase letters denote quantities in the moving frame, and uppercase letters refer to those in the original frame.

\subsection{Literature Overview}\label{sub:litover}

%The stability of Couette flow has a long history, dating back to the pioneering works of Lord Kelvin \cite{kelvin1887stability} and Reynolds \cite{reynolds1883xxix}. Since then, numerous studies have investigated its stability properties and the transition to turbulence.

%In recent years, significant progress has been made on the nonlinear stability of the Navier–Stokes equations, making it an active and evolving area of research.

%The first upper bound on a Sobolev stability threshold for Couette flow was obtained in \cite{bedrossian2016sobolev}, and later refined to the threshold of $
%1/3$ in \cite{masmoudi2022stability}. This bound can be further improved to hold uniformly in the viscosity parameter $\mu$ by employing Gevrey spaces \cite{li2025asymptotic,bedrossian2014enhanced} by exploiting the underlying inviscid dynamics studied in \cite{bedrossian2015inviscid,dengZ2019,dengmasmoudi2018}. At lower regularity, an optimal threshold of $1/2$ has been established in \cite{masmoudi2020enhanced}, with optimality dictated by the presence of exponential transient growth \cite{li2024dynamical}.

The stability of shear flows has a long history, beginning with the pioneering works of Lord Kelvin \cite{kelvin1887stability} and Reynolds \cite{reynolds1883xxix}. Since then, numerous experiments and applied papers have investigated its stability properties and the transition to turbulence. In recent years, significant progress has been made on the nonlinear stability of Couette flow, establishing this as an active and evolving area of research.

For the $2d$ Navier-Stokes equation around Couette flow, the first upper bound on a Sobolev stability threshold for Couette flow was obtained in \cite{bedrossian2016sobolev}, and later refined to the threshold of $1/3$ in \cite{masmoudi2022stability}. This bound can be further improved to hold uniformly in the viscosity parameter by employing Gevrey spaces \cite{li2025asymptotic,bedrossian2014enhanced} by exploiting the underlying inviscid dynamics studied in \cite{bedrossian2015inviscid,dengZ2019,dengmasmoudi2018}. At lower regularity, an optimal threshold of $1/2$ has been established in \cite{masmoudi2020enhanced}, with optimality dictated by the presence of exponential transient growth \cite{li2024dynamical}. The results on Couette flow are generalized to monotone shear flows under suitable spectral assumptions. The linear behavior in this setting was established in \cite{Zill3,Zhang2015inviscid}, while nonlinear stability is obtained in \cite{ionescu2020nonlinear,masmoudi2024nonlinear,beekie2024transition}.

The $2d$ MHD equations around Couette flow and constant magnetic field, stability thresholds are studied in  \cite{Dolce,chen2024sobolev,knobel2024sobolev,wang2024stability,jin2025stability,knobel2025suppression}, leading to a threshold of $1/3+$. Partially vanishing dissipation yields significantly different stability results. Here, the advected heat flow breaks the symmetry, leading to significantly different linear and nonlinear behaviour. See \cite{knobel2023echoes,zhao2024asymptotic} for the resistive and inviscid case, \cite{Knobelideal}
for the inviscid and non-resistive case, \cite{knobel2024nr} for the viscous and non-resistive case, and  \cite{knobel2025sobolev} for the spatially anisotropic case of horizontal resistivity.

For three-dimensional fluids, stability is more intricate since additional effects arise, and the threshold is expected to be significantly larger. For the Navier-Stokes equations around Couette flow, a cause of instability is an effect known as the lift-up effect. This produces linear algebraic growth of the energy by a magnitude of $\mu^{-1}$ \cite{reddy1998stability}. The nonlinear stability is studied in several works \cite{bedrossian2017stability, bedrossian2020dynamics, wei2021transition, bedrossian2022dynamics}, resulting in a Sobolev stability threshold of $1$, which seems to be optimal due to the lift-up effect. For stratified and rotating fluids, the lift-up is suppressed by the linear dynamics, and dispersive effects lead to an improvement of the threshold \cite{zelati2024stability,zelati2025stability}.

In $3d$ MHD equations around Couette flow, a constant magnetic field yields an interaction between velocity and magnetic field and drastically changes the dynamics. This has been studied in \cite{liss2020sobolev, rao2025stability, wang2025stabilitythreshold3dmhd} for constant magnetic fields of the form $\alpha(\sigma,0,1)^T$. Here, the lift-up effect in the $x$ average is suppressed by a magnetic field in the $z$ direction. However, the transport of the magnetic field by the Couette flow leads to a growing mechanism in the magnetic field, which competes with the coupling by the constant magnetic field.  In particular, for a constant magnetic field $\alpha(\sigma,0,1)^T$ when we consider Fourier frequencies $(k,l)$ of the $(x,z)$ variables such that $\sigma k+l \approx 0$ the magnetic field grows by a magnitude of $\mu^{-\frac 13 }$ in the linear dynamics. This is precisely where our contribution lies: for $\sigma =0$ the magnetic field growth shifts to the $l=0$ (i.e., z-averaged) mode. A detailed analysis of this interaction leads us to identify the stability threshold of 5/6 and a refined understanding of transient growth below the threshold of $1$.

The papers studying the stability of the MHD equations around Couette flow are embedded in a spectrum of papers investigating the stability of the MHD equations for different stationary states and geometries. We briefly summarize some relevant works. For the ideal MHD equations perturbed around a constant magnetic field, small localized perturbations are stable \cite{bardos1988longtime}. The extension to the fully dissipative setting is done in \cite{he2018global,wei2017global}, while results for partially dissipative systems are done in \cite{deng2018large,wu2021global} and references therein. For a strictly monotone positive magnetic field $B=(b(y),0,0)$, the linearized dynamics exhibit damping \cite{ren2017linear}. The interplay of a monotone magnetic field and shear flow can further give rise to the formation of so-called magnetic islands in the linearized dynamics \cite{ren2021long,zhai2021long}. In the geometry of two rotating cylinders, the MHD equations around Taylor–Couette flow are unstable  \cite{navarro2025nonlinear} due to exponential amplification associated with the Ponomarenko dynamo mechanism \cite{navarro2025spectral}.

In summary, it seems a constant magnetic field tends to stabilize perturbations, whereas its absence allows instabilities in the magnetic field to develop. This observation is confirmed in the present work: the presence of a constant magnetic field stabilizes the $z$-dependent modes, while the $z$-independent modes display both linear and nonlinear transient growth.

\subsection{Notations and Conventions}\label{sub:notation}

The zero modes in $x$ and its fluctuation are 
\begin{align}
    f_{=}(y,z) = \frac 1 {2\pi}\int_{\T} f(x,y,z)\dd x, \qquad 
    f_{\neq} =f-f_{=}. 
\end{align}
When two subscripts are present, we make a statement about the $z$-modes: The double-zero mode and its fluctuation with relation to the $x$-average are
\begin{equation}
    f_{==}(y) =  \frac 1 {2\pi}\int f_{=}(y,z)\dd z, \qquad f_{=\neq}=f_{=}-f_{==},
\end{equation}
while the nonzero-zero and nonzero-nonzero modes are
\begin{equation}
    f_{\neq=}(x,y) =  \frac 1 {2\pi}\int f_{\neq}(x,y,z)\dd z, \qquad f_{\neq\neq}=f_{\neq}-f_{\neq=}.
\end{equation}
Lastly, we need further notation to distinguish between the main unknowns
\begin{equation}
f_{\sim}= f- f_{\neq = },\qquad  f_\diamond(x,y) = \frac 1 {2\pi}\int_{\T} f(x,y,z)\dd z,\qquad f_{\ndiamond}=f- f_{ \diamond}.
\end{equation}
In particular, we obtain the splittings
\begin{align*}
    f=f_{\sim}+ f_{\neq = }=f_{\sim}+ f_\diamond - f_{==}. 
\end{align*}

We denote the time-dependent partial derivatives as 
\begin{align*}
    \p_y^t&=\p_y-t \p_x, & \nabla_t &=(\p_x, \p_y^t, \p_z),& 
    \Delta_t &=\p_x^2 + (\p_y^t)^2 + \p_z^2, & \Lambda_t &=(-\Delta_t)^{\frac 12 }.
\end{align*}
When summing over all derivatives, we write 
\begin{align*}
    \p_i^t = \begin{cases}
        \p_x & \text{ if }i=x,\\
        \p_y^t & \text{ if }i=y,\\
        \p_z & \text{ if }i=z. 
    \end{cases}
\end{align*}
Furthermore, we define the $2d$ derivatives for the $z$ independent unknowns  as  
\begin{align*}
    \tilde \nabla &=(\p_x, \p_y), &\tilde \nabla_t &=(\p_x, \p_y^t), &\tilde \nabla^{\perp}_t &=( \p_y^t ,-\p_x).
\end{align*}
The $2d$ derivatives for $x$ independent unknowns  as  
\begin{align*}
     \nabla_{y,z} &=(\p_y, \p_z).
\end{align*}
For a vector $v=(v^x,v^y, v^z)^T\in \R^3 $ we define
\begin{align*}
    v^{x,y} &=(v^x,v^y),&v^{y,z} &=(v^y,v^z).
\end{align*}
For a function $f:   \T\times \R \times \T\to \R$ we denote it's Fourier transform as 
\begin{align*}
    \hat f (k,\eta,l) &= \frac 1{(2\pi)^{\frac 3 2 }} \iiint_{ \T\times \R \times \T }\e^{i(kx+\eta y +l z)  }f(x,y,z) \dd x\dd y \dd z.
\end{align*}

 In most parts of the manuscript, we omit writing the hat of the Fourier transform and write $f(k,\eta,l)$. 
 
 \medskip
 
 We write for the spatial Lebesgue spaces $L^p=L^p(\T\times \R \times \T)$ and Sobolev spaces $H^N=H^N(\T\times \R \times \T)$. For the Lebesgue spaces with time dependence, we write $L^p_tH^N=L^p([0,t];H^N)$ and functions which map continuously in time into Sobolev spaces we write $C_tH^N=C([0,t];H^N )$.

\section{Effects and Heuristics}\label{sec:Heu}
In this section, we first discuss the dynamics of the $z$-average, which ultimately leads to the identification of the threshold $\gamma = 5/6$. Then we prove stability of the linearized equations leading to Proposition \ref{prop:lin}.

\subsection{The $2d$ MHD and the $z$-Average}\label{sec:zav}

To understand the dynamics of the $z$-average, we consider the $2d$ MHD system around Couette flow in the comoving frame:
\begin{align}\label{MHD2d}\tag{2d-MHD}\begin{split}
    \partial_t v + e_1 v^y - 2 \partial_x \tilde\nabla_t \Delta_t^{-1} v^y &= \mu \Delta_t v + (b \cdot \tilde\nabla_t)b - (v\cdot \tilde\nabla_t)v - \tilde\nabla_t \pi, \\
    \partial_t b - e_1 b^y &= \mu \Delta_t b + (b\cdot \tilde\nabla_t)v - (v\cdot \tilde\nabla_t)b,\\
    \tilde \nabla_t \cdot v= \tilde \nabla_t \cdot b&= 0,
\end{split}\end{align}
for $(t,x,y) \in \R_+ \times \T \times \R$ and $v,b: \ \R_+\times \T \times \R \to \R^2$\footnote{The $v,b$ are only $2d$ in this subsection.}.  

For solutions of \eqref{MHD3}, the $(x,y)$ components of the $z$-average satisfy equations of the form \eqref{MHD2d} with a forcing from the non-$z$-averaged modes. Each solution of \eqref{MHD2d} therefore represents a solution of \eqref{MHD3} that is independent of $z$. Consequently, understanding the dynamics of \eqref{MHD2d} is essential for understanding the $z$-averaged behavior of the full $3d$ equations.

The $2d$ MHD equations around Couette flow with an additional constant magnetic field have been studied in several recent works (see \cite{Dolce}). A constant magnetic field stabilizes the system by preventing the growth of the magnetic perturbation. In contrast, equations \eqref{MHD2} reduce to the $2d$ case with no constant magnetic field, in which the magnetic perturbation may grow in the linear dynamics. This leads to a significantly larger stability threshold.

To adapt to the growth of the magnetic field, we introduce the $2d$ vorticity and magnetic potential:
\begin{align*}
    f &= \partial_y^t v^x - \partial_x v^y, & 
    g &= \Lambda_t^{-2}(\partial_y^t b^x - \partial_x b^y).
\end{align*}
We use the letters $f$ and $g$ instead of the usual $w$ and $\phi$ for $2d$ systems, used for vorticity and magnetic potential, since in $3d$ vorticity and magnetic potential are different. The unknowns $f$ and $g$ satisfy the equations 
\begin{align}\begin{split}\label{MHD2fg}
    \partial_t f &= \mu \Delta_t f + (\tilde\nabla^\perp \Lambda_t^{-2} f \cdot \tilde\nabla) f 
                    + \tilde\nabla^\perp_t \cdot \big((\tilde\nabla^\perp g \cdot \tilde\nabla) \tilde\nabla_t^\perp g\big), \\
    \partial_t g &= \mu \Delta_t g + (\tilde\nabla^\perp \Lambda_t^{-2} f \cdot \tilde\nabla) g.
\end{split}\end{align}
For large $g$ the term
\[
\tilde\nabla^\perp_t \cdot \big((\tilde\nabla^\perp g \cdot \tilde\nabla) \tilde\nabla_t^\perp g\big) = (\tilde\nabla^\perp g \cdot \tilde\nabla) \Delta_t g
\]
is the most problematic: it grows in time and breaks the structural stability exploited in \cite{knobel2025suppression} to obtain thresholds of $1/3$.  

For the $2d$ MHD system \eqref{MHD2fg}, our analysis yields the following consequences:

\begin{cor}[Stability of $f$ and $g$]\label{cor:fgh}
    Let $N\ge 7$. There exists $\delta >0$ such that for all $5/6 \le \gamma \le 1$ and $0 < \mu \le 1$, the following stability holds:  
    Consider $x$-average–free initial data satisfying
    \begin{align*}
        \mu^{\beta_1} \| f_{in}\|_{H^N} + \| g_{in}\|_{H^{N+1}} = \varepsilon \le \delta \mu^{\gamma}.
    \end{align*}
    Then the corresponding solution $(f,g)$ of \eqref{MHD2fg} belongs to $C([0,\infty);H^N)$ and satisfies the critical energy estimate
    \begin{align*}
        \mu^{1-\gamma}\| f\|_{L^\infty_t H^N}
        + \| g \|_{L^\infty_t H^{N+1}}
        + \mu^{\tfrac{1}{2}} \| \nabla_t g\|_{L^2_t H^{N+1}}
        \lesssim \varepsilon.
    \end{align*}
\end{cor}

\begin{cor}[Optimality the energy functional]\label{cor:b1f}
    Let $N \ge 8$. Consider $0 \le \gamma \le 5/6$ and $0\le \beta_1 \le 1/4$, and suppose the following stability statement holds:
    \begin{itemize}
        \item[] There exist constants $0 < \delta, \mu_0 \le 1$ such that the following holds uniformly in $0 < \mu \le \mu_0$:  
        for initial data $(f_{in}, g_{in})$ satisfying
        \begin{align*}
            \mu^{\beta_1}\| f_{in}\|_{H^N} + \| g_{in}\|_{H^{N+1}} = \varepsilon \le \delta \mu^\gamma,
        \end{align*}
        the corresponding solution $(f,g)$ of \eqref{MHD2fg} satisfies 
        \begin{align*}
            \mu^{\beta_1}\| f\|_{L^\infty_t H^N}
            + \| g\|_{L^\infty_t H^{ N+1}}
            + \mu^{\frac{1}{2}} \| \nabla_\tau g\|_{L^2_t H^{N+1}}
            \lesssim \varepsilon,
        \end{align*}
        for all $t \le \mu^{-\frac{1}{3}}$.
    \end{itemize}
    Then necessarily
    \[
        \gamma = \frac{5}{6}, \qquad \beta_1 = \frac{1}{6}.
    \]
\end{cor}

\begin{cor}[Below the threshold $5/6$]\label{cor:opt}
    Let $N \ge 8$. Consider $0 \le \gamma \le 5/6$ and $0\le \beta$, and suppose the following stability statement holds:
    \begin{itemize}
        \item[] There exist constants $0 < \delta, \mu_0 \le 1$ such that the following holds uniformly in $0 < \mu \le \mu_0$:  
        for initial data $(f_{in}, g_{in})$ satisfying
        \[
            \| f_{in}, g_{in}\|_{H^N} = \varepsilon \le \delta \mu^\gamma,
        \]
        the corresponding solution $(f,g)$ of \eqref{MHD2fg} satisfies 
        \begin{align*}
            \mu^{\tfrac{1}{6}} \| f \|_{L^\infty_t H^7}
            + \| g \|_{L^\infty_t H^8}
            \lesssim \mu^{-\beta}\varepsilon,
        \end{align*}
        for all $t \le \mu^{-\tfrac{1}{3}}$. 
    \end{itemize}
    Then it follows that 
    \[
        \beta \ge \min\!\left(\tfrac{5}{6} - \gamma, \tfrac{1}{6}\right).
    \]
\end{cor}

Corollary~\ref{cor:fgh} follows directly from Theorem~\ref{thm:main}, while Corollaries~\ref{cor:b1f} and~\ref{cor:opt} are proven in Sections~\ref{sec:nlgrow} and~\ref{sec:nlgrow2}, respectively. In Subsection~\ref{sec:heuz}, we provide heuristics explaining why the $2d$ threshold is optimal.

\begin{rem}
    Corollary~\ref{cor:fgh} complements the result of \cite{knobel2025suppression}, together providing stability thresholds for the $2d$ MHD equations with %\emph{any}
    constant magnetic field configuration. A striking contrast emerges: in the presence of a nonzero constant magnetic field, the equations exhibit improved stability, whereas in the vanishing-field case the threshold deteriorates substantially. 

    An analogous---though reversed---phenomenon is known for the Boussinesq equations linearized around Couette flow with an affine temperature profile: as the slope of the affine profile vanishes, the long-time stability and inviscid daming rates behavior changes qualitatively compared to the strictly affine case, \cite{yang2018linear, niu2024improved,enciso2025linear}. %In \eqref{enciso2025linear} authors classify stratification regimes around monoton shear flows depending on the local Richardsson number. 

    For the $2d$ MHD equations, it remains an intriguing open direction to determine precisely how the balance between dissipation and the magnitude of the constant magnetic field governs the linear dynamics and the associated nonlinear thresholds.
\end{rem}

\subsection{Heuristics on the threshold of the $z$-average}\label{sec:heuz} 
In this subsection, we give a heuristic explanation for why the stability threshold in Corollary~\ref{cor:fgh} is optimal. The argument follows the approach developed in \cite{bedrossian2015inviscid,dengZ2019} and \cite{knobel2025suppression} for general fluid models.  

We take the Fourier transform in $(x,y)$, with frequency variables $(k,\eta)$, and omit the hat notation in what follows. The functions $f$ and $g$ exhibit strong interactions at the so-called \emph{critical times} 
\[
t_k = \frac{\eta}{k},
\]
at which the following resonance chain can occur:
\[
g(k+1,\eta)\quad \longrightarrow \quad f(k,\eta)\quad \longrightarrow \quad g(k-1,\eta).
\]
Since $t_k \leq t_{k-2}$, growth produced by $g(k+1,\eta)$ at mode $g(k-1,\eta)$ can be iterated and thus cascaded to lower modes, leading to further amplification.  

For times $t\approx \mu^{-\frac{1}{3}}$, frequencies $\eta \leq \mu^{-\frac{1}{3}}$, and bounded $k\leq C$, dissipation plays only a negligible role. Hence, for the purposes of this heuristic, we ignore the effect of dissipation on high frequencies.  
We use the identity 
\begin{align*}
     \nabla^\perp_t\cdot \bigl((\nabla^\perp g\cdot \nabla)\nabla^\perp_t g \bigr)
     = (\nabla^\perp g\cdot \nabla)\Delta_t g,
\end{align*}
and focus on high--low interactions of the form
\begin{align}
    \p_t f^{hi}  \approx (\nabla^\perp g^{hi}\cdot \nabla)\Delta_t g^{low},   \qquad \p_t g^{hi}  \approx (\nabla^\perp \Lambda_t^{-2} f^{hi}\cdot \nabla ) g^{low}.
\end{align}
As a toy model, we take $g_{low}(t,x) = -\eps \cos(x)$, which solves \eqref{MHD2fg} in the inviscid case $\mu=0$. Substituting this ansatz gives the simplified system
\begin{align} 
    \p_t f  = (1+t^2) \sin(x)\,\p_y g,   \qquad \p_t g  =- \sin(x)\,\p_y \Lambda_t^{-2} f.
\end{align}
After Fourier transform, this reads
\begin{align}\begin{split}
    \p_t f(t,k,\eta)   &= \pm \langle t\rangle^2 \eta \, g (t,k\pm 1 , \eta),   \\
    \p_t g(t,k,\eta)   &= \mp \frac{\eta}{(k\pm 1)^2}\,\frac{1}{\langle t-\tfrac{\eta}{k\pm 1}\rangle^2}\, f(t,k\pm 1 , \eta).
\end{split}\end{align}
Fix $\eta$ and restrict to the interacting modes $g(k+1), f(k), g(k-1)$. The dominant transfer is $g(k+1)\to f(k)\to g(k-1)$, which gives
\begin{align}
    \p_t g(t,k+1 ) = 0,\qquad \p_t f(t,k)   \approx  t^2 \eta \, g (t,k+1 ),   \qquad \p_t g (t,k-1)  \approx  \frac{\eta}{k^2}\frac{1}{\langle t-\frac{\eta}{k}\rangle^2} f(t,k).
\end{align}
The relevant interaction appears fro times $t\approx \frac \eta k$, i.e. on the interval $t_{k+1}\le t\le t_{k-1}$. Starting from $f(t_{k+1},k)=g(t_{k+1},k-1)=0$, integration yields
\begin{align*}
     f(t,k)  \approx \eps \eta g(t_{k+1},k+1) \int_{t_{k+1}}^t \tau^2\, \dd\tau  \approx \eps \eta g(t_{k+1},k+1) \, t^2 (t-t_{k+1})
\end{align*}
and hence
\begin{align*}
     g(t_{k-1}, k-1) &\approx \eps \frac{\eta}{k^2}\int_{t_{k+1}}^{t_{k-1}} \frac{1}{\langle t-\tfrac{\eta}{k}\rangle^2} f(\tau,k)\, \dd\tau \approx 
     \eps^2g(t_{k+1},k+1) t_k^4 (t_k-t_{k+1}) .
\end{align*}
Since $t_{k}-t_{k+1}\approx k^{-1}t_{k}\approx t_{k}\approx  \mu^{-\frac13}$, we obtain
\begin{align}
    f(t_{k-1},k)   &\approx  \eps \eta \mu^{-1} g(t_{k+1},k+1), \label{eq:fheu}\\
    g(t_{k-1},k-1) &\approx (\eps \mu^{-\frac56})^2 g(t_{k+1},k+1). \label{eq:gheu}
\end{align}
From \eqref{eq:gheu} we see that echo chains are suppressed provided
\[
\eps \leq \delta \mu^{\frac56}
\]
for some small constant $\delta>0$. Meanwhile, \eqref{eq:fheu} shows that $g$ induces growth of $f$ by a factor $\eta \mu^{-\frac16}$. In Sobolev norms, this corresponds to 
\[
\|f\|_{H^{N}} \approx \mu^{-\frac16} \|\p_y g\|_{H^{N}}.
\]
Stability at this threshold is therefore achieved by employing the modified energy functional
\begin{align*}
    \|g\|_{H^{N+1}} + \mu^{\frac16}\|f\|_{H^{N}}.
\end{align*}

\subsection{Linear dynamics and Proof of Proposition \ref{prop:lin}}
We now prove the linear stability estimates for the system \eqref{MHD2lin}.  
The linearized equations in the comoving frame read
\begin{align}\begin{split}
    \p_t v+e_1 v^y -2 \p_x \nabla_t \Delta_t^{-1} v^y &=\mu \Delta_t v+ \alpha \p_z b,  \\
     \p_t b-e_1 b^y  &=\mu \Delta_t b+\alpha  \p_z v,\\
     \nabla_t \cdot v&=\nabla_t \cdot b=0.\label{MHD2lin2}
\end{split}
\end{align}
From the structure of \eqref{MHD2lin2}, we have that modes behave differently if they depend or are independent of the $x$ or $z$ variable. In particular, for $z$ dependent modes, the interaction induced by the constant magnetic field suppresses the growth of the magnetic field by using the decay of the velocity field (also called inviscid damping).

When we consider the $z$-averaged component, the constant magnetic field vanishes, and we adapt to new unknowns. We define
\begin{align}\label{eq:fghlin}
    f=\p_y^t v^x_{\diamond}- \p_x v^y_{\diamond},\qquad
    g=\Lambda_t^{-2} (\p_y^t b^x_{\diamond}- \p_x b^y_{\diamond}),\qquad
    h=\begin{pmatrix} v^z_{\diamond}\\ b^z_{\diamond}\end{pmatrix}.
\end{align}
These unknowns satisfy
\[
    \p_t f  =\mu \Delta_t f, \qquad 
    \p_t g  =\mu \Delta_t g, \qquad 
    \p_t h  =\mu \Delta_t h,
\]
hence each is governed by a pure advection–diffusion equation. In Fourier space,
\[
    \hat f(t,k,\eta)=\exp\left(-\mu \int_0^t |k,\eta-k\tau ,l|^2 \,\dd\tau\right) \hat f_{in}(k,\eta).
\]
Evaluating the integral gives
\begin{align}
     \int_0^t |k,\eta-k\tau ,l|^2 \, \dd\tau
     %=(k^2+l^2)t +t\left(\frac{k^2t^2 }{3}-k\eta t+\eta^2 \right) 
     =(k^2+l^2)t +t\left(\frac{k^2t^2 }{12} +\left(\frac{kt}{2} - \eta\right)^2 \right)
     \ge \frac{k^2}{12} t^3.\label{eq:advheat}
\end{align}
Thus, for $k\neq 0$,
\[
    |\hat f(t,k,\eta)| \le \e^{-\frac{1}{12}\mu k^2 t^3 }|\hat f_{in}(k,\eta)|.
\]
In particular,
\[
    \| f_{\neq} \|_{H^N}\le \e^{-\frac{1}{12}\mu t^3 }\|f_{in}\|_{H^N},
\]
and the same estimate holds for $g$ and $h$. Reconstructing $v_{\neq =}$ and $b_{\neq =}$ from $(f,g,h)$ through \eqref{eq:fghlin} gives the desired $z$-average bounds \eqref{eq:lin_zaverage}.

For functions dependent on $z$, we use the interaction introduced by the constant magnetic field $\alpha$. Modifying the unknowns of \cite{knobel2025suppression}, we introduce the $\alpha$ stabelized unknowns
\[
    \tilde v  = v_{\sim } +(\alpha \p_z )^{-1} e_1 b^y_{\sim}, \qquad\qquad
    \tilde b  = b_{\sim}.
\]
Then $(\tilde v,\tilde b)$ is equivalent to $(v_\sim,b_\sim)$ in Sobolev norms and since $(b_{\sim}^y)_{\diamond}=(b_{\sim}^y)_{==}=0$ the $\tilde v$ is well defined. The equations become
\begin{align*}
    \p_t \tilde v &= \mu \Delta_t \tilde v + \alpha \p_z \tilde  b +2\nabla_t \Delta^{-1}_t \p_x  \tilde v^y,\\
    \p_t \tilde b &= \mu \Delta_t  \tilde b+ \alpha \p_z \tilde v,\\
    \nabla_t \cdot \tilde v &= \p_x(\alpha \p_z )^{-1}  \tilde b^y,\qquad \nabla_t \cdot \tilde b =0.
\end{align*}
The system is no longer incompressible: the divergence term interacts with the pressure-like contribution $2\nabla_t \Delta^{-1}_t \p_x  \tilde v^y$ and can, in principle, cause finite growth. To control this, we employ a time-dependent Fourier multiplier $A_L$, defined for $a\in L^2$ by
\[
    \calF (A_L a )(k,\eta,l ) 
    = \exp\left( -\int_0^t \mu |k,\eta-k\tau ,l |^2 
      +\frac{2}{\alpha} \Bigl| \frac {k^2 }{l}\Bigr| 
      \frac{\mathbf{1}_{k,l\neq 0}}{|k,\eta-k\tau,l |^2} \, \dd\tau \right)\hat a(k,\eta,l).
\]
By construction, the energy $ \|A_L(\tilde v , \tilde b )\|_{H^2 }^2$ is non-increasing. Differentiating in time gives
\begin{align*}
    \frac12 \p_t \|A_L(\tilde v , \tilde b )\|_{H^2 }^2 
    &= \langle A_L \tilde v , A_L 2\nabla_t \p_x \tilde v^y \rangle_{H^2} 
       - \frac{2}{\alpha}\Bigl\| \frac{|\p_x|}{|\p_z|^{1/2}}\Lambda_t^{-1}A_L(\tilde v,\tilde b)\Bigr\|_{H^2}^2.
\end{align*}
An integration by parts shows
\[
    \langle A_L \tilde v , A_L 2\nabla_t \p_x \tilde v^y \rangle_{H^2} 
    \le \frac{2}{\alpha}\Bigl\| \frac{|\p_x|}{|\p_z|^{1/2}}\Lambda_t^{-1}A_L(\tilde v,\tilde b)\Bigr\|_{H^2}^2,
\]
hence the energy is non-increasing. Therefore
\[
    \|A_L(\tilde v ,\tilde b )\|_{H^2} \lesssim \|(v_{in},b_{in})\|_{H^2}.
\]
Combining with \eqref{eq:advheat} yields
\[
    \|v_{\neq\neq},b_{\neq\neq}\|_{H^2}
    \lesssim \e^{-\frac{1}{12}\mu t^3 }\|(v_{in}, b_{in})\|_{H^2}.
\]
Finally, to obtain decay of the $y$-components we use incompressibility $\nabla_t \cdot v=0$:
\begin{align*}
    \| v^y_{\neq\neq}\|_{L^2}
    &= \left\| \frac{1}{\sqrt{\p_x^2 + (\p_y^t)^2}}
       \begin{pmatrix}\p_x \\ \p_y^t\end{pmatrix} v^y_{\neq\neq}\right\|_{L^2} \\
    &= \left\| \frac{1}{\sqrt{\p_x^2 + (\p_y^t)^2}}
       \begin{pmatrix}\p_x v^y_{\neq\neq}\\ -\p_x v^x_{\neq\neq}-\p_z v^z_{\neq\neq}\end{pmatrix}\right\|_{L^2}\\
    &\lesssim \langle t\rangle^{-1}\|v\|_{H^2}
     \lesssim \langle t\rangle^{-1} \e^{-\frac{1}{12}\mu t^3 }\|(v_{in},b_{in})\|_{H^2},
\end{align*}
and the same bound holds for $b^y_{\neq\neq}$. This completes the proof of Proposition~\ref{prop:lin}.

%\subsection{The $x$ average part of the MHD equations }
%Here we could give a heuristic why the $x$ average needs a threshold of $1$. Is that relevant enough?

\section{Main Theorems and Energies} \label{sec:Thm}
%Similar to $2$ dimensional fluid equations around Couette flow \cite{knobel2025suppression}, 

To track the dynamics of equation \eqref{MHD3}, we use adapted unknowns corresponding to the linear dynamics. In this section, we recall these unknowns, state the stability theorem, define the weights and energies, and establish the energy estimates we require to prove the stability theorem.  We recall that $(v,b)$ are defined in \eqref{eq:defvb} and the subscripts $=$, $\neq$, $\sim$ and $\diamond$ are defined in Subsection \ref{sub:notation}. 

%\noindent The \textbf{non  $z$-average unknowns} state\\
%\noindent The \textbf{non ($x$ dependent $z$-average) unknowns} state\\
%\noindent The \textbf{non ($x$ dependent $z$-average) unknowns} state\\
%\noindent The \textbf{ $\sim$ unknowns}, which exclude the growth in the $z$-average state 

\noindent The \textbf{$\alpha$ stabelized unknowns} state
\begin{align*}
    \tilde v  &= v_{\sim  } +(\alpha \p_z )^{-1} e_1 b^y_{\sim},&
    \tilde b  &= b_{\sim  }.
\end{align*}
The \textbf{$z$-average unknowns} state
\begin{align*}
    f&=\p_y^t v^x_{\diamond  }- \p_x v^y_{\diamond  }, &
    g&=\Lambda_t^{-2} (\p_y^t b^x_{\diamond  }- \p_x b^y_{\diamond  }), &
    h&=\begin{pmatrix}
        h_1 \\ h_2 
    \end{pmatrix}= \begin{pmatrix}
        v^z_{\diamond   }\\b^z_{\diamond   }
    \end{pmatrix}.
\end{align*}
For these unknowns, we obtain the equations
\begin{align}
\label{meq1}
&\begin{cases}
    \p_t \tilde v = \mu \Delta_t \tilde v + \alpha \p_z \tilde  b +2\nabla_t \Delta^{-1}_t \p_x  \tilde v^y +((b\cdot\nabla_t) b -(v\cdot\nabla_t) v -\nabla_t\pi)_{\sim}\\
    \qquad + e_1(\alpha \p_z)^{-1}  ((b\cdot\nabla_t) v^y - (v\cdot\nabla_t) b^y)_{\neq\neq},\\
    \p_t \tilde b= \mu \Delta_t  \tilde b+ \alpha \p_z \tilde v+((b\cdot\nabla_t) v -(v\cdot\nabla_t) b )_{\sim},\\
    \nabla_t \cdot \tilde v = \p_x(\alpha \p_z )^{-1}  b^y,\qquad \nabla_t \cdot  \tilde b =0,    \\
    \Delta_t\pi = \p_i^tb^j\p_j^tb^i -\p_i^tv^j\p_j^tv^i,
\end{cases}\\
\label{meq2}&\begin{cases}
    \p_t f  =\mu    \Delta_t f + \tilde \nabla_t^\perp \cdot ((b\cdot \nabla_t) b^{x,y} -(v\cdot\nabla_t) v^{x,y})_{\diamond},  \\
     \p_t g  =\mu \Delta_t g     + \Lambda_{t}^{-2}\tilde \nabla_t^\perp\cdot  ((b\cdot\nabla_t) v^{x,y}  -(v\cdot\nabla_t) b^{x,y})_{\diamond}, \\
     \p_t h = \mu \Delta_t h+  \begin{pmatrix}
        ((b\cdot\nabla_t)b^z-(v\cdot \nabla_t)v^z)_{\diamond}\\
        ((b\cdot \nabla_t)v^z -(v\cdot \nabla_t)b^z)_{\diamond}
    \end{pmatrix},\\
    \end{cases}\\
    \label{meq=}&\begin{cases}
    \tilde v_{==}= \p_y\p_y^{-2} f_= e_1 +h_{1,=} e_3,\\
    \tilde b_{==}= -\p_y g_= e_1 +h_{2,=} e_3.
    \end{cases}
\end{align}
If the equations \eqref{meq=} hold initially, they hold for all times. %Under sufficient regularity requirements a solution $(\tilde v,\tilde b, f,g,h)$ of  \eqref{meq1}, \eqref{meq2} and \eqref{meq=} is equivalent to a solution $(v,b)$ of \eqref{MHD3}. 
We now state the stability theorem:
\begin{thm}[Stable nonlinear transient growth region]\label{thm:main}
        Let  $N\ge 7$, $\alpha\neq 0$  there exists a $\delta >0$ such that for all $5/6\le\gamma\le 1 $, $0<\mu \le 1 $ and $x$-average free initial data $(v_{in},b_{in})$ for which the adapted unknowns satisfy
    \begin{align*}
        \Vert   \tilde v_{in},  \tilde b_{in},  g_{in},  h_{in}\Vert_{H^{N+1}}+\mu^{(1-\gamma) }\Vert f_{in}\Vert_{H^N} =\eps  &\le \delta \mu^{\gamma },
    \end{align*}
    the corresponding solution $(v, b)$ to \eqref{MHD3} is in $C([0,\infty);H^N)$. For some constant  $c(\alpha)>0$ the solution satisfies the following energy estimates for all $t\ge 0$:
    
    \medskip
    
\noindent $\bullet$ Stability of adapted unknown:
    \begin{align*}
        \Vert   \tilde v,  \tilde b, g,  h\Vert_{H^{N}}+\mu^{1-\gamma }\Vert f\Vert_{H^N}  &\lesssim \eps,
    \end{align*}
    
\noindent $\bullet$ Doubly nonzero modes undergo inviscid damping and enhanced dissipation:
    \begin{align*}
        \Vert   v_{\neq \neq},   b_{\neq \neq}\Vert_{L^2 }+\langle t\rangle \Vert    v^y_{\neq \neq},   b_{\neq \neq}^y\Vert_{L^2 } &\lesssim \eps e^{-c\mu^{\frac1 3 } t },
    \end{align*}
    
\noindent $\bullet$ The $z$-averaged velocity undergoes nonlinear transient growth,  inviscid damping, and  enhanced dissipation:
\begin{align*}
        \mu^{1-\gamma} (\Vert f_{\neq}\Vert_{L^2}+\langle t\rangle \Vert    v^x_{\neq =}\Vert_{L^2 }+ \langle t\rangle ^2 \Vert    v^y_{\neq =}\Vert_{L^2 })+\Vert  v^z_{\neq =}\Vert_{L^2 } &\lesssim \eps e^{-c\mu^{\frac1 3 } t },
    \end{align*}

\noindent $\bullet$  Corresponding magnetic components undergo linear transient growth and  enhanced dissipation:
    \begin{align*}
        \Vert g_{\neq}\Vert_{L^2}+ \langle t\rangle ^{-1} \Vert    b^x_{\neq =}\Vert_{L^2 }+ \Vert    b^{y}_{\neq =},    b^{z}_{\neq =}\Vert_{L^2 }&\lesssim \eps e^{-c\mu^{\frac1 3 } t },
    \end{align*}
    
\noindent $\bullet$  The critical energy functional stays bounded:
    \begin{align*}
        \mu^{1-\gamma} \Vert f\Vert_{L^\infty_t H^N}+\Vert g\Vert_{L^\infty_t H^{N+1}}+\mu^{\frac 12 }\Vert \nabla_\tau g\Vert_{L^2_t H^{N+1}}&\lesssim \eps,
    \end{align*}
\noindent $\bullet$ Linear transient growth of the magnetic field: There exists initial data such that for times $t\le \frac 1 {10} \mu^{-\frac 1 3 } $ it holds
    \begin{align*}
        \Vert    b^x_{\neq =}\Vert_{L^2 }&\gtrsim t \eps,
    \end{align*} 
    %In particular,  for the time $T=\frac 1{10}\mu^{\frac 13 }$, it holds that $ \Vert    b^x_{\neq =}\Vert_{L^2 }(T)\gtrsim \mu^{-\frac 13 }  \delta^2\eps $.\\
    
\noindent $\bullet$  Nonlinear growth in the $z$-average: There exists a $\mu_0>0$ and initial data satisfying $f_{in}=0$ such that for $\mu\le \mu_0$ and a time  $t_1>0$ it holds 
    \begin{align*}
        \Vert f(t_1)\Vert_{L^2} \gtrsim\delta  \mu^{-(1-\gamma)  }   \eps.
    \end{align*}
\end{thm}
The precise version of the optimality of the energy functional states: 
\begin{thm}[Optimality of energy functional]\label{thm:opt}
    Let $N \ge 8$. Consider $0 \le \gamma \le 5/6$ and $0 \le \beta_1 \le 1/4$, and assume the following stability statement holds:
\begin{itemize}
        \item[] There exist constants $0 < \delta, \mu_0 \le 1$ such that the following holds uniformly in $0 < \mu \le \mu_0$:  
        Consider, $z$ independent and $x$-average free initial data, for which $(F_{in}, G_{in})$ satisfy the bound
        \begin{align*}
            \mu^{\beta_1}\|f_{in}\|_{H^N} + \|g_{in}\|_{H^{N+1}} = \varepsilon \le \delta \mu^\gamma.
        \end{align*}
        Then the adapted unknowns $(f,g)$ in \eqref{eq:adapted_unknown} of the solution $(v,b)$ to \eqref{MHD2} satisfy
        \begin{align*}
            \mu^{\beta_1}\|f\|_{H^N}
            + \|g\|_{H^{N+1}}
            + \mu^{\frac{1}{2}} \|\nabla_\tau g\|_{L^2_t H^{N+1}}
            \lesssim \varepsilon
        \end{align*}
        for all times $t\le \mu^{-\frac 13 }$. 
\end{itemize}
    
    Then it holds that 
    \begin{align*}
        \gamma = \frac{5}{6}, \qquad \beta_1 = \frac{1}{6}.
    \end{align*}
\end{thm}
The minimal transition away from the threshold states: 

\begin{prop}[Below the threshold of $5/6$]\label{thm:below}
    Let $N\ge8$. Consider $0\le \gamma\le 5/6$ and $0\le \beta$, and assume the following stability statement holds:
\begin{itemize}
        \item[]
There exist constants $0<\delta,\mu_0\le1$ such the following holds uniformly for $0<\mu\le\mu_0$: Consider the $x$-average free initial data satisfying 
        \begin{align*}
            \|(v_{in},b_{in})\|_{H^N} = \varepsilon \le \delta \mu^\gamma.
        \end{align*}
        Then the adapted unknowns $(f,g)$ in \eqref{eq:adapted_unknown} of the solution $(v,b)$ to \eqref{MHD2} satisfy
        \begin{align*}
            \mu^{\frac16}\|f\|_{H^7} + \|g\|_{H^8} \lesssim \mu^{-\beta}\varepsilon
        \end{align*}
        for all $t\le \mu^{-\frac13}$.
\end{itemize}
    Then necessarily
    \begin{align*}
        \beta \ge \min\!\Big(\tfrac{5}{6}-\gamma,\tfrac{1}{6}\Big).
    \end{align*}
\end{prop}

We prove the stability part of Theorem~\ref{thm:main} by a bootstrap method using time time-dependent Fourier multiplier.  To define the bootstrap energies, we use the time-dependent Fourier multipliers
\begin{align*}
     A(t,k,\eta, l ) &= \begin{cases}
         e^{-c\mu^{\frac 13 } t  }M^{-1} (t,k,\eta,l)\langle k ,\eta,l\rangle^{N},& k\neq 0,  \\
         M^{-1} (t,0,\eta,l)\langle \eta,l \rangle^{N},& k= 0,
    \end{cases}\\
    A^g(t,k,\eta ) &=A(k,\eta,0 )\left(1+\vert k \vert +\tfrac {\vert k,\eta \vert}{\langle \mu^{\frac 13 }t \rangle }\textbf{1}_{k\neq 0}\right),\\
    M(t,k,\eta, l ) &= (m  M_\mu)(t,k,\eta, l ).
\end{align*}
The main weight is denoted as
\begin{align*}
   m (t,k,\eta,l)  &=\exp\left( c^{-1}(1+\tfrac 1 {\vert \alpha\vert} )   \int_{-\infty}^t \sum_{\substack{i, j\\ j \neq 0}} \frac {\vert j \vert^2  }{\vert j,\eta-j\tau , i  \vert^2  }\langle j-k , i-l\rangle^{-4}\dd\tau \right),
\end{align*}
which is a modification of the weight $m_1$ used in \cite{knobel2025suppression}. The enhanced dissipation weight is denoted as 
\begin{align*}
   \frac {\p_t M_\mu}{ M_\mu}(t,k,\eta ,l )&=\frac {\mu^{\frac 13 } }{1+\mu^{\frac 23 }(t-\frac \eta k)^2 }, & k\neq 0, \\
   M_\mu(0,k,\eta ,l )&=M_\mu(t, 0,\eta ,l )=1,
\end{align*}
which is a standard weight \cite{liss2020sobolev}.

To establish stability, we use one further unknown: The \textbf{inviscid damping unknowns} are defined as 
\begin{align*}
    \tilde \rho_1&= \langle\p_x\rangle ^{-1} \Lambda_t \tilde v_{\sim }^y +\tfrac {\p_x} {\langle \p_x\rangle \alpha\p_z} \p_y^t \Lambda_t^{-1}\tilde b_{\sim }^y,  &
    \tilde \rho_2 &= \langle \p_x\rangle^{-1} \Lambda_t\tilde b_{\sim }^y,
\end{align*}
and they are used to obtain inviscid damping estimates on $v^y$ and $b^y$.

Let $T>0$ and $C_2$ , then we say the bootstrap hypothesis holds if
\begin{align}
    \label{boot1a}\Vert  A (\tilde v,\tilde b) \Vert_{L^\infty_T L^2}^2+\int_0^T\mu\Vert \nabla_\tau A (\tilde v,\tilde b) \Vert_{L^2}^2 + \Vert \sqrt{\tfrac {\p_t M }M }A (\tilde v,\tilde b)
    \Vert_{L^2}^2\dd\tau &\le C_2 \eps^2, \\
    \label{boot1f}\Vert  A f \Vert_{L^\infty_T L^2}^2 +\int_0^T\mu\Vert \nabla_\tau A f \Vert_{L^2}^2 + \Vert \sqrt{\tfrac {\p_tM }M }A f\Vert_{L^2}^2\dd\tau &\le C_2 (\mu^{-(1-\gamma)}\eps)^2  , \\
    \label{boot1g}\Vert  A^g g \Vert_{L^\infty_T L^2}^2 +\int_0^T\mu\Vert \nabla_\tau A^gg\Vert_{L^2}^2 + \Vert \sqrt{\tfrac {\p_t M }M }A^gg\Vert_{L^2}^2\dd\tau &\le C_2 \eps^2,\\
    \label{boot1h}\Vert  A h \Vert_{L^\infty_T L^2}^2 +\int_0^T\mu\Vert \nabla_\tau A h \Vert_{L^2}^2 + \Vert \sqrt{\tfrac {\p_tM }M }A h\Vert_{L^2}^2\dd\tau &\le C_2 \eps^2  , \\
    \label{boot1q}\Vert  (v^{y,z}_=,b^{y,z}_= ) \Vert_{L^\infty_T H^{N}}^2 +\int_0^T\mu \Vert \nabla (v^{y,z}_=,b^{y,z}_= )\Vert_{H^{N}}^2\dd\tau &\le C_2 \delta^2 \mu^2,\\
         \Vert A\tilde  \rho\Vert_{L^\infty_T L^2}^2 +\int_0^T\mu \Vert A \nabla_\tau \tilde \rho\Vert_{L^2}^2 +\Vert A \sqrt{\tfrac {\p_t M}M }\tilde  \rho\Vert_{L^2}^2 \dd\tau &\le C_2 \eps^2.\label{boot2}
\end{align}
We prove stability by the following bootstrap proposition:
\begin{prop}\label{prop:1}
   Under the assumption of Theorem~\ref{thm:main} and let the bootstrap assumptions \eqref{boot1a}-\eqref{boot2} hold for an interval $[0,T]$ then there exists $c,\delta>0$ such that the estimates \eqref{boot1a}-\eqref{boot2} holds with $<$. 
\end{prop}
The proof of this proposition is divided into the Propositions \ref{prop:main}, \ref{prop:fgest}, \ref{prop:qest}, and \ref{prop:Invest}, which we prove in the upcoming Sections \ref{sec:mainenergy}, \ref{sec:zaver}, \ref{sec:xaver}, and \ref{sec:invis}. The lower bounds are achived by the following proposition:

\begin{prop}\label{prop:nlgrowth}  There exists $\mu_0>0$ such that if  the stability statments of Theorem~\ref{thm:main} hold, then for all $0<\mu\le \mu_0$ and $5/6\le \gamma $ there exist $z$-average free initial data $(v_{in},b_{in})$ satisfying $f_{in}=0$ and the assumption of Theorem~\ref{thm:main} such that $f$ exhibits nonlinear growth 
\begin{align*}
     \Vert f(T)\Vert_{L^2}   &\gtrsim \delta^2  \mu^{2\gamma-1 }
\end{align*}
for $T=\frac 1 {10} \mu^{-\frac 13 }$. Furthermore, we have the lower bound 
\begin{align*}
    \Vert  g\Vert_{L^\infty_T H^1}   &\gtrsim \delta \mu^\gamma.
\end{align*}
\end{prop}

The proof of Proposition \ref{prop:nlgrowth} is in Section \ref{sec:nlgrow}. By using these two propositions, we prove Theorem~\ref{thm:main}. 

\begin{proof}[Proof of Theorem~\ref{thm:main}]
We first prove the stability part. For the sake of contradiction, we assume that there exists a maximal time $T^\ast$ such that the bootstrap assumption \eqref{boot1a}-\eqref{boot2} holds. Then by \eqref{prop:1} the estimates improve and hold with $<$. By local existence, this contradicts the maximality of $T^\ast$. Therefore, stability follows, and the inviscid damping and enhanced dissipation estimates are a consequence of the definition of $A$ and that $v$ and $b$ are divergence-free.  The linear and nonlinear growth is a consequence of Proposition \ref{prop:nlgrowth}. 
\end{proof}
 For Theorem~\ref{thm:main}, it is only left to prove Proposition \ref{prop:nlgrowth}-\ref{prop:Invest}. To prove these propositions, we will use the following lemma using inviscid damping and enhanced dissipation: 
\begin{lemma}\label{lem:Inviscid}
    Let \eqref{boot1a}-\eqref{boot2} hold on $[0,T]$, then 
\begin{align*}
    \Vert  A(v_{\neq \neq }^{y},b_{\neq \neq }^{y}) \Vert_{L^2_T L^2}&\lesssim \eps,\\
    \Vert  A(v_{\neq \neq },b_{\neq \neq },h_{\neq},  \rho_{\neq} ) \Vert_{L^2_T L^2}&\lesssim\mu^{-\frac 1 6} \eps,\\
    \Vert  Af_{\neq} \Vert_{L^2_T L^2}&\lesssim \mu^{-\frac 1 6-(1-\gamma)} \eps,\\
    \Vert  A^g g_{\neq} \Vert_{L^2_T L^2}&\lesssim \mu^{-\frac 1 6} \eps.
\end{align*}
\end{lemma}
\begin{proof}
The first estimate is a consequence of $$\Vert A(v_{\neq \neq }^{y},b_{\neq \neq }^{y})\Vert_{L^2_T L^2} \approx \Vert \p_x\Lambda^{-1}_t  A \tilde \rho_{\neq \neq}\Vert_{L^2_T L^2}\lesssim \Vert\sqrt{\tfrac {\p_t m}m } A \tilde \rho_{\neq \neq}\Vert_{L^2_T L^2}.$$
The other estimates are is a standard property of $M_\mu$, i.e. using Lemma \ref{lem:Mmu} we obtain $$\mu^{\frac 1 6 } \Vert  A v_{\neq \neq }\Vert_{L^2}\lesssim \Vert \sqrt{\tfrac {\p_t M_\mu}{M_\mu} } A v_{\neq \neq }\Vert_{L^2_TL^2} +\mu^{\frac 12 } \Vert  A v_{\neq \neq }\Vert_{L^2_TL^2},$$
and similar to the other unknowns. Using \eqref{boot1a}-\eqref{boot2} we obtain the lemma. 
\end{proof}

\subsection{Notation for nonlinear interaction}
To establish energy estimates, the nonlinear terms in the equation yield trilinear terms. Since in the linear dynamics, averages or non-averages behave differently, the nonlinear interaction changes accordingly, and we distinguish between several cases, namely $\sim, \diamond, f, g,h $ and $=$:

For nonlinear terms with three different terms, we write 
\begin{align*}
    NL_{\ast_1 \ast_2 \to \ast_3 }, \quad \quad \ast_i \in \{\sim, \diamond, f, g,h,= \},
\end{align*}
to denote that $\ast_1$ and $ \ast_2$ act on $\ast_3$.

For nonlinear terms with two different terms, we write 
\begin{align*}
    NL_{\ast_1 \to \ast_2 }, \qquad\quad  \ast_i \in \{\sim, \diamond, f, g,h,= \},
\end{align*}
to denote that $\ast_1$ acts on $\ast_2$.

Nonlinear terms where one type of unknowns is acting on itself are denoted as 
\begin{align*}
    NL_{\ast }, \qquad\qquad \quad  \ast \in \{\sim, \diamond, f, g,h,=\}. 
\end{align*}

Each case refers to either a subscript or an unknown. In particular, we distinguish between the subscripts $\diamond$ and  $\sim$. Note that the subscript $\diamond$ includes the modes $\neq =$ and the subscript $\sim$ includes the modes $\neq \neq$ and $\cancel\diamond$. Furthermore, the $\diamond$ modes split into the three unknowns $f$, $g$ and $h$. For the $\sim$ modes, the unknowns $\tilde v$, $\tilde b$, $v_\sim $ and $b_\sim $ are treated in terms of norm estimates in the same way, and we often write $a_\sim$ or $\tilde a$. We write $=$ for terms consisting of the $x$ average or double average remains.% \red{Is this useful or confusing?}

This notation is used, for example, in Lemmata \ref{lem:Esim} and \ref{lem:Efg} and their proofs.

\section{Non-Average Energy Estimates}\label{sec:mainenergy}

In this section, we do the main energy estimates and improve the bootstrap assumption \eqref{boot1a}. In particular, we prove the following proposition:
\begin{prop}\label{prop:main}
    Under the assumption of Theorem~\ref{thm:main} and let the bootstrap assumptions \eqref{boot1a}-\eqref{boot2} hold on the interval $[0,T]$. Then there exists a constant $C_{3,\sim}>0$ such that 
    \begin{align*}
    \Vert  A (\tilde v,\tilde b) \Vert_{L^\infty_TL^2}^2+\int_0^T\mu\Vert \nabla_\tau A (\tilde v,\tilde b) \Vert_{L^2}^2 + \Vert \sqrt{\tfrac {\p_t M }M }A (\tilde v,\tilde b)
    \Vert_{L^2}^2\dd\tau &\le (C_1+4cC_2+\delta C_{3,\sim})\eps^2.
\end{align*}
\end{prop}
To prove this proposition, we first reorganize the terms: 
\begin{lemma}\label{lem:Esim}
 Under the assumption of Theorem~\ref{thm:main} and let the bootstrap assumptions \eqref{boot1a}-\eqref{boot2} hold, then we obtain the inequality
\begin{align*}
    \tfrac 1 2 \p_t \Vert A (\tilde v,\tilde b)\Vert_{L^2}^2 &+\mu\Vert \nabla_t A (\tilde v,\tilde b)\Vert_{L^2}^2 + \Vert \sqrt{\tfrac {\p_t m}m}A   (\tilde v,\tilde b)  \Vert_{L^2}^2\le L+  \sum_{a\in \Gamma} NL_{\sim}^a+NL_{\diamond \to \sim }^a +LNL,
\end{align*}
where we denote $\Gamma=\{a=(a^1,a^2,a^3)\in\{(v,v,v),(v,b,b),(b,v,b),(b,b,v)\}\}$ and:\\
\textbf{Linear terms:}
\begin{align*}
    L&=c\mu^{\frac 1 3 }\Vert A  (\tilde v_{\neq \neq },\tilde b_{\neq \neq })\Vert_{L^2}^2  + \tfrac 2 \alpha \vert \langle A \p_x \p_z^{-1} b^y_{\neq \neq } , A\p_x \Delta_t^{-1} v^y _{\neq \neq } \rangle\vert,
\end{align*}
\textbf{Nonlinear $\sim$ self interation terms:} 
\begin{align*}
    NL_{ \sim }^a&= \vert \langle A  \tilde a^1   , A((a_\sim^2 \cdot \nabla_t) a_\sim^3 )_{\sim } \rangle\vert, 
\end{align*}
\textbf{Nonlinear action of $\diamond$ onto $\sim$:} 
\begin{align*}
    NL_{\diamond\to \sim }^a &=  \vert \langle A  \tilde a^1   , A( ( a_{\neq=}^2\cdot \nabla_t) a_\sim^3)_\sim   \rangle\vert  +\vert \langle A  \tilde a^1   , A( (a_\sim ^2\cdot \nabla_t) a_{\neq=} ^3 )_{\sim } \rangle\vert \\
    &\quad +\vert \langle A   a_{==}^1   ,A( (a_{\neq=}^2\cdot  \nabla_t) a_{\neq=}^3  ) \rangle\vert,
\end{align*}
\textbf{Lower nonlinear terms:}
\begin{align*}
    LNL &=\alpha^{-1} \vert \langle A  \p_x \p_z^{-1}  b^y_\sim ,A \pi _{\sim}\rangle\vert  +\alpha^{-1} \vert \langle A  b^x_{\neq \neq} , \p_z^{-1} ( (b\cdot \nabla_t) v^y - (v\cdot\nabla_t) b^y )_{\sim} \rangle \vert.
\end{align*}
\end{lemma}
Lemma \ref{lem:Esim} is obtained by direct calculations and the identity
\begin{align*}
    ((a^2 \cdot \nabla_t) a^3 )_{\sim }&=((a^2_\sim \cdot \nabla_t) a^3_\sim )_{\sim }+((a^2_\sim \cdot \nabla_t) a^3_{\neq =} )_{\sim }+((a^2_{\neq =} \cdot \nabla_t) a^3_\sim )_{\sim }+((a^2_{\neq =} \cdot \nabla_t) a^3_{\neq =} )_{== }.
\end{align*} 
Proposition \ref{prop:main} is a consequence of the following Lemma: 
\begin{lemma}\label{lem:main}
    Under the assumption of Theorem~\ref{thm:main} and let the bootstrap assumptions \eqref{boot1a}-\eqref{boot2} hold on the interval $[0,T]$. Then it holds the estimate
    \begin{align*}
    \int_0^T L \ \dd\tau &\le 4 cC_2 \eps^2, \\
    \int_0^T  \sum_{a\in \Gamma} NL_{\sim}^a+NL_{\diamond \to \sim }^a+LNL \ \dd\tau & \lesssim \delta  \eps^2. 
\end{align*}
\end{lemma}

\subsection{Linear Estimates} For the first term we use Lemma \ref{lem:Mmu}
\begin{align*}
    c\mu^{\frac 1 3 }\Vert A  (\tilde v_{\neq \neq },\tilde b_{\neq \neq })\Vert_{L^2}^2&\le c\Vert A  \sqrt{\tfrac {\p_t M_\mu } { M_\mu }}(\tilde v_{\neq \neq },\tilde b_{\neq \neq })\Vert_{L^2}^2+c\mu \Vert A  \Lambda_t (\tilde v_{\neq \neq },\tilde b_{\neq \neq })\Vert_{L^2}^2
\end{align*}
and for the second term Lemma \ref{lem:m}
\begin{align*}
     \tfrac 2 \alpha\vert  \langle A \p_x \p_z^{-1} b^y_{\neq \neq } , A\p_x \Delta_t^{-1} v^y _{\neq \neq } \rangle\vert &\le c \Vert A  \sqrt{\tfrac {\p_t m } {m }}(\tilde v_{\neq \neq },\tilde b_{\neq \neq })\Vert_{L^2}^2. 
\end{align*}
Thus, integrating in time yields 
\begin{align*}
\int_0^T L \dd\tau \le 4 c C_2 \eps^2.
\end{align*}

\subsection{Bound on $NL_{\sim}$,} since $v_\sim$ and $b_\sim$ satisfy the same bootstrap assumption, we do not distinguish between $a^i_\sim $ and therefore for the sake of simplicity we omit writing the $i$.   First, we use Plancherel's identity to obtain 
\begin{align*}
    NL_{\sim }^a&=\vert  \langle A  \tilde a  , A( ( a_\sim \cdot \nabla_t) a_\sim ) \rangle\vert \\
    &\le \sum_{k,\tilde k , l,\tilde l }\iint\dd (\eta, \xi)\textbf{1}_{\Omega_\sim}\vert A(k,\eta,l)\vert  \left \vert a(k-\tilde k ,\eta-\xi, l-\tilde l) \cdot \begin{pmatrix}
        \tilde k \\ \xi-kt \\ \tilde l 
    \end{pmatrix}\right\vert\vert Aa\vert (k,\eta ,l ) \vert a\vert(\tilde k ,\xi,\tilde l ) \\
    &=\sum_{k,\tilde k , l,\tilde l }\iint\dd (\eta, \xi) \textbf{1}_{\Omega_\sim}(\textbf{1}_{\Omega_{\neq} }+\textbf{1}_{\Omega_{=} })\calI
\end{align*}
where we denote the sets 
\begin{align*}
    \Omega_\sim &=\{(k,\tilde k, l ,\tilde l ): \ (l\neq 0 \text{ or } k=0), \ (\tilde l\neq 0 \text{ or }\tilde k=0), \ (l-\tilde l\neq 0 \text{ or } k- \tilde k=0)\},\\
    \Omega_{\neq}&= \{ (k,\tilde k, l ,\tilde l ): \ k \neq 0 \text{ or }  k\neq \tilde k\},\\
    \Omega_{=}&= \{ (k,\tilde k, l ,\tilde l ): \  k= \tilde k= 0 \}.
\end{align*}
By the choice of $\Omega_\sim$ we obtain 
\begin{align*}
    \sum_{k,\tilde k , l,\tilde l }\iint\dd (\eta, \xi)\textbf{1}_{\Omega_\sim } \textbf{1}_{\Omega_{\neq} }\calI&\lesssim \Vert Aa_{\neq \neq }\Vert_{L^2 }\Vert Aa_{\sim }\Vert_{L^2 }\Vert A\nabla_t a_{\sim  }\Vert_{L^2 }.
\end{align*}
For the second term, we use that $a^y(0 ,\eta-\xi, 0)=0$, and thus if 
\begin{align*}
    \left \vert a(0 ,\eta-\xi, l-\tilde l) \cdot \begin{pmatrix}
        0 \\ \xi \\ \tilde l 
    \end{pmatrix}\right\vert\neq 0
\end{align*}
 on  $\Omega_\sim \cap\Omega_=$ we obtain either $l\neq \tilde l$ or $l\neq 0 $. We estimate 
\begin{align*}
      \left \vert a(0 ,\eta-\xi, l-\tilde l) \cdot \begin{pmatrix}
        0 \\ \xi \\ \tilde l 
    \end{pmatrix}\right\vert\le   \left \vert a^{y,z}_=(\eta-\xi, l-\tilde l)\right\vert\vert \xi, \tilde l \vert \left(\vert \eta-\xi, l-\tilde l\vert +\vert \eta, l\vert \right).
\end{align*}
Therefore,
\begin{align*}
    \sum_{k,\tilde k , l,\tilde l }\textbf{1}_{\Omega_\sim }\textbf{1}_{\Omega_{=} }\calI&\lesssim \Vert A \nabla a_{=} \Vert_{L^2}^2\Vert A a^{y,z}_=\Vert_{L^2}+\Vert A \nabla a_{= } \Vert_{L^2}\Vert A  a_{=  } \Vert_{L^2}\Vert A \nabla a^{y,z}_=\Vert_{L^2}.
\end{align*}
Combining these estimates, integrating in time and using the bootstrap assumption yields 
\begin{align*}
    \int_0^T NL_\sim \ \dd\tau \lesssim \mu^{-\frac 23 } \eps^3 +\mu^{-1} (\delta \mu)\eps^2  \le \delta \eps^2. 
\end{align*}

\subsection{Bound on $NL_{\diamond\to \sim }$,} in this section we estimate the forcing of $f$, $g$ and $h$ onto the $\sim$ terms
\begin{align*}
    NL_{\diamond\to \sim }^a &=  \vert \langle A  \tilde a^1,  A((a_{\neq =}^2 \cdot  \nabla_t) a_\sim ^3)_\sim   \rangle\vert  +\vert \langle A  \tilde a^1   , A( (a_{\sim}^2\cdot \nabla_t) a_{\neq =}^3  )_{\sim } \rangle\vert \\
    &\quad +\vert \langle A  \tilde a_{==}^1   ,A( (a_{\neq =}^2 \cdot\nabla_t)a_{\neq =}^3  )_{\sim } \rangle\vert 
\end{align*}
where $(a^1,a^2,a^3)\in \Gamma $. Here for $a_{\neq=}^i$ we distinguish between 
\begin{align*}
     b_{\neq =}&= \begin{pmatrix}
        \p_y^t g \\ -\p_x g \\ h_2 
    \end{pmatrix},&
    v_{\neq =}&= \begin{pmatrix}
       \tilde \Lambda_t^{-2} \p_y^t f \\ -\tilde \Lambda_t^{-2}\p_x f \\ h_1 
    \end{pmatrix}.
\end{align*}
For the $a^i_\sim$ we do not distinguish between different $i$ since $v_\sim$ and $b_\sim$ satisfy the same energy estimates. 
Therefore, we split
\begin{align*}
    \sum_{(a^1,a^2,a^3)\in \Gamma} \vert NL_{\diamond\to \sim }^a\vert &\le  NL_{f\to \sim,1   } + NL_{f\to \sim,2   } \\
    &\quad + NL_{g\to \sim,1   } + NL_{g\to \sim,2   } \\
    &\quad + NL_{h\to \sim   }  \\
    &\quad +  NL_{\diamond\to =,1 }^a+ NL_{\diamond\to = ,2}^a  ,
\end{align*}
where we denote 
\begin{align*}
    NL_{f\to \sim,1  }&=\vert\langle A   \tilde a^{x,y} , A((a_{\sim  }\cdot \nabla_t) v_{\neq =}^{x,y})\rangle\vert,\\
     NL_{f\to \sim,2  }&=\vert\langle A \tilde a  , A((v_{\neq =}^{x,y}\cdot\tilde \nabla_t) a_{\sim })\rangle\vert,\\
     NL_{g\to \sim,1  }&= \vert\langle A  \tilde a^{x,y}  , A((a_{\sim  }\cdot \nabla_t) b_{\neq =}^{x,y})\rangle\vert,\\
     NL_{g\to \sim,2  }&=\vert\langle A \tilde a  , A((b_{\neq =}^{x,y}\cdot \tilde \nabla_t) a_{\sim })\rangle\vert,\\
     NL_{h\to \sim  }&= \vert\langle A  \tilde a^z  , A((a_{\sim  }\cdot\nabla_t) a_{\neq =}^z)\rangle +\langle A \tilde a , A( a_{\neq =}^{z}\p_z a_{\sim })\rangle\vert,\\
     NL_{\diamond\to = ,1}&=\sum_{(a,a^2,a^3)\in \Gamma}\vert\langle A  \tilde a_{==}^x  ,A( (a_{\neq =}^2\cdot\nabla_t) a_{\neq =}^{3,x}  )_{== } \rangle\vert,\\
    NL_{\diamond\to =,2  }&=\sum_{(a,a^2,a^3)\in \Gamma}\vert\langle A  \tilde a_{==}^z   ,A( (a_{\neq =}^2\cdot\nabla_t) a_{\neq =}^{3,z}  )_{== } \rangle\vert.
\end{align*}
In the following, we estimate all terms separately. The estimate on $NL_{g\to \sim,1}$ and $NL_{g\to \sim,2}$ are the most relevant for the threshold. 

\textbf{Bound on $NL_{f\to \sim,1  }$}, we estimate directly 
\begin{align*}
    NL_{f\to \sim,1  }&=\vert\langle A   \tilde a^{x,y} , A((a_{\sim  }\cdot\nabla_t) v_{\neq =}^{x,y})\rangle\vert=\vert\langle A  \tilde a^{x,y}  , A((a_{\sim  }\cdot \nabla_t)  (\Lambda_t^{-2} \tilde \nabla_t^\perp f_{\neq } )\rangle\vert\\
    &=\vert\langle A  \tilde a^{x,y}_{\neq \neq }  , A((a_{\sim  }\cdot \nabla_t)  (\Lambda_t^{-2} \tilde \nabla_t^\perp f_{\neq } )\rangle\vert+\vert\langle A  \tilde a^{x,y}_{= }  , A((a_{\neq \neq   }\cdot \nabla_t)  (\Lambda_t^{-2} \tilde \nabla_t^\perp f_{\neq } )\rangle\vert\\
    & \lesssim \Vert A a_{\neq \neq } \Vert_{L^2}\Vert Aa_{\sim }\Vert_{L^2}\Vert A f_{\neq}\Vert_{L^2 }.
\end{align*}
Integrating in time and using the bootstrap assumption yields 
\begin{align*}
    \int_0^T NL_{f\to \sim,1  }\dd\tau &\lesssim \mu^{-\frac 13 } \eps^2(\mu^{-(1-\gamma)} \eps)\lesssim  \delta  \eps^2. 
\end{align*}

\textbf{Bound on $NL_{f\to \sim,2  }$}, we estimate directly 
\begin{align*}
    NL_{f\to \sim,2  }&=\vert\langle A \tilde a  , A((v_{\neq =}^{x,y}\cdot \tilde \nabla_t) a_{\sim })\rangle\vert=\vert\langle A \tilde a  , A((\Lambda_t^{-2} \tilde \nabla_t^\perp f_{\neq}  \cdot \tilde \nabla_t) a_{\sim })\rangle\vert\\
    & \lesssim \Vert A a_{\sim  } \Vert_{L^2}\Vert A\nabla_t a_{\sim }\Vert_{L^2}\Vert A\Lambda_t^{-1} f_{\neq}\Vert_{L^2 }\\
    & \lesssim \Vert A a_{\sim  } \Vert_{L^2}\Vert A\nabla_t a_{\sim }\Vert_{L^2}\Vert A\sqrt{\tfrac{\p_tm}m } f_{\neq}\Vert_{L^2 }.
\end{align*}
Integrating in time and using the bootstrap assumption yields 
\begin{align*}
    \int_0^T  NL_{f\to \sim,2  }\dd\tau &\lesssim \mu^{-\frac 12  } \eps^2 (\mu^{-(1-\gamma)} \eps)\lesssim \delta \eps^2 . 
\end{align*}

\textbf{Bound on $NL_{g\to \sim,1}$}, we use $a_{\sim }^{x,y} =(a_{\sim }^x, a_{\neq \neq }^y +a_{= }^y)^T$ and $\p_z g=0$ to split 
\begin{align*}
    NL_{g\to \sim,1}&=\vert\langle A  \tilde a^{x,y}  , A((a_{\sim  }\cdot\nabla_t) b_{\neq =}^{x,y})\rangle\vert=\vert\langle A  \tilde a^{x,y} , A((a_{\sim  }^{x,y}\cdot \tilde \nabla_t) \tilde \nabla_t^\perp g_{\neq}  )\rangle\vert\\
    &=\vert\langle A \tilde a^{x,y} , A(a_{\sim }^x  \p_x \tilde \nabla_t^\perp g_{\neq} )\rangle \vert+\vert\langle A  \tilde a^{x,y} , A(a_{\neq  \neq  }^y  \p_y^t\tilde  \nabla_t^\perp g_{\neq} )\rangle\vert+\vert\langle A \tilde a_{\neq \neq }^{x,y} , A(a_{=  }^y  \p_y^t  \tilde \nabla_t^\perp g_{\neq} )\rangle\vert.
\end{align*}
We estimate these terms separately, the first we estimate by 
\begin{align*}
    \vert\langle A \tilde a , A(a_{\sim}^x  \p_x \nabla_t^\perp g_{\neq} )\rangle\vert&\lesssim \Vert A a_\sim \Vert_{L^2 }\Vert Aa_{\neq \neq }^x \Vert_{L^2 } \Vert  A \p_x \nabla_t g \Vert_{L^2  }\lesssim \Vert A a_\sim \Vert_{L^2 }\Vert Aa_{\neq \neq }^x \Vert_{L^2 } \Vert  A^g  \nabla_t g \Vert_{L^2  }.
\end{align*}
Using the definition of $A^g$ we estimate 
\begin{align*}
    \vert\langle A  \tilde a , A(a_{\neq  \neq  }^y  \p_y^t \tilde \nabla_t^\perp g_{\neq} )\rangle\vert&\lesssim  e^{-c \mu^{\frac 13 } t }\Vert  A a_\sim \Vert_{L^2 } \Vert A a_{\neq \neq }^y \Vert_{L^2 } \Vert  A  \nabla_t\p_y^t  g_{\neq} \Vert_{L^2  }\\
    &\lesssim  te^{-c \mu^{\frac 13 } t }\Vert A a_\sim \Vert_{L^2 } \Vert A a_{\neq \neq }^y \Vert_{L^2 } \Vert  A \nabla_t  g_{\neq} \Vert_{H^1  }\\
    &\lesssim  t\langle t \mu^{\frac 1 3}\rangle e^{-c \mu^{\frac 13 } t }\Vert A a_\sim \Vert_{L^2 } \Vert A a_{\neq \neq }^y \Vert_{L^2 } \Vert  A^g \nabla_t  g_{\neq} \Vert_{L^2  }.
\end{align*}
By partial integration, we infer 
\begin{align*}
    \vert\langle A \tilde a_{\neq \neq }^{x,y} , A(a_{=    }^y  \p_y^t  \tilde \nabla_t^\perp g_{\neq} )\rangle\vert&=\vert \langle A \p_y^t\tilde a_{\neq \neq }^{x,y} , A(a_{=    }^y    \tilde \nabla_t^\perp g_{\neq} )\rangle\vert+\langle A \tilde a_{\neq \neq }^{x,y} , A(\p_y a_{=    }^y    \tilde \nabla_t^\perp g_{\neq} )\rangle\vert \\
    &\lesssim (\Vert A \p_y^t a_\sim \Vert_{L^2 } \Vert A a_{=    }^y \Vert_{L^2 } +\Vert A  a_\sim \Vert_{L^2 } \Vert A\p_y a_{=    }^y \Vert_{L^2 })\Vert  A \nabla_t  g_{\neq} \Vert_{L^2 }.
\end{align*}
We integrate in time, use Lemma \ref{lem:Inviscid} and the bootstrap assumption to  estimate 
\begin{align*}
    \int_0^T NL_{g\to \sim,1} \dd\tau \lesssim \mu^{-\frac 56 }  \eps^3 + \mu^{-1} \eps^2 (\delta \mu) \lesssim \delta \eps^2  .
\end{align*}

\textbf{Bound on $NL_{g\to \sim,2  }$}, we use Plancherel's identity 
\begin{align*}
    NL_{g\to \sim,2  }&\le\vert \langle A \tilde a  , A((b_{\neq =}^{x,y}\cdot \tilde \nabla_t) a_{\sim })\rangle\vert =\vert \langle A \tilde a  , A((\tilde  \nabla^\perp_t g_{\neq}  \cdot\tilde  \nabla_t) a_{\sim })\rangle\vert \\
    &\le\sum_{\substack{k,\tilde k, l\\  k- \tilde k\neq 0 }}\iint\dd (\eta ,\xi )\vert \eta \tilde k -k\xi\vert A(k,\eta,l)  \vert A\tilde a \vert (k,\eta,l) \vert g \vert (k-\tilde k , \eta-\xi)\vert a_\sim \vert(\tilde k , \xi,l ).
\end{align*}
We split this into 
\begin{align*}
    NL_{g\to \sim,2  }
    &=\sum_{\substack{k,\tilde k, l\\ \tilde k, k- \tilde k\neq 0 }}\iint\dd (\eta ,\xi ) \textbf{1}_{\Omega_\sim} (\textbf{1}_{\Omega_{R, g\to\sim}}+\textbf{1}_{\Omega_{T, g\to\sim}}) \vert \eta \tilde k -k\xi\vert A(k,\eta,l)  \vert A\tilde a \vert (k,\eta,l) \\
    &\qquad \qquad \qquad \qquad \qquad\cdot \vert g \vert (k-\tilde k , \eta-\xi)\vert a_\sim \vert(\tilde k , \xi,l ) \\
    &\qquad + \sum_{\substack{k, l\\  k \neq 0 }}\iint\dd (\eta ,\xi ) \textbf{1}_{\Omega_\sim} \vert k\xi\vert A(k,\eta,l) \vert A\tilde a \vert (k,\eta,l) \vert g \vert (k , \eta-\xi)\vert a_\sim \vert(0 , \xi,l )\\
    &=R_{g\to \sim,2  }+T_{g\to \sim,2  }+E_{g\to \sim,2  }, 
\end{align*}
according to the sets $\Omega_{R, g\to\sim}=\{(k,\tilde k, \eta, \xi,l)\in \Z^2\times \R^2 \times \Z : \ \vert k-\tilde k, \eta-\xi\vert \ge \vert\tilde k ,\xi,l\vert \}$ and $\Omega_{T, g\to\sim}=\Omega_{R, g\to\sim}^c$. We estimate the reaction term by
\begin{align*}
    R_{g\to \sim,2  }&\lesssim e^{-c \mu^{\frac 13} t } \Vert A a_\sim  \Vert_{L^2} \Vert A a_{\neq \neq } \Vert_{L^2} \Vert A g_{\neq} \Vert_{H^1}\\
    &\lesssim e^{-c \mu^{\frac 13} t } \langle t \mu^{\frac 13 }\rangle \Vert A a_\sim  \Vert_{L^2} \Vert A a_{\neq \neq } \Vert_{L^2} \Vert A^g g_{\neq} \Vert_{L^2}.
\end{align*}
Integrating in time and using the bootstrap assumption yields
\begin{align*}
    \int_0^T R_{g\to \sim,2  } \dd\tau &\lesssim \mu^{-\frac 13 } \eps^3\le \delta \eps^2.
\end{align*}
For the transport, we use Lemma \ref{lem:m} to estimate 
\begin{align*}
    \vert \eta \tilde k-k\xi\vert&\le\vert \tilde k,\xi\vert  \vert k-\tilde k ,\eta-\xi\vert \le(\vert \tilde k,\xi-\tilde k t \vert +\vert \tilde k t \vert\textbf{1}_{\tilde k\neq 0}) \vert \vert k-\tilde k ,\eta-\xi\vert \\
    &\le \left(1+\tfrac t {\langle t-\frac \xi {\tilde k} \rangle }\right)   \vert \tilde k,\xi-\tilde kt \vert   \vert k-\tilde k ,\eta-\xi\vert \\
    &\le \left(1+t\sqrt{\tfrac {\p_t m } m (k,\eta) }\right)   \vert \tilde k,\xi-\tilde kt \vert   \vert k-\tilde k ,\eta-\xi\vert^4
\end{align*}
and so we estimate 
\begin{align*}
    T_{g\to \sim,2  } &\lesssim e^{-c\mu^{\frac 13 }t } \left(\Vert Aa_{\neq \neq }\Vert_{L^2}+t\Vert A\sqrt{\tfrac {\p_t m } m} a_{\neq \neq }\Vert_{L^2}\right) \Vert A\Lambda_t  a_{\sim }  \Vert_{L^2} \Vert A g_{\neq} \Vert_{L^2}. 
\end{align*}
Integrating in time and using the bootstrap assumption yields
\begin{align*}
    \int_0^T T_{g\to \sim,2  } \dd\tau  &\lesssim  \mu^{-\frac 56 } \eps^3\le \delta \eps^2  .
\end{align*}
We estimate 
\begin{align*}
    E_{g\to \sim,2  }&\lesssim  \Vert A a_{\neq \neq }\Vert_{L^2} \Vert A\p_y  a_{=}  \Vert_{L^2} \Vert A \p_x g_{\neq} \Vert_{L^2}\\
    &\lesssim \Vert A a_{\neq \neq }\Vert_{L^2} \Vert A\p_y  a_{=}  \Vert_{L^2} \Vert A^g g_{\neq} \Vert_{L^2}
\end{align*}
and thus after integrating in time 
\begin{align*}
    \int_0^T E_{g\to \sim,2  }\dd\tau &\lesssim \mu^{-\frac 23 } \eps^3\le \delta \eps^2.
\end{align*}

\textbf{Bound on $NL_{h\to \sim  }$}, we estimate directly
\begin{align*}
     NL_{h\to \sim  }&= \vert\langle A  \tilde a^z  , A((a_{\sim  }\cdot \nabla_t) a_{\neq =}^z)\rangle+\langle A \tilde a , A( a_{\neq =}^{z}\p_z a_{\sim })\rangle\vert\\
     &= \vert\langle A  \tilde a ^z , A((a_{\sim  }\cdot \nabla_t) h_{\neq} )\rangle +     \langle A \tilde a  , A(h_{\neq} \p_z  a_{\sim })\rangle\vert\\
     &=\vert\langle A  \tilde a ^z_{\neq \neq} , A((a_{\sim  }\cdot \nabla_t) h_{\neq} )\rangle+ \langle A  \tilde a^z_{= } , A((a_{\neq \neq}\cdot \nabla_t) h_{\neq} )\rangle +     \langle A \tilde a  , A(h_{\neq} \p_z  a_{\sim })\rangle\vert\\
     &\lesssim \Vert A  a_{\sim } \Vert_{L^2 }\left(\Vert  A  a_{\neq \neq} \Vert_{L^2 }\Vert A \nabla_t  h \Vert_{L^2 }+\Vert  A \nabla_t  a \Vert_{L^2 }\Vert A  h_{\neq} \Vert_{L^2 }\right).
\end{align*}
Integrating in time and using the bootstrap assumption yields
\begin{align*}
     \int_0^T NL_{h\to \sim  }\dd\tau &\lesssim \mu^{-\frac 23 } \eps^3\le \delta \eps^2 .
\end{align*}

\textbf{Bound on $NL_{\diamond \to = ,1}^a$}, we distinguish again between $v_{\neq =}$ and $b_{\neq =}$ and so we estimate the terms 
\begin{align*}
    \sum_{a\in \Gamma}  NL_{\diamond\to =,1 }^a &= \vert \langle  A v_{==}^x, A((\tilde \nabla^\perp_t \tilde \Lambda^{-2}_t f_{\neq}\cdot \tilde  \nabla_t) \tilde \Lambda^{-2}_t\p_y^t  f_{\neq})\rangle\vert  + \vert \langle  A v_{==}^x, A((\tilde \nabla^\perp g_{\neq} \cdot \tilde \nabla) \p_y^t g_{\neq})\rangle\vert  \\
    &\quad+\vert\langle  A b_{==}^x, A((\tilde \nabla^\perp \tilde \Lambda^{-2}_t f\cdot \tilde  \nabla)  \p_y^t  g_{\neq})\rangle\vert +\vert \langle  A b_{==}^x, A((\tilde \nabla^\perp g_{\neq} \cdot \tilde \nabla) \p_y^t \Lambda^{-2}_tf_{\neq})\rangle\vert\\
    &= NL_{ff\to =,1 }+NL_{gg\to =,1 }+NL_{fg\to =,1 }+NL_{gf\to =,1 }.
\end{align*}
We estimate directly
\begin{align*}
    NL_{ff\to =,1 }&\lesssim e^{-c \mu^{\frac 13 }t } \Vert Av_{==}\Vert_{L^2 }\Vert A f_{\neq}\Vert_{L^2 }^2,\\
     NL_{gg\to =,1 } &\lesssim e^{-c \mu^{\frac 13 }t } \Vert Av_{==}\Vert_{L^2 }\Vert A g_{\neq} \Vert_{H^1} \Vert A\nabla_t  g_{\neq} \Vert_{H^1}, \\
    NL_{gf\to =,1 } &\lesssim  e^{-c \mu^{\frac 13 }t }\Vert Ab_{==}\Vert_{L^2 } \Vert A  g_{\neq} \Vert_{H^1}\Vert \nabla \Lambda_t^{-1} Af\Vert_{L^2}.
    %\left(\Vert  A f_{\neq} \Vert_{L^2}+t \Vert  A\sqrt {\tfrac {\p_t m }m } f_{\neq} \Vert_{L^2}) \right).
\end{align*}
Therefore, using $\Vert A  g_{\neq} \Vert_{H^1}\le \langle t\mu^{\frac 13 }\rangle  \Vert A^g  g_{\neq} \Vert_{L^2}$ and $\Vert \nabla \Lambda_t^{-1} Af\Vert_{L^2}\lesssim \Vert  A f_{\neq} \Vert_{L^2}+t \Vert  A\sqrt {\tfrac {\p_t m }m } f_{\neq} \Vert_{L^2}) $, integrating in time and using the bootstrap assumption yields
\begin{align}
    \int_0^T NL_{ff\to =,1 }+NL_{gg\to =,1 }+NL_{gf\to =,1 } \dd\tau&\lesssim \mu^{-\frac 56 }\eps^3 \le \delta \eps^2. \label{eq:dia=1}
\end{align}
Now we turn to  $NL_{fg\to =,1}$, we use Plancherel's identity 
\begin{align*}
    NL_{fg\to =,1}
    &\lesssim\sum_{\tilde k,\tilde l\neq 0} \iint\dd (\eta,\xi) (\textbf{1}_{\Omega_{R,fg\to =,1}}+\textbf{1}_{\Omega_{T,fg\to =,1}})\frac{\vert \eta \tilde k \vert}{\vert \tilde k  , \eta -\xi +\tilde k t\vert^2 }\langle \eta \rangle^{N} \vert f\vert(-\tilde k ,\eta-\xi,-\tilde l) \\
    &\qquad \qquad\qquad \cdot \vert \p_y^t g \vert(\tilde k ,\xi,\tilde l) \vert A b^x\vert(0,\eta,0)\\
    &=R_{fg\to =,1} +T_{fg\to =,1}
\end{align*}
with the sets $\Omega_{R,fg\to =,1}=\{(\eta,\xi)\in \R^2: \ \vert \eta-\xi\vert \ge\vert  \xi \vert \}$ and $\Omega_{T,fg\to =,1}=\Omega_{fg\to =,1}^c$. For $R_{fg\to =,1}$ we use Lemma \ref{lem:m} to estimate
\begin{align*}
    \frac{\vert \eta \tilde k \vert}{\vert \tilde k  , \eta -\xi +\tilde k t\vert^2 }&\le \left(1+\frac{t}{\vert \tilde k  , \eta -\xi +\tilde k t\vert^2 }\right) \vert\tilde k ,\xi,\tilde l\vert^2 \\
    &\le \left(1+t \sqrt{\frac {\p_t m}{m}(0,\eta,0)}\sqrt{\frac {\p_t m}{m}(-\tilde k ,\eta-\xi,-\tilde l)}\right) \vert\tilde k ,\xi,\tilde l\vert^5. 
\end{align*}
Therefore, 
\begin{align*}
    R_{fg\to =,1}&\lesssim \left(\Vert A b_{==}\Vert_{L^2}  \Vert A f\Vert_{L^2}+t \Vert\sqrt{\tfrac {\p_t m}{m}} A b_{==}\Vert_{L^2}  \Vert\sqrt{\tfrac {\p_t m}{m}} A f\Vert_{L^2}\right) \Vert \p_y^t g_{\neq}\Vert_{H^7}\\
    &\lesssim e^{-c\mu^{\frac 13 } t } t\langle \mu^{\frac 13 } t \rangle \left(\Vert A b_{==}\Vert_{L^2}  \Vert A f\Vert_{L^2}+t \Vert\sqrt{\tfrac {\p_t m}{m}} A b_{==}\Vert_{L^2}  \Vert\sqrt{\tfrac {\p_t m}{m}} A f\Vert_{L^2}\right) \Vert A^g  g\Vert_{L^2 }. 
\end{align*}
Integrating in time and using the bootstrap assumption yields 
\begin{align}
    \int_0^T R_{fg\to =,1}\dd\tau &\lesssim \mu^{-\frac 23} \eps^2 (\mu^{-(1-\gamma)} \eps) \le \delta \eps^2. \label{eq:Rfg}
\end{align}
We estimate the transport term by using the definition of $A^g$
\begin{align*}
    T_{fg\to =,1} &\lesssim \Vert \p_y A  b^x_{==} \Vert_{L^2} \Vert \Lambda_t^{-2} f\Vert_{H^2}\Vert A \p_y^t g \Vert_{L^2}\lesssim \langle t\rangle ^{-1} \langle \mu^{\frac 1 3 } t \rangle e^{-c\mu^{\frac 13 } t } \Vert \p_y A  b^x_{==} \Vert_{L^2} \Vert A f\Vert_{L^2 }\Vert A^g g \Vert_{L^2}.
\end{align*}
Integrating in time and using the bootstrap assumption yields 
\begin{align}
    \int_0^T T_{fg\to =,1}\dd\tau &\lesssim \mu^{-\frac 12} \eps^2 (\mu^{-(1-\gamma)} \eps) \le \delta \eps^2. \label{eq:Tfg}
\end{align}
Combining the estimates \eqref{eq:dia=1}, \eqref{eq:Rfg} and \eqref{eq:Tfg} we obtain the estimate 
\begin{align}
    \int_0^T \sum_{a\in \Gamma}  NL_{\diamond\to =,1 }^a \dd\tau &\lesssim  \delta \eps^2. 
\end{align}

\textbf{Bound on $NL_{\diamond \to = ,2}^a$}, we have $a_{\neq =}^z=h_{i,\neq }$ for some $i=1,2$ and so 
\begin{align*}
    \sum_{a\in \Gamma} NL_{\diamond\to =,2  }&= \sum_{a\in \Gamma} \vert\langle A  \tilde a_{==}^z    ,A( (a_{\neq =}^2\cdot \nabla_t) a_{\neq =}^{3,z}  )_{\sim } \rangle\vert\\
    &\le \vert \langle A   \tilde a_{==}^z    ,A( (v_{\neq =}\cdot \nabla_t) h_i  )_{\sim } \rangle\vert + \vert  \langle A  \tilde a_{==}^z    ,A( (b_{\neq =}\cdot \nabla_t) h_i  )_{\sim } \rangle\vert.
\end{align*}
Since $\p_z h_i =0$ we have $(v_{\neq =} \cdot \nabla_t )h_i=- (\tilde \nabla^\perp_t \Lambda_t^{-2} f  \cdot \tilde \nabla_t )h_i $, which yields
\begin{align*}
    \vert\langle A  \tilde a_{==}^z , A( (v_{\neq =}\cdot \nabla_t) h_i  )_{\sim } \rangle\vert&\lesssim  \Vert A  \tilde a_{==}^z\Vert_{L^2 } \Vert A\Lambda_t^{-1} f_{\neq}\Vert_{L^2 }\Vert A\nabla_t h\Vert_{L^2 }\\
    &\lesssim   \Vert A  \tilde a_{==}^z\Vert_{L^2 } \Vert A\sqrt{\tfrac {\p_t m}m } f_{\neq}\Vert_{L^2 }\Vert A\nabla_t h\Vert_{L^2 }.
\end{align*}
Integrating in time and using the bootstrap assumption gives 
\begin{align*}
    \int_0^T \vert\langle A  \tilde a_{==}^z , A( (v_{\neq =}\cdot \nabla_t) h_i )_{\sim } \rangle\vert \dd\tau \le \mu^{-\frac 12 } \eps\delta \mu  (\mu^{-(1-\gamma)} \eps) \le \delta \eps^2. 
\end{align*}
For the second term by \eqref{meq=} we have $v_{==}^z=\tilde v_{==}^z $ and $b_{==}^z=\tilde b_{==}^z $, and so 
\begin{align*}
    \vert\langle A  \tilde a_{==}^z    ,A(( b_{\neq =}\cdot \nabla_t)h_i  )_{\sim } \rangle\vert&\lesssim \Vert A(v_{=}^{y,z},b_{=}^{y,z}) \Vert_{L^2}\Vert A\nabla_t g \Vert_{L^2}\Vert A\nabla_t h \Vert_{L^2}.
\end{align*}
Integrating in time and using the bootstrap assumption yields 
\begin{align*}
    \int_0^T \vert\langle A  \tilde a_{==}^z    ,A(( b_{\neq =}\cdot \nabla_t) h_i )_{\sim } \rangle\vert \dd\tau &\lesssim \mu^{-1} (\delta \mu)\eps^2 \le \delta \eps^2. 
\end{align*}
Combining these estimates, we obtain 
\begin{align*}
    \int_0^T NL_{\diamond\to =,2  }\dd\tau&\lesssim \delta \eps^2.
\end{align*}
Therefore, we conclude all the necessary estimates for Lemma \ref{lem:main}. By Lemma \ref{lem:Esim}, we conclude Proposition \ref{prop:main}.

\subsection{Bound on $LNL$} These estimates are simpler that the $NL$ estimates in this section and the nonlinear pressure estimates $NLP_\rho$ done in Section \ref{sec:invis}, we omit the details. %Note that and the nonlinear pressure for $\rho$, $NLP_\rho$ done in Section \ref{sec:invis}  for the sake of simplicity we omit the details.

\section{The $z$-Average Energy Estimates }\label{sec:zaver}
In this section, we estimate all the terms involving the $z$ average, i.e. we improve the bootstrap assumptions \eqref{boot1f}-\eqref{boot1h}. We do this by proving the following proposition:
\begin{prop}\label{prop:fgest}
    Under the assumption of Theorem~\ref{thm:main} and let the bootstrap assumption hold on the interval $[0,T]$. Then there exists a constant $C_{3,fgh}>0$ such that 
    \begin{align*}
    \Vert  A f \Vert_{L^\infty_T L^2}^2 +\int_0^T\mu\Vert \nabla_\tau A f \Vert_{L^2}^2 + \Vert \sqrt{\tfrac {\p_tM }M }A f\Vert_{L^2}^2\dd\tau &\le ( C_2+ \delta C_{3,fgh} ) (\mu^{-(1-\gamma)}\eps)^2  , \\
    \Vert  A^g g \Vert_{L^\infty_T L^2}^2 +\int_0^T\mu\Vert \nabla_\tau A^gg\Vert_{L^2}^2 + \Vert \sqrt{\tfrac {\p_t M }M }A^gg\Vert_{L^2}^2\dd\tau &\le ( C_2+ \delta C_{3,fgh} )\eps^2,\\
    \Vert  A h \Vert_{L^\infty_T L^2}^2 +\int_0^T\mu\Vert \nabla_\tau A h \Vert_{L^2}^2 + \Vert \sqrt{\tfrac {\p_tM }M }A h\Vert_{L^2}^2\dd\tau &\le( C_2+ \delta C_{3,fgh}) \eps^2 .
\end{align*}
\end{prop}
We recall the equations for the adapted $z$ average unknowns
\begin{align}\begin{split}
    \p_t f & =\mu    \Delta_t f+ \tilde \nabla_t^\perp \cdot ((b\cdot \nabla_t) b^{x,y} -(v\cdot \nabla_t) v^{x,y})_{\diamond},  \\
     \p_t g & =\mu \Delta_t g  - \tilde \Lambda_{t}^{-2}\tilde \nabla_t^\perp\cdot  ((b\cdot \nabla_t) v^{x,y}-(v\cdot \nabla_t) b^{x,y})_{\diamond}, \\
     \p_t h &= \mu \Delta_t h+  \begin{pmatrix}
        ((b\cdot \nabla_t)b^z-(v\cdot \nabla_t)v^z)_{\diamond}\\
        ((b\cdot \nabla_t)v^z -(v\cdot \nabla_t)b^z)_{\diamond}\label{eq:fgh2}
    \end{pmatrix}.
\end{split}\end{align}
Note that there is no linear interaction, but there are nonlinear interactions. It is crucial to estimate precisely the nonlinear interaction between the unknowns $f$ and $g$, see Subsection \ref{sec:heuz} for a heuristic. The interaction of $f$ and $g$ states
\begin{align*}
    \tilde \nabla_t^\perp \cdot ((b_{\diamond}\cdot \nabla_t )b^{x,y}_{\diamond} -(v_{\diamond}\cdot \nabla_t) v^{x,y}_{\diamond})_{\diamond}&=\tilde \nabla^\perp_t\cdot((\tilde \nabla^\perp g\cdot \tilde \nabla) \tilde \nabla_t^\perp  g)+(\nabla^\perp \Lambda_t^{-2}  f\cdot \tilde \nabla)   f,\\
    \tilde \Lambda_{t}^{-2}\tilde \nabla_t^\perp   \cdot ((b_{\diamond}\cdot \nabla_t) v^{x,y}_{\diamond}-(v_{\diamond}\cdot  \nabla_t) b^{x,y}_{\diamond})_{\diamond}&=(\tilde \nabla^\perp \tilde \Lambda_t^{-2} f\cdot \tilde \nabla g).
\end{align*}
Since $((a_{\ndiamond}\cdot \nabla_t) a_\diamond^{x,y})_\diamond=((a_\diamond\cdot \nabla_t) a_{\ndiamond}^{x,y})_\diamond=0$ we have 
\begin{align*}
    ((a\cdot \nabla_t) a^{x,y})_\diamond &=((a_\diamond\cdot \nabla_t) a_\diamond^{x,y})_\diamond +((a_{\ndiamond }\cdot \nabla_t) a_{\ndiamond}^{x,y})_\diamond.
\end{align*}
Therefore, for $f$ and $g$ we obtain the equations
\begin{align*}
    \p_t f   &=\mu \Delta_t f+  \tilde \nabla^\perp_t\cdot ((\tilde \nabla^\perp g\cdot \tilde \nabla) \tilde \nabla_t^\perp  g)+(\tilde \nabla^\perp \Lambda_t^{-2}  f\cdot \tilde \nabla)   f+ \tilde \nabla^\perp_t \cdot ((b_{\ndiamond }\cdot \nabla_t) b_{\ndiamond }^{x,y}-(v_{\ndiamond }\cdot \nabla_t) v_{\ndiamond }^{x,y})_\diamond,   \\
     \p_t g   &=\mu \Delta_t g+ (\tilde \nabla^\perp \Lambda_t^{-2} f\cdot \tilde \nabla) g +\tilde \Lambda^{-2}_t \tilde \nabla^\perp_t\cdot ((b_{\ndiamond }\cdot \nabla_t) v_{\ndiamond }^{x,y}-(v_{\ndiamond }\cdot \nabla_t) b_{\ndiamond }^{x,y})_\diamond.
\end{align*}
 The necessary estimates are summarized in the following Lemma:
\begin{lemma}[Energy in $f$, $g$ and  $h$]\label{lem:Efg}
Let the bootstrap assumptions \eqref{boot1a}-\eqref{boot2} hold, then we obtain the inequality
\begin{align*}
    \frac 1 2 \p_t \Vert A f\Vert_{L^2}^2  +\mu\Vert \nabla_t A  f \Vert_{L^2}^2 + \Vert \sqrt{\tfrac {\p_t M }M }A   f \Vert_{L^2}^2&\le L_f +NL_{f\to f }+NL_{g\to f } +NL_{\sim\to f},\\
    \frac 1 2 \p_t \Vert A^g g\Vert_{L^2}^2 +\mu\Vert \nabla_t A^g  g \Vert_{L^2}^2 + \Vert \sqrt{\tfrac {\p_t M }M }A^g g \Vert_{L^2}^2&\le  L_g+NL_{f\to g } +NL_{\sim\to g},\\
    \frac 12  \p_t \Vert A h \Vert_{L^2 }^2 +\mu \Vert A\nabla_t  h\Vert_{L^2 }^2 +\Vert\sqrt{\tfrac {\p_t M } M } A h\Vert_{L^2 }^2
    &\le L_h +  NL_{f\to h }+NL_{g\to h } + NL_{\sim\to h}.
\end{align*}
Where we denote: \\
\textbf{Linear terms:}
\begin{align*}
    L_f&= c \mu^{\frac 1 3 }\Vert A f_{\neq } \Vert_{L^2 }^2, &
    L_g&= c \mu^{\frac 1 3 }\Vert A^g g_{\neq } \Vert_{L^2 }^2,&
    L_h&=c\mu^{\frac 13}\Vert A h_{\neq}\Vert_{L^2}^2. 
\end{align*}
\textbf{Nonlinear $fg$ interactions:}
\begin{align*}
    NL_{f\to f}&= \vert \langle A f, A((\tilde \nabla^\perp \Lambda_t^{-2}f\cdot \tilde \nabla)  f)\rangle  \vert , \\
    NL_{g\to f}&= \vert\langle A f, A\tilde \nabla^\perp_t\cdot ((\tilde \nabla^\perp g\cdot \tilde \nabla) \tilde \nabla_t^{\perp}   g)\rangle\vert, \\
    NL_{f\to g}&=\vert\langle A^g g,A^g ((\tilde \nabla^\perp \Lambda_t^{-2} f\cdot \tilde \nabla) g)\rangle\vert.
\end{align*}
\textbf{Nonlinear $\sim$ forcing:}
\begin{align*}
    NL_{\sim\to f}&= \sum_{a\in\{v,b\}}\vert \langle A f ,  A\tilde \nabla^\perp_t\cdot ((a_{\ndiamond}\cdot  \nabla_t) a^{x,y}_{\ndiamond} )_\diamond \rangle\vert  ,\\
    NL_{\sim\to g}&=\sum_{a^1,a^2\in \{v,b\} } \vert \langle A^gg ,A^g \Lambda_t^{-2} \tilde \nabla^\perp_t\cdot ((a_{\ndiamond}^1\cdot \nabla_t) a^{2, x,y}_{\ndiamond} )_\diamond \rangle\vert .
\end{align*}
\textbf{Nonlinear forcings onto $h$:}
\begin{align*}
    NL_{f\to h }&=\vert \langle Ah_i, A((\tilde \nabla^\perp_t \Lambda_t^{-2} f\cdot  \tilde \nabla_t) h_j) \rangle\vert, \\
    NL_{g\to h }&=\vert \langle Ah_i, A((\tilde \nabla^\perp g \cdot \tilde \nabla) h_j) \rangle\vert, \\
    NL_{\sim\to h }&=\sum_{a^1,a^2\in \{v,b\} }\vert \langle Ah_i, ((a_{\ndiamond }^1\cdot \nabla_t)a^{2,z}_{\ndiamond })_\diamond \rangle\vert .
\end{align*}
    
\end{lemma}

\begin{lemma}\label{lem:fgest}
    Under the assumption of Theorem~\ref{thm:main} and let the bootstrap assumption on the interval $[0,T]$. Then it holds the estimate 
    \begin{align*}
    \int_0^T L_f +NL_{f\to f }+NL_{g\to f } +NL_{\sim\to f}\dd\tau&\le (c +\delta C) (\mu^{-(1-\gamma)}\eps) ^2, \\
   \int_0^T  L_g+NL_{f\to g } +NL_{\sim\to g}\dd\tau &\le ( c +\delta C) \eps^2,\\
   \int_0^T  L_h +  NL_{f\to h }+NL_{g\to h } + NL_{\sim\to h} \dd\tau &\le ( c +\delta C) \eps^2.
\end{align*}
\end{lemma}

Proposition \ref{prop:fgest} is a consequence of Lemma \ref{lem:Efg} and Lemma \ref{lem:fgest}. The linear estimates of Lemma \ref{lem:fgest} are done precisely as the linear term of the last section, we omit the details. In the following, we prove the nonlinear estimates of Lemma \ref{lem:fgest}. In Subsection \ref{sec:fgself} we estimate the \textbf{nonlinear $fg$ self interactions},  in Subsection \ref{sec:simtofg} we estimate the \textbf{nonlinear $\sim$ forcing}, and in Subsection \ref{sec:ontoh} we estimate the \textbf{nonlinear forcing onto $h$}.

\subsection{Nonlinear $fg$ Interaction }\label{sec:fgself} In the following, we bound the different nonlinear terms which arise due to the interaction between $f$ and $g$. 

\textbf{Bound on $NL_{f\to f}$:} By \eqref{meq=} we have $\nabla^\perp \Lambda^{-2}f_= =  - v_{==}^x e_1 $ and so 
\begin{align*}
     NL_{f\to f}&=\vert\langle A f, A((\tilde \nabla^\perp \tilde \Lambda_t^{-2}f\cdot \tilde \nabla)  f)\rangle\vert\\
     &= \vert\langle A f, A((\tilde \nabla^\perp_t \tilde \Lambda_t^{-2}f_{\neq} \cdot \tilde \nabla_t)  f)\rangle\vert +\vert \langle A f_{\neq}, A(v_{==}^x \p_x   f_{\neq})\rangle\vert\\
     &\le \left(\Vert A f\Vert_{L^2} \Vert A\sqrt{\tfrac{\p_t m} m } f\Vert_{L^2}+\Vert A f_{\neq} \Vert_{L^2} \Vert Av_{==}\Vert_{L^2} \right)\Vert \nabla_t f\Vert_{L^2} .
\end{align*}
Therefore, after integrating in time 
\begin{align*}
    \int_0^T NL_{f\to f}\dd\tau &\lesssim \mu^{-\frac 23 }\eps ( \mu^{-(1-\gamma)} \eps)^2\le \delta ( \mu^{-(1-\gamma)} \eps)^2.
\end{align*}

\textbf{Bound on $NL_{g\to f} $:} We split 
\begin{align*}
    (\tilde \nabla^\perp g\cdot \tilde \nabla) \tilde \nabla_t^\perp  g&=(\tilde \nabla^\perp g_{\neq}\cdot \tilde \nabla) \tilde \nabla_t^\perp  g_{\neq}+(\tilde \nabla^\perp g_{=}\cdot \tilde \nabla) \tilde \nabla_t^\perp  g_{\neq}+(\tilde \nabla^\perp g_{\neq}\cdot\tilde \nabla) \tilde \nabla_t^\perp  g_=\\
    &=(\tilde \nabla^\perp g_{\neq}\cdot \tilde \nabla) \tilde \nabla_t^\perp  g_{\neq}-\p_y  g_{=}\p_x  \tilde \nabla_t^\perp  g_{\neq}-\p_xg_{\neq}\p_y^2 g_= e_1. 
\end{align*}
Thus, we obtain the terms
\begin{align}\begin{split}
    NL_{g\to f}&= \vert\langle Af, A\tilde \nabla^\perp_t((\tilde \nabla^\perp g_{\neq}\cdot \tilde \nabla) \tilde \nabla_t^\perp  g_{\neq})\rangle\vert + \vert\langle Af, A\tilde \nabla^\perp_t(\p_y  g_{=}\p_x  \tilde \nabla_t^\perp  g_{\neq})\rangle\vert\\
    &\quad + \vert\langle Af, A\p_y^t(\p_xg_{\neq}\p_y^2 g_=)\rangle\vert,\label{NLgtof}
\end{split}\end{align}
which we bound separately. We estimate the first term by
\begin{align*}
    \vert\langle Af, A\tilde \nabla^\perp_t\cdot ((\tilde \nabla^\perp g_{\neq}\cdot \tilde \nabla) \tilde \nabla_t^\perp  g_{\neq})\rangle \vert&\lesssim e^{-c\mu^{\frac 13 }t }\Vert A\nabla_t f \Vert_{L^2} \Vert Ag \Vert_{H^1}\Vert A\nabla_t g \Vert_{H^1} \\
    &\lesssim \langle t \mu^{\frac 13 }\rangle^2 e^{-c\mu^{\frac 13 }t }\Vert A\nabla_t f \Vert_{L^2} \Vert A^gg \Vert_{L^2}\Vert A^g\nabla_t g \Vert_{L^2}. 
\end{align*}
After integrating in time we have 
\begin{align}
    \int_0^T \vert\langle Af, A\nabla^\perp_t\cdot((\nabla^\perp g_{\neq}\nabla) \nabla_t^\perp  g_{\neq})\rangle\vert \dd\tau&\lesssim \mu^{-1}(\mu^{-(1-\gamma)}\eps)\eps^2\le \delta (\mu^{-(1-\gamma)}\eps)^2 . \label{NLgtof1}
\end{align}
By \eqref{meq=} we have $\p_y g_= = - b_{==}^x  $, which we use to estimate 
\begin{align*}
     \vert\langle Af, A\nabla^\perp_t\cdot(\p_y  g_{=}\p_x  \nabla_t^\perp  g_{\neq})\rangle\vert &\lesssim \Vert A \nabla^\perp_t f\Vert_{L^2 } \Vert b_{==}^x  \Vert_{H^N}  \Vert A \p_x\nabla^\perp_t g\Vert_{L^2 }\\
     &\lesssim\Vert A \nabla^\perp_t f\Vert_{L^2 } \Vert b_{==}^x \Vert_{H^N}  \Vert A^g\nabla^\perp_t g\Vert_{L^2 }.
\end{align*}
Integrating in time yields
\begin{align}
     \int_0^T \vert\langle Af, A\nabla^\perp_\tau\cdot (\p_y  g_{=}\p_x  \nabla_t^\perp  g_{\neq})\rangle\vert \dd\tau &\lesssim \mu^{-1} (\mu^{-(1-\gamma)}\eps)\eps^2 \le \delta (\mu^{-(1-\gamma)}\eps)^2 .\label{NLgtof2}
\end{align}
Again we use $\p_y g_= = - b_{==}^x  $  to estimate 
\begin{align*}
    \vert\langle Af, A\p_y^t(\p_xg_{\neq}\p_y^2 g_=)\rangle\vert&\lesssim \Vert \p_y^t Af \Vert_{L^2} \Vert A \p_x g_{\neq} \Vert_{L^2} \Vert \p_y b_{==}^x  \Vert_{H^N}\\
    &\lesssim \Vert \p_y^t Af \Vert_{L^2} \Vert A^g g_{\neq} \Vert_{L^2} \Vert \p_y b_{==}^x \Vert_{H^N}.
\end{align*}
Integrating in time yields 
\begin{align}
    \int_0^T \vert\langle Af, A\p_y^t(\p_xg_{\neq}\p_y^2 g_=)\rangle\vert \dd\tau &\lesssim \mu^{-1} (\mu^{-(1-\gamma)}\eps)\eps^2 \le \delta (\mu^{-(1-\gamma)}\eps)^2. \label{NLgtof3}
\end{align}
Thus by \eqref{NLgtof}, \eqref{NLgtof1}, \eqref{NLgtof2} and \eqref{NLgtof3} we estimate 
\begin{align*}
    \int_0^T NL_{g\to f}\dd\tau &\lesssim \delta (\mu^{-(1-\gamma)}\eps)^2. 
\end{align*}

\textbf{Bound on $NL_{f\to g}$: } We use Plancherel's identity 
\begin{align*}
     NL_{f\to g}&=\vert\langle A^g g,A^g  ((\tilde \nabla^\perp \Lambda_t^{-2} f\cdot \tilde \nabla) g)\rangle\vert \\
     &= \sum_{k,\tilde k }\iint\dd (\eta,\xi)  \tfrac {\vert \eta \tilde k-k\xi\vert\vert A^g(k,\eta)\vert }{\vert k-\tilde k, \eta-\xi-(k-\tilde k)t \vert^2 }\vert A^g g\vert(k,\eta)\vert f\vert(k-\tilde k,\eta-\xi) \vert g\vert(\tilde k,\xi) \\
     &\qquad \qquad \qquad \cdot (\textbf{1}_{\Omega_R } + \textbf{1}_{\Omega_T }+ \textbf{1}_{\Omega_{=,1} }+ \textbf{1}_{\Omega_{=,2} })\\
     &= R_{f\to g}+T_{f\to g}+ NL_{f\to g,=,1}+ NL_{f\to g, =,2}, 
\end{align*}
where we split according to 
\begin{align*}
    \Omega_R&= \{((k,\eta),(\tilde k,\xi))\in (\Z\times \R)^2: \ \vert k-\tilde k , \eta -\xi\vert  \ge  \vert \tilde k , \xi\vert, \ k \neq \tilde k  , \ \tilde k \neq 0   \},\\
    \Omega_T&= \{((k,\eta),(\tilde k,\xi))\in (\Z\times \R)^2: \ \vert k-\tilde k , \eta -\xi\vert  \le \vert \tilde k ,   \xi\vert, \ k \neq  \tilde k , \ \tilde k  \neq 0   \},\\
    \Omega_{=1}&= \{((k,\eta),(\tilde k,\xi))\in (\Z\times \R)^2: \  k =  \tilde k    \},\\
    \Omega_{=2}&= \{((k,\eta),(\tilde k,\xi))\in (\Z\times \R)^2: \  \tilde k =0   \}.
\end{align*}
\underline{Reaction term,} on $\Omega_R$, we  use the bound 
\begin{align*}
     A^g(k,\eta)\lesssim  \vert k-\tilde k ,\eta-\xi \vert A(k-\tilde k ,\eta-\xi) \vert \tilde k ,\xi\vert
\end{align*}
to estimate  
\begin{align*}
    \tfrac {\vert \eta l-k\xi\vert\vert A^g(k,\eta)\vert }{\vert k-\tilde k, \eta-\xi-(k-\tilde k)t \vert^2 }&\lesssim\tfrac {\vert k-\tilde k  , \eta -\xi\vert^2 } {\vert k-\tilde k\vert^2 } \langle t-\tfrac {\eta-\xi}{k-\tilde k }\rangle^{-2}A(k-\tilde k ,\eta-\xi)  \vert \tilde k,\xi\vert^2  \\
    &\lesssim \left(1 + \tfrac {t^2}{\langle t- \frac{ \eta-\xi}{k-\tilde k } \rangle ^2 }\right) A(k-\tilde k ,\eta-\xi)\vert \tilde k ,\xi\vert^2\\
    &\lesssim \left(1+ t^2 \sqrt{\tfrac {\p_t m}{m}(k,\eta)} \sqrt{\tfrac {\p_t m}{m}(k-\tilde k,\eta-\xi )}\right)A(k-\tilde k ,\eta-\xi)\vert \tilde k ,\xi\vert^5.
\end{align*}
Therefore, using that $\tilde k,k-\tilde k \neq 0 $ we obtain 
\begin{align*}
    R_{f\to g}&\lesssim e^{-c\mu^{\frac 13 }t}\Vert A^g  g \Vert_{L^2}^2 \Vert A f \Vert_{L^2} + e^{-c\mu^{\frac 13 }t} t^2 \Vert A^g \sqrt{\tfrac {\p_t m} m}   g \Vert_{L^2} \Vert A \sqrt{\tfrac {\p_t m } m} f \Vert_{L^2}\Vert A^gg \Vert_{L^2}.
\end{align*}
Integrating in time yields 
\begin{align}
    \int_0^T  R_{f\to g} \dd\tau &\lesssim \mu^{-\frac 23 }\eps^2 (\mu^{-(1-\gamma)} \eps )\le \delta \eps^2 . \label{eq:fgR}
\end{align}

\underline{Transport term,} we use the estimate 
\begin{align*}
     A^g(k,\eta)\le  A^g(\tilde k,\xi) \vert k-\tilde k,\eta-\xi\vert 
\end{align*} 
to infer 
\begin{align*}
    T_{f\to g} &\lesssim \sum_{k,\tilde k,k-\tilde k\neq 0 } \iint\dd (\eta ,\xi)  \ \tfrac {\vert \tilde k,\xi\vert }{\langle t\rangle^2}    \vert A^g g\vert(k,\eta)\vert \Lambda^5 f\vert(k-\tilde k,\eta-\xi) \vert A^g g\vert(\tilde k,\xi).
\end{align*}
By
\begin{align*}
    \vert \tilde k ,\xi \vert &= \vert \tilde k,\xi-\tilde kt \vert +\tilde kt \le t  \vert \tilde k,\xi-\tilde k t \vert,
\end{align*}
we infer 
\begin{align*}
    \tfrac {\vert \tilde k ,\xi \vert}{\langle t\rangle^2}  &\lesssim  \tfrac {\vert \tilde k,\xi-\tilde kt \vert}{\langle t\rangle}.
\end{align*}
Therefore, we estimate 
\begin{align*}
    \vert T_{f\to g} \vert\lesssim \langle t\rangle^{-1} \Vert A^g \Lambda_t g\Vert_{L^2}\Vert Af\Vert_{L^2}\Vert A^g  g\Vert_{L^2}.
\end{align*}
Integrating in time yields 
\begin{align}
    \int_0^T T_{f\to g}\dd\tau \lesssim \mu^{-\frac 1 2 } \eps^2 (\mu^{-(1-\gamma)} \eps)  \le \delta \eps^2 .\label{eq:fgT}
\end{align}

\underline{First Average term,} we estimate the term 
\begin{align*}
    NL_{f\to g,=,1}&= \sum_{k \neq 0 } \iint\dd (\eta,\xi) \tfrac { \vert k \vert A^g(k,\eta) }{\vert \eta-\xi \vert}\vert A^g g\vert(k,\eta)\vert f\vert(0,\eta-\xi) \vert g\vert(k,\xi). 
\end{align*}
For  $k \neq 0$ we obtain
\begin{align*}
    \vert A^g(k,\eta)\vert\le  \vert \eta -\xi\vert A(0,\eta- \xi)  +   A^g(k,\xi).
\end{align*}
Since $f_== \p_yv_{==}^x= \p_y\tilde v_{==}^x $ we infer $\tfrac 1 {\vert \eta-\xi \vert} \vert f\vert(0,\eta-\xi)= \vert  \tilde v^x\vert(0,\eta-\xi,0) $ and thus 
\begin{align*}
    \tfrac { \vert k \vert A^g(k,\eta)\vert }{\vert \eta-\xi \vert}\vert f\vert(0,\eta-\xi)\le \vert k \vert \vert A f \vert(0,\eta-\xi)+ \vert \tilde  v^x\vert(0,\eta-\xi,0) \vert k \vert A^g(k,\xi).
\end{align*}
Therefore, we estimate 
\begin{align*}
    NL_{f\to g,=,1}&\lesssim \Vert Af \Vert_{L^2 } \Vert A^g g_{\neq}  \Vert_{L^2}^2+ \Vert A\tilde v^x_{==} \Vert_{L^2 } \Vert A^g g_{\neq}  \Vert_{L^2} \Vert A^g \p_x g_{\neq}  \Vert_{L^2}.
\end{align*}
Integrating in time yields 
\begin{align}
    \int_0^T NL_=\dd\tau &\lesssim  \mu^{-\frac 13 } (\mu^{-(1-\gamma)} \eps)\eps^2 + \mu^{-\frac 23 }  \eps^3\lesssim  \delta \eps^2 \label{eq:fg=1}.
\end{align}

\underline{Second Average term,} we estimate 
\begin{align*}
    NL_{f\to g,=,2}=\sum_{k\neq 0  } \iint\dd (\eta,\xi) \tfrac {\vert k\xi\vert\vert A^g(k,\eta)\vert }{\vert k, \eta-\xi-kt \vert^2 }\vert A^g g\vert(k,\eta)\vert f\vert(k,\eta-\xi) \vert g\vert(0,\xi).
\end{align*}
By the definition of $A^g$ we obtain 
\begin{align*}
    A^g(k,\eta)  & \le \left(\vert k\vert+\tfrac{\vert k, \eta-\xi \vert } {\langle t\mu^{\frac 13 } \rangle} \right) A(k,\eta-\xi) +\tfrac{\vert k, \xi\vert } {\langle t\mu^{\frac 13 } \rangle}   A^g(0,\xi) .
\end{align*}
Then we estimate 
\begin{align*}
    \tfrac {\vert k \xi  \vert  }{\vert k, \eta-\xi-kt \vert^2 } \left(\vert k\vert+\tfrac {\vert k, \eta-\xi\vert }{\langle t\mu^{\frac 13 } \rangle}\right) A(k,\eta-\xi) 
    &\lesssim   \left(1+\tfrac {t}{\langle t\mu^{\frac 13 } \rangle}\right) A(k ,\eta-\xi) \langle \xi \rangle^3\le \mu^{-\frac 13 } A(k ,\eta-\xi)\langle \xi \rangle^3
\end{align*}
and 
\begin{align*}
     \tfrac {\vert \xi k \vert  }{\vert k, \eta-\xi-kt \vert^2 }\tfrac{\vert k, \xi\vert } {\langle t\mu^{\frac 13 } \rangle} k  A^g(0,\xi)&\le \mu^{-\frac 13 } \vert \xi\vert  A^g(0,\xi)\vert k,\eta-\xi\vert^3.
\end{align*}
Therefore, by using $\p_y g_==-b^x_{==}$ we estimate 
\begin{align*}
   NL_{f\to g,=,2}& \lesssim\mu^{-\frac 13 } \Vert A^g g_{\neq } \Vert_{L^2 } \Vert A f_{\neq } \Vert_{L^2 }(\Vert A  g_=\Vert_{L^2 }+\Vert A b^x_{==}\Vert_{L^2 }).
\end{align*}
Integrating in time yields 
\begin{align}
    \int_0^T NL_{=,2}\dd\tau &\lesssim \mu^{-\frac 23 } (\mu^{-(1-\gamma)} \eps)\eps^2\lesssim \delta \eps^2.\label{eq:fg=2}
\end{align}
Combining estimates \eqref{eq:fgR}, \eqref{eq:fgT}, \eqref{eq:fg=1} and \eqref{eq:fg=2} we obtain the estimate 
\begin{align*}
    \int_0^T  NL_{f\to g} \dd\tau      &\lesssim \delta \eps^2.
\end{align*}

\subsection{Nonlinear $\sim$ Forcing}\label{sec:simtofg} Now we estimate the forcing of $\sim$ onto $f$ and $g$. In the following, since $v_{\ndiamond}$ and $b_{\ndiamond}$ satisfy the same bootstrap assumption, we just write $a_{\ndiamond}$ instead of $a_{\ndiamond}$. We estimate directly 
\begin{align*}
    NL_{\sim\to f}&= \langle A f ,   A\tilde \nabla^\perp_t\cdot ((a_{\ndiamond}\cdot  \nabla_t) a^{x,y}_{\ndiamond} )_\sim\rangle \lesssim \Vert A \nabla_t f \Vert_{L^2 }\Vert A \nabla_ta_\sim  \Vert_{L^2 }\Vert Aa_\sim \Vert_{L^2 }.
\end{align*}
Integrating in time yields 
\begin{align*}
    \int_0^T NL_{\sim\to f}\dd\tau &\lesssim \mu^{-1} (\mu^{-(1-\gamma)}\eps) \eps^2\lesssim  \delta (\mu^{-(1-\gamma)}\eps)^2 . 
\end{align*}
To bound the  forcing onto $g$ we split 
\begin{align*}
     NL_{\sim\to g}& \le \vert \langle A^gg_{\neq} ,A^g\Lambda_t^{-2}\tilde \nabla^\perp_t \cdot  ((a_{\sim }\cdot  \nabla_t) a^{x,y}_{\sim })_\diamond\rangle\vert + \vert \langle A^g g_= , \p_y\vert \p_y\vert^{-2} A^g ((a_\sim\cdot \nabla_t) a^{x,y}_\sim )_{==}\rangle\vert \\
     &=NL_{\sim\to g,1}+NL_{\sim\to g,2}
\end{align*}
and bound these terms separately. \\
\textbf{Bound on $NL_{\sim\to g,1}$,} by Plancherel's identity we obtain 
\begin{align*}
    NL_{\sim\to g,1}&= \sum_{\substack{k,\tilde k, l\\ k\neq 0 }}\iint\dd (\eta,\xi) \frac{A^g(k,\eta)}{\vert k,\eta-kt\vert } \vert A^g g \vert(k,\eta) \vert a_\sim \vert(k-\tilde k,\eta-\xi,-l) \vert \Lambda_t a_\sim \vert(\tilde k ,\xi,l ). 
\end{align*}
For $k\neq 0$ we have 
\begin{align*}
    \frac{A^g(k,\eta)}{\vert k,\eta-kt\vert }&=\frac{1+\vert k\vert +\frac{\vert k,\eta \vert}{\langle \mu^{\frac 13 }t \rangle} }{\vert k,\eta-kt\vert }A(k,\eta)\lesssim \left(1+\frac{t}{\langle \mu^{\frac 13 }t \rangle}  \langle t-\tfrac \eta k \rangle^{-1} \right) A(k,\eta)\\
    &\lesssim \left(1+\mu^{-\frac 1 3 } \sqrt{\tfrac {\p_t m }m (k,\eta)} \right) (A(k-\tilde k,\eta-\xi,-l)+A(\tilde k,\xi,l )).
\end{align*}
Thus we estimate 
\begin{align*}
    NL_{\sim\to g,1}
    &\lesssim \left( \Vert A^g   g_{\neq}  \Vert_{L^2}+\mu^{-\frac 13 } \Vert A^g  \sqrt{\tfrac {\p_t m }m} g_{\neq}  \Vert_{L^2}\right)\Vert A \nabla_ta_\sim  \Vert_{L^2 }\Vert Aa_\sim \Vert_{L^2 }.
\end{align*}
Integrating in time yields 
\begin{align}
    \int_0^T NL_{\sim\to g,1} \dd\tau &\lesssim \mu^{-\frac56  } \eps^3\le \delta \eps^2.  \label{eq:NLsg1}
\end{align}
\textbf{Bound on $NL_{\sim\to g,2}$,} we use that $A^g g_==Ag_=$ and  $((a_{\ndiamond }\cdot \nabla_t) a^{x}_{\ndiamond })_{==}=(\nabla_t\cdot (a_{\ndiamond } a^{x}_{\ndiamond }) )_{==}=\p_y( a^{y}_{\ndiamond}  a^{x}_{\ndiamond})_{==}$ to estimate  
\begin{align*}
    NL_{\sim\to g,2}
    &= \vert \langle A^gg_= , A^g \p_y\p_y^{-2}((a_{\ndiamond}\cdot \nabla_t) a^{x}_{\ndiamond} )_{==}\rangle\vert =\vert  \langle A^gg_= ,A( a^{y}_{\ndiamond}  a^{x}_{\ndiamond})_{==} \rangle\vert \\
    &\lesssim \Vert A^g  g_{=}  \Vert_{L^2}(\Vert Aa_{\neq\neq}^y \Vert_{L^2 }\Vert Aa_{\neq \neq} \Vert_{L^2 }+\Vert Aa_{= \neq}^y \Vert_{L^2 }\Vert Aa_{= \neq} \Vert_{L^2 })\\
    &\lesssim \Vert A^g  g_{=}  \Vert_{L^2}(\Vert Aa_{\neq \neq}^y \Vert_{L^2 }\Vert Aa_{\neq \neq} \Vert_{L^2 } +\Vert A\nabla a_{=    }^y \Vert_{L^2 }\Vert A\nabla_t a_{\sim} \Vert_{L^2 })
\end{align*}
Integrating in time gives 
\begin{align}
    \int_0^T NL_{\sim\to g,2}\dd\tau &\lesssim \mu^{-\frac 13 } \eps^3+ \mu^{-1}\delta \mu \eps^2 \lesssim  \delta \eps^2 .\label{eq:NLsg2}
\end{align}
Therefore, combining estimate \eqref{eq:NLsg1} and \eqref{eq:NLsg2}, we obtain the estimate 
\begin{align}
    \int_0^T NL_{\sim\to g}\dd\tau &\lesssim  \delta \eps^2 .
\end{align}

\subsection{Nonlinear Forcing Onto $h$}\label{sec:ontoh} In this subsection we bound all the nonlinear terms that act on $h$. \\
\textbf{Bound on $NL_{f\to h }$,} we use that $\nabla^\perp_t \Lambda_t^{-2} f_==- v_{==}^x e_1$  to split and estimate 
\begin{align*}
    NL_{f\to h }&=\vert \langle Ah , A((\tilde \nabla^\perp_t \tilde \Lambda_t^{-2} f\cdot \tilde  \nabla_t) h ) \rangle\vert =\vert \langle Ah , A((\tilde \nabla^\perp_t \tilde \Lambda_t^{-2} f_{\neq}\cdot  \tilde \nabla_t) h ) \rangle\vert  +\vert \langle Ah_{\neq} , A(v_{==}^x  \p_x  h_{\neq} ) \rangle\vert \\
    &\lesssim (\Vert Ah \Vert_{L^2}\Vert A\Lambda_t^{-1}f \Vert_{L^2}+\Vert Ah_{\neq}  \Vert_{L^2}\Vert Av_{==}^x \Vert_{L^2})\Vert A\Lambda_th  \Vert_{L^2}.
\end{align*}
Therefore, integrating in time yields 
\begin{align*}
    \int_0^T NL_{f\to h } \dd\tau &\lesssim \mu^{-\frac 12 }(\mu^{-\frac 16}\eps) \eps^2 \le \delta \eps^2 . 
\end{align*}
\textbf{Bound on $NL_{g\to h }$,} we split
\begin{align*}
     NL_{g\to h }&= \vert \langle A h_{\neq }  , A((\tilde \nabla^\perp g_{\neq}\cdot \tilde \nabla) h_{\neq} )_{\neq } \rangle\vert +\vert \langle A h_{\neq }  , A((\p_y g_= \p_x) h_{\neq } -\p_x g_{\neq}\p_y h_= )_{\neq } \rangle\vert \\ 
     &\quad+\vert \langle A h_{=}  , A((\tilde \nabla^\perp g\cdot \tilde \nabla )h )_= \rangle\vert \\
     &=NL_{g\to h, 1}+NL_{g\to h, 2}+NL_{g\to h, 3}.
\end{align*}
For $NL_{g\to h, 1}$, by Plancherel's identity we obtain 
\begin{align*}
    NL_{g\to h, 1}&= \langle A h_{\neq }  , A((\nabla^\perp g_{\neq}\cdot \nabla) h_{\neq}) \rangle\\
    &= \sum_{\substack{k,\tilde k\neq 0 \\ \tilde k-k \neq 0 }} \iint\dd (\eta,\xi) \vert \eta \tilde k - k \xi \vert A(k,\eta)  \vert A h\vert(k,\eta)  \vert  g\vert(k-\tilde k,\eta - \xi) \vert  h\vert(\tilde k,\xi ). 
\end{align*}
Then we estimate
\begin{align*}
     \vert \eta \tilde k - k \xi \vert &= \vert k,\eta-kt \vert \frac {\vert \eta \tilde k - k \xi \vert}{\vert k,\eta-kt \vert}\\
     &\le \vert k,\eta-kt \vert \langle \tfrac \eta k \rangle \tfrac 1 {\langle t -\frac \eta k \rangle }\min\left(\vert k-\tilde k , \eta-\xi\vert, \vert \tilde k , \xi\vert\right) \\
      &\le \vert k,\eta-kt \vert (1 + \tfrac t {\langle t -\frac \eta k \rangle })\min\left(\vert k-\tilde k , \eta-\xi\vert, \vert \tilde k , \xi\vert\right) \\
      &\le \vert k,\eta-kt \vert  \left(1 +  t \sqrt{\frac {\p_t m } m (k-\tilde k , \eta-\xi)}+  t \sqrt{\frac {\p_t m } m (\tilde k , \eta-\xi)}\right)\\
      &\qquad \qquad \cdot \min\left(\vert k-\tilde k , \eta-\xi\vert^5, \vert \tilde k , \xi\vert^5\right).
\end{align*}
Therefore, we obtain the estimate 
\begin{align*}
     NL_{g\to h, 1}&\lesssim  e^{-c\mu^{\frac 13 } t } \Vert  A \Lambda_t h \Vert_{L^2} \left(\Vert A  g \Vert_{L^2}  \Vert A h \Vert_{L^2} +t \Vert A \sqrt{\tfrac {\p_t m } m} h \Vert_{L^2} \Vert A  g \Vert_{L^2} +t \Vert A \sqrt{\tfrac {\p_t m } m} g \Vert_{L^2} \Vert A h \Vert_{L^2} \right).
\end{align*}
Integrating in time yields 
\begin{align*}
     \int_0^T NL_{g\to h, 1}\dd\tau &\lesssim \mu^{-\frac 56}\eps^3 \le \delta \eps^2 .
\end{align*}
For $NL_{g\to h,2}$ we estimate directly 
\begin{align*}
    NL_{g\to h, 2}&\lesssim \Vert A h_{\neq } \Vert_{L^2} \left( \Vert  A \p_y g_=\Vert_{L^2 }\Vert  A \p_x h_{\neq} \Vert_{L^2 } + \Vert  A \p_x g_{\neq}\Vert_{L^2 }\Vert  A \p_y h_{=} \Vert_{L^2 } \right)\\
    &\lesssim \Vert A h_{\neq } \Vert_{L^2} \left( \Vert  A b_{==}^x\Vert_{L^2 }\Vert  A \nabla_t  h_{\neq} \Vert_{L^2 } + \Vert  A^g g_{\neq}\Vert_{L^2 }\Vert  A \nabla h_{=} \Vert_{L^2 } \right).
\end{align*}
Therefore, after integrating in time 
\begin{align*}
   \int_0^T  NL_{g\to h, 2}\dd\tau &\lesssim \mu^{-\frac 23 }  \eps^3\le \delta \eps^2. 
\end{align*}
For $NL_{g\to h, 3}$, by Plancherel's identity, we obtain 
\begin{align*}
    NL_{g\to h, 3}&= \vert  \langle A h_=  , A(\tilde \nabla^\perp g\cdot \tilde \nabla h ) \rangle\vert\le \sum_{\substack{k\neq 0  }} \iint \dd (\eta,\xi) \vert \eta  k  \vert A(0,\eta)  \vert A h\vert(0,\eta)  \vert  g\vert(- k,\eta - \xi) \vert  h\vert( k,\xi ) \\
&\lesssim \Vert A\p_y  h_= \Vert_{L^2}\Vert A  h_{\neq}\Vert_{L^2}\Vert Ag_{\neq}\Vert_{L^2}
\end{align*}
and therefore after integrating in time
\begin{align*}
    \int_0^T   NL_{g\to h, 2}\dd\tau &\lesssim \mu^{-\frac 23 } \eps^3 \lesssim \delta \eps^2.
\end{align*}
\textbf{Bound on $NL_{\sim\to h }$,} we estimate directly 
\begin{align*}
    NL_{\sim\to h }&=\langle Ah_i, ((a_{\ndiamond }\cdot \nabla_t)a^z_{\ndiamond })_\diamond \rangle=\langle Ah_{i,\neq }, ((a_{\ndiamond }\cdot \nabla_t)a^z_{\ndiamond })_\diamond \rangle +\langle Ah_{i, =}, ((a_{\ndiamond }\cdot \nabla_t)a^z_{\ndiamond  })_{==} \rangle\\
    &\lesssim (\Vert Ah_{\neq} \Vert_{L^2}\Vert Aa_{\sim } \Vert_{L^2}+\Vert Ah_= \Vert_{L^2}\Vert Aa_{\ndiamond } \Vert_{L^2})\Vert A\nabla_t a_{\sim  }  \Vert_{L^2}
\end{align*}
We have $h_{i, =} =  a_{==}^z $ for some $a\in\{v,b\}$ and so
\begin{align*}
    NL_{\sim\to h }
    &\lesssim (\Vert Ah_{\neq} \Vert_{L^2}\Vert Aa_{\sim } \Vert_{L^2}+\Vert Av^{z}_{==} \Vert_{L^2}\Vert A\nabla_t a_{\ndiamond } \Vert_{L^2})\Vert A\nabla_t a_{\sim  }  \Vert_{L^2}.
\end{align*}
Therefore, after integrating in time, we infer 
\begin{align*}
    \int_0^T  NL_{\sim\to h }\dd\tau &\lesssim \mu^{-\frac 23} \eps^3+ \mu^{-1} \delta\mu \eps^2  \lesssim \delta \eps^2.
\end{align*}
With these estimates, we conclude Lemma \ref{lem:fgest}. By Lemma \ref{lem:Efg}, we infer Proposition \ref{prop:fgest}.

\section{Improved $x$-Average Estimates}\label{sec:xaver}
In this subsection, we improve the bootstrap assumption \eqref{boot1q}.  We consider the $x$ average unknowns $(v^{y,z}_=, b^{y,z}_=)$ and establish an improved threshold.  In particular, we prove the following proposition:
\begin{prop}\label{prop:qest}
    Under the assumption of Theorem~\ref{thm:main} and let the bootstrap assumption on the interval $[0,T]$ hold. Then there exists a constant $C_{3,=}>0$ such that 
    \begin{align*}
    \Vert  (v^{y,z}_=, b^{y,z}_=) \Vert_{L^\infty_T H^{N}}^2 +\int_0^T\mu\Vert \nabla (v^{y,z}_=, b^{y,z}_=) \Vert_{H^{N}}^2\dd\tau &\le (C_1 +\delta C_{3,=})  (\delta \mu)^2.
\end{align*}
\end{prop}

The unknowns $v^{y,z}_=$ and $ b^{y,z}_=$ satisfy the 2d MHD equations around a constant magnetic field (without Couette flow), forced by the $x$ dependent unknowns:
\begin{align*}
    \p_t v^{y,z}_=+(v^{y,z}_=\cdot \nabla_{y,z}) v^{y,z}_=+\nabla_{y,z} \pi_= &=\mu \Delta v^{y,z}_=+\alpha \p_z v^{y,z}_= +  ( b^{y,z}_=\cdot \nabla_{y,z})  b^{y,z}_=+F_1,\\
    \p_t b^{y,z}_=+(v^{y,z}_=\cdot \nabla_{y,z})  b^{y,z}_= \ \, \qquad \qquad  &=\mu \Delta  b^{y,z}_=+\alpha \p_z  b^{y,z}_= +( b^{y,z}_=\cdot \nabla_{y,z}) v^{y,z}_=+F_2,\\
    v^{x,y}_{=,in}=b^{x,y}_{=,in}=0, 
\end{align*}
where we denote 
\begin{align*}
    F_1&=((b\cdot \nabla_t) b^{y,z} -(v\cdot \nabla_t) v^{y,z})_{= }-(b^{y,z}_=\cdot \nabla_{y,z})  b^{y,z}_= +(v^{y,z}_=\cdot \nabla_{y,z}) v^{y,z}_=,\\
    F_2&= ((b\cdot \nabla_t) v^{y,z} -(b\cdot \nabla_t) v^{y,z})_{=  }-( b^{y,z}_=\cdot \nabla_{y,z}) v^{y,z}_=+(v^{y,z}_=\cdot \nabla_{y,z})  b^{y,z}_= .
\end{align*}
\begin{lemma}\label{lem:qstruc}
Under the assumptions of Theorem~\ref{thm:main} and assuming the bootstrap assumption holds for $t\in [0,T]$, then we obtain 
\begin{align*}
    \frac 12 \p_t \Vert (v^{y,z}_=, b^{y,z}_=) \Vert_{H^{N}}^2+ \mu \Vert\nabla  (v^{y,z}_=, b^{y,z}_=) \Vert_{H^{N}}^2 &\le NL_= +NL_{\sim \to = } +NL_{\diamond \to =}.
\end{align*}
For $\Gamma = \{(v,v,v), (v,b,b), (b, v,b), (b, b,v)\}$ we denote \\
\textbf{Nonlinear average self interaction}
\begin{align*}
    NL_=&= \sum_{a^i \in\{v,b\}}\vert \langle a^{1,y,z}_=, (a^{2,y,z}_=\cdot \nabla_{y,z}) a^{3,y,z}_= \rangle_{H^{N}}\vert .
\end{align*}
\textbf{Nonlinear $\sim$ forcing}
\begin{align*}
    NL_{\sim \to = }&= \sum_{(a^1,a^2, a^3)\in\Gamma}\vert \langle a^{1,y,z}_=, (a_{\neq \neq }^2\cdot \nabla_t) a_{\neq \neq } ^{3,y,z}\rangle_{H^{N}}\vert.
\end{align*}
\textbf{Nonlinear $\diamond $ forcing}
\begin{align*}
    NL_{\diamond, \sim \to = }&=  \sum_{(a^1,a^2, a^3)\in\Gamma}\vert \langle a^{1,y,z}_=, (a_{\neq = }^2\cdot \nabla_t) a_{\neq \neq }^{3,y,z} \rangle_{H^{N}}\vert, \\
    NL_{\sim, \diamond \to = }&= \sum_{(a^1,a^2, a^3)\in\Gamma}\vert \langle a^{1,y,z}_=, (a_{\neq \neq  }^2\cdot \nabla_t) a_{\neq = }^{3,y,z} \rangle_{H^{N}}\vert, \\
    NL_{\diamond  \to = }&= \sum_{(a^1,a^2, a^3)\in\Gamma}\vert \langle a^{1,y,z}_{==}, ((a_{\neq =  }^2\cdot\nabla_t) a_{\neq = }^{3,y,z})_{==} \rangle_{H^{N}}\vert. 
\end{align*}
\end{lemma}
This lemma is obtained by direct calculation using the identity
\begin{align*}
    ((v\cdot \nabla_t)v)_{=}&=((v_{= }\cdot \nabla_t)v_{= })_{=}+((v_{\neq \neq }\cdot \nabla_t)v_{\neq \neq })_{=}+((v_{\neq = }\cdot \nabla_t)v_{\neq \neq })_{=}\\
    &\quad    +((v_{\neq = }\cdot \nabla_t)v_{\neq \neq })_{=}+((v_{\neq = }\cdot \nabla_t)v_{\neq = })_{=}.
\end{align*}

\begin{lemma}\label{lem:qest}
    Under the assumptions of Theorem~\ref{thm:main} and let the bootstrap assumptions hold on $[0,T]$, then there exists a $C_{2,=}$ such that we obtain the estimate 
    \begin{align*}
        \int_0^T NL_= +NL_{\sim \to = } +NL_{\diamond, \sim \to = }+NL_{\sim, \diamond \to = }+NL_{\diamond \to = }\dd\tau &\le C_{2,=} \delta (\delta \mu )^2.
    \end{align*}
\end{lemma}
Since we consider initial data with vanishing $x$ average we have $v_{=,in}^{y,z}=b_{=,in}^{y,z}=0$ and so this lemma yields Proposition \ref{prop:qest}. For the remainder of this section, we prove Lemma \ref{lem:qest}. As in the previous sections, we write $a_\sim $ instead of $ a^1_\sim $ and $a^2_\sim$ since the bootstrap assumption is the same. For the $\neq =$ terms, we distinguish between $v_{\neq =}$ and $b_{\neq =}$ and change to the $fgh$ unknowns, since their behavior is different.

\textbf{The averag self interaction estimates:} We drop the indice $i$ of $a^{i,y,z}_=$ since the play no essential role. With  $  a^{y,z}_{==} (y) = a^{z}_{==}(y) e_z $ we estimate 
\begin{align*}
    NL_=&= \vert \langle a^{y,z}_=, (a^{y,z}_=\cdot \nabla) a^{y,z}_= \rangle_{H^{N}}\vert \le \vert \langle a^{y,z}_=,( a^{y,z}_{=\neq }\cdot  \nabla )a^{y,z}_= \rangle_{H^{N}}\vert + \vert  \langle a^{y,z}_=,  a^{y,z}_{==}\p_z a^{y,z}_= \rangle_{H^{N}}\vert \\
    &=  \vert \langle a^{y,z}_=,( a^{y,z}_{=\neq }\cdot  \nabla )a^{y,z}_= \rangle_{H^{N}}\vert + \vert  \langle a^{y,z}_{=\neq },  a^{y,z}_{==}\p_z a^{y,z}_{=\neq } \rangle_{H^{N}}\vert \\
    &\lesssim \Vert a^{y,z}_= \Vert_{H^{N}}\Vert \nabla a^{y,z}_= \Vert_{H^{N}}^2. 
\end{align*}
Therefore, after integrating in time 
\begin{align*}
    \int_0^T NL_= \dd\tau &\lesssim \mu^{-1} (\delta \mu)^3\lesssim\delta (\delta \mu) ^2 .     
\end{align*}
\textbf{Bound on $ NL_{\sim \to =}$,} we estimate directly 
\begin{align*}
    NL_{\sim \to =}&=\vert \langle a^{y,z}_=, (a_{\neq \neq }\cdot \nabla_t) a_{\neq \neq } ^{y,z}\rangle_{H^{N}}\vert \lesssim \Vert a^{y,z}_= \Vert_{H^{N} }\Vert A a_{\neq \neq }\Vert_{L^2 }\Vert A \nabla_t a_{\neq \neq }\Vert_{L^2 }. 
\end{align*}
Integrating in time yields
\begin{align*}
    \int_0^T  NL_{\sim \to = } \dd\tau &\lesssim \mu^{-\frac 23 } (\delta \mu) \eps^2\lesssim \delta (\delta \mu)^2.
\end{align*}
\textbf{Bound on $NL_{\diamond, \sim \to = }$,} we split this term into  
\begin{align*}
      NL_{\diamond, \sim \to = }&=  \sum_{(a^1,a^2, a^3)\in\Gamma_a}\vert \langle a^{1,y,z}_=, (a_{\neq = }^2\cdot \nabla_t) a_{\neq \neq }^{3,y,z} \rangle_{H^{N}}\vert\\
     &\le \vert \langle a^{y,z}_=, ( (\tilde \nabla^\perp g_{\neq}\cdot  \tilde \nabla) a_{\neq \neq }^{y,z}) \rangle_{H^N}\vert  +\vert \langle  \langle a^{y,z}_=, ( (\tilde \nabla^\perp_t \tilde \Lambda_t^{-2}f _{\neq}\cdot \tilde \nabla_t) a_{\neq \neq }^{y,z}) \rangle_{H^N}\vert  +\vert \langle  a^{y,z}_=,( h_{\neq}\p_z a_{\neq \neq }^{y,z}) \rangle_{H^N}\vert \\
     &= NL_{g, \sim \to = }+ NL_{f, \sim \to = }+NL_{h, \sim \to = }.
\end{align*}
For $ NL_{g, \sim \to = }$ we obtain by partial integration 
\begin{align*}
    NL_{g, \sim \to = }&=\vert \langle (\tilde \nabla\otimes a^{y,z}_=), ( \tilde \nabla^\perp g_{\neq} \otimes  a_{\neq \neq }^{y,z}) \rangle_{H^N}\vert \lesssim e^{-c\mu^{\frac 13 } t } \Vert \nabla a^{y,z}_= \Vert_{L^2 }\Vert Ag_{\neq} \Vert_{H^1} \Vert A  a_{\neq \neq }^{y,z}\Vert_{L^2 },\\
    &\lesssim e^{-c\mu^{\frac 13 } t }\langle t \mu^{\frac 13 } \rangle  \Vert \nabla a^{y,z}_= \Vert_{L^2 }\Vert A^g g \Vert_{L^2} \Vert A  a_{\neq \neq }\Vert_{L^2 }.
\end{align*}
Therefore, after integrating in time 
\begin{align}
    \int_0^T  NL_{g, \sim \to = }\dd\tau &\lesssim \mu^{-\frac 23 }  (\delta \mu) \eps^2 \lesssim \delta (\delta \mu)^2.\label{eq:NLgsimq}
\end{align}
For $ NL_{f, \sim \to = }$, we estimate directly
\begin{align*}
     NL_{f, \sim \to = }&=\vert \langle  a^{y,z}_=, ( (\tilde \nabla^\perp_t  \Lambda_t^{-2}f \cdot \tilde \nabla_t) a_{\neq \neq }^{y,z}) \rangle_{H^N}\vert \lesssim \Vert Aa^{y,z}_=\Vert_{L^2}\Vert A \Lambda_t^{-1}f \Vert_{L^2}\Vert A \nabla_t a_{\neq \neq }^{y,z}\Vert_{L^2}. 
\end{align*}
Thus, after integrating in time 
\begin{align}
     \int_0^T  NL_{f, \sim \to = }\dd\tau &\lesssim \mu^{-\frac 12 }(\delta \mu )(\mu^{-(1-\gamma)} \eps )\eps \lesssim \delta (\delta \mu)^2. \label{eq:NLfsimq}
\end{align}
Similar we estimate 
\begin{align}
    \int_0^T  NL_{h, \sim \to =}\dd\tau \lesssim \int_0^T \Vert Aa^{y,z}_=\Vert_{L^2}\Vert A h \Vert_{L^2}\Vert A \nabla_\tau  a_{\neq \neq }^{y,z}\Vert_{L^2} \dd \tau  &\lesssim \mu^{-\frac 23 }(\delta \mu)\eps^2 \lesssim \delta (\delta \mu)^2. \label{eq:NLhsimq}
\end{align}
Combining the estimates \eqref{eq:NLgsimq}, \eqref{eq:NLfsimq} and \eqref{eq:NLhsimq}, we obtain 
\begin{align*}
     \int_0^T NL_{\diamond, \sim \to =}\dd\tau&\lesssim  \delta (\delta \mu)^2.
\end{align*}

\textbf{Bound on $NL_{\sim, \diamond \to =} $,} we split this term into 
\begin{align*}
     NL_{\sim, \diamond \to = }&= \sum_{(a^1,a^2, a^3)\in\Gamma_a}\vert \langle a^{1,y,z}_=, (a_{\neq \neq  }^2\cdot \nabla_t) a_{\neq = }^{3,y,z} \rangle_{H^{N}}\vert\\
    &\le\vert \langle a^{y}_=,( (a_{\neq \neq  }\cdot \nabla_t) \p_x g_{\neq }  )\rangle_{H^N}\vert+\vert \langle a^{y}_=, ((a_{\neq \neq  }\cdot \nabla_t) \p_x \Lambda_t^{-2}f_{\neq }  )\rangle_{H^N}\vert \\
    &\quad +\vert \langle a^{z}_=,( (a_{\neq \neq  }\cdot \nabla_t)  h_{\neq } ) \rangle_{H^N}\vert \\
    &= NL_{\sim, g \to = }+ NL_{\sim, f \to = }+ NL_{\sim, h  \to = }.
\end{align*}
We estimate $NL_{\sim, g \to =}$ by 
\begin{align*}
    NL_{\sim, g \to = }&\lesssim \Vert  Aa^{y}_=\Vert_{L^2 } \Vert A a_{\neq \neq  }\Vert_{L^2 } \Vert A\p_x\nabla_t   g_{\neq }  \Vert_{L^2 } \lesssim \Vert  Aa^{y}_=\Vert_{L^2 } \Vert A a_{\neq \neq  }\Vert_{L^2 } \Vert A^g\nabla_t   g_{\neq }  \Vert_{L^2 } .
\end{align*}
Integrating in time yields
\begin{align}
    \int_0^T NL_{\sim, g \to =}\dd\tau &\lesssim \mu^{-\frac 23 } (\delta \mu ) \eps^2 . \label{eq:NLsimgq}
\end{align}
The term $NL_{\sim, f \to = }$ we estimate by 
\begin{align*}
    NL_{\sim, f \to =}&\lesssim \Vert Aa^{y,z}_= \Vert_{L^2}\Vert A a_{\neq \neq  }\Vert_{L^2} \Vert Af_{\neq} \Vert_{L^2}.
    \end{align*}
Integrating in time yields
\begin{align}
    \int_0^T NL_{\sim, f \to = }\dd\tau &\lesssim \mu^{-\frac 1 3 } (\delta \mu )(\mu^{-(1-\gamma)} \eps )\eps \le \delta (\delta \mu)^2.  \label{eq:NLsimfq}
\end{align}
We estimate $ NL_{\sim, h  \to = }$ by 
\begin{align*}
     NL_{\sim, h  \to = }&=\vert \langle Aa^{z}_=,A( (a_{\neq \neq  }\cdot \nabla_t)  h_{\neq } ) \rangle \vert \lesssim \Vert Aa^{y,z}_= \Vert_{L^2}\Vert A a_{\neq \neq} \Vert_{L^2} \Vert A\nabla_t  h_{\neq } \Vert_{L^2}.
    \end{align*}
Integrating in time yields
\begin{align}
     \int_0^T  NL_{\sim, h  \to = }\dd\tau &\lesssim \mu^{-\frac 23} (\delta \mu) \eps^2\lesssim \delta \eps^2.\label{eq:NLsimhq}
\end{align}
Combining estimates \eqref{eq:NLsimgq},  \eqref{eq:NLsimfq} and  \eqref{eq:NLsimhq} we obtain 
\begin{align*}
    \int_0^T NL_{\sim, \diamond   \to = }\dd\tau&\lesssim \delta (\delta \mu)^2.
\end{align*}
\textbf{Bound on $NL_{\diamond \to = }$,} since $ a^{y}_{==}=0$ we obtain  
\begin{align*}
    NL_{\diamond \to = }&= \sum_{(a^1,a^2, a^3)\in\Gamma_a}\vert \langle a^{1,y,z}_{==}, ((a_{\neq =  }^2\cdot\nabla_t) a_{\neq = }^{3,y,z})_{==} \rangle_{H^{N}}\vert\\
    &= \sum_{(a^1,a^2, a^3)\in\Gamma_a}\vert \langle a^{1,z}_{==}, ((a_{\neq =  }^2\cdot\nabla_t) a_{\neq = }^{3,z})_{==} \rangle_{H^{N}}\vert. \\
    &\le \vert \langle A  a^{z}_{==} ,  A((\tilde \nabla^\perp g\cdot  \tilde \nabla) h )\rangle\vert +\vert \langle A  a^{z}_{==}, A( (\tilde \nabla^\perp \tilde \Lambda_t^{-2} f \cdot \tilde \nabla) h )\rangle\vert \\
    &=NL_{g, h \to = }+NL_{f,h \to = }.
\end{align*}
For $NL_{g, h \to= }$ we use Plancherel's identity to infer 
\begin{align*}
    NL_{g, h \to = }&= \sum_{k\neq 0 }\iint \dd (\eta ,\xi ) \vert \eta k\vert  A(0,\eta,0)\vert Aa^z\vert(0,\eta,0)\vert g \vert (-k,\eta-\xi) \vert h \vert (k,\xi) \\
    &\lesssim \Vert A\p_y  a^{z}_{==}\Vert_{L^2}\Vert Ag_{\neq}\Vert_{L^2}\Vert Ah_{\neq}\Vert_{L^2}. 
\end{align*}
Integrating in time yields 
\begin{align*}
    \int_0^T NL_{g, h \to = }\dd\tau &\lesssim \mu^{-\frac 23 } (\delta \mu) \eps^2  \lesssim \delta (\delta \mu)^2.
\end{align*}
Similar we estimate 
\begin{align*}
    NL_{f, h \to = }
    &\lesssim \Vert A\p_y  a^{y,z}_{==} \Vert_{L^2}\Vert A\Lambda_t^{-1} f_{\neq}\Vert_{L^2}\Vert Ah_{\neq}\Vert_{L^2}.
\end{align*}
Finally, we estimate 
\begin{align*}
    \int_0^T  NL_{f, h \to = }\dd\tau &\lesssim \mu^{-\frac 12 } (\delta \mu) (\mu^{-(1-\gamma)} \eps) \eps  \lesssim \delta (\delta \mu)^2.
\end{align*}
Therefore, we obtain all the necessary estimates to conclude Lemma \ref{lem:qest}.   By Lemma \ref{lem:qstruc} we obtain Proposition \ref{prop:qest}. 

\section{Inviscid Damping Estimates}\label{sec:invis}
In this section, we improve the bootstrap assumption  \eqref{boot2}. First, we derive the equations for the adapted unknowns from  $\tilde v_{\sim }^y$ and $\tilde b_{\sim}^y$, then we state the main proposition, which we prove for the rest of the section. We have the equation 
\begin{align*}
&\begin{cases}
    \p_t \tilde v^y  = \mu \Delta_t \tilde v^y  + \alpha \p_z \tilde  b^y  +2\p_y^t  \Delta^{-1}_t \p_x  \tilde v^y +((b\cdot \nabla_t) b^y -(v\cdot \nabla_t) v^y  -\p_y^t \pi)_{\sim},\\
    \p_t \tilde b^y= \mu \Delta_t  \tilde b^y+ \alpha \p_z \tilde v^y+((b\cdot \nabla_t) v^y -(v\cdot \nabla_t) b^y )_{\sim },\\
    \Delta_t \pi = \p_i^tb^j\p_j^tb^i -\p_i^tv^j\p_j^tv^i.
\end{cases}
\end{align*}
Then we define the unknowns 
\begin{align*}
    \rho_1 &= \langle\p_x\rangle ^{-1} \Lambda_t \tilde v^y,& \rho_2 &=\langle \p_x\rangle^{-1} \Lambda_t\tilde b^y, 
\end{align*}
which satisfy the equations 
\begin{align*}
&\begin{cases}
    \p_t \rho_1  = \mu \Delta_t \rho_1  + \alpha \p_z \rho_2 +\p_x\p_y^t  \Delta^{-1}_t   \rho_1 +\langle\p_x\rangle ^{-1} \Lambda_t((b\cdot \nabla_t) b^y -(v\cdot\nabla_t) v^y  -\p_y^t \pi)_{\sim },\\
    \p_t \rho_2= \mu \Delta_t  \rho_2+ \alpha \p_z \rho_1-\p_x \p_y^t  \Delta^{-1}_t \rho_2+\langle\p_x\rangle ^{-1} \Lambda_t((b\cdot \nabla_t) v^y -(v\cdot \nabla_t) b^y )_{\sim},\\
    \Delta_t \pi = \p_i^tb^j\p_j^tb^i -\p_i^tv^j\p_j^tv^i.
\end{cases}
\end{align*}
Note that by definition of $\rho$ we have $\rho_{\diamond } =0$, since $\int \tilde v^y dz =v^y_{==} =0$. Then, we define the adapted unknowns 
\begin{align*}
    \tilde \rho_1&= \rho_1 +\tfrac {\p_x} {\alpha\p_z} \p_y^t \Lambda_t^{-2}\rho_2,  &
    \tilde \rho_2 &= \rho_2,
\end{align*}
which satisfy the equation
\begin{align*}
    \p_t \tilde \rho_1  &= \mu \Delta_t \tilde \rho_1  + \alpha \p_z \tilde \rho_2 + \tfrac 1 {\alpha \p_z } \p_x^2(2(\p_y^t)^2-\p_x^2)  \Delta_t^{-2} \tilde \rho_2\\
    &\quad  + \langle\p_x\rangle ^{-1}\Lambda_t((b\cdot \nabla_t) b^y -(v\cdot \nabla_t) v^y  -\p_y^t \pi)_{\sim}+\tfrac {\p_x} {\alpha\p_z \langle\p_x\rangle} \p_y^t \Lambda_t^{-1}((b\cdot \nabla_t) v^y -(v\cdot \nabla_t) b^y)_{\neq \neq },\\
    \p_t \tilde \rho_2 &= \mu \Delta_t  \tilde \rho_2+ \alpha \p_z\tilde  \rho_1+\langle\p_x\rangle^{-1} \Lambda_t((b\cdot\nabla_t) v^y -(v\cdot \nabla_t) b^y )_{\sim},\\
    \Delta_t \pi &= \p_i^tb^j\p_j^tb^i -\p_i^tv^j\p_j^tv^i.
\end{align*}
\begin{prop}\label{prop:Invest}
    Under the assumption of Theorem~\ref{thm:main} and let the bootstrap assumption on the interval $[0,T]$ hold. Then there esinsts a $C_{3,\rho}>0$  such that we obtain the estimate 
        \begin{align*}
         \Vert A\tilde  \rho\Vert_{L^\infty_T  L^2}^2 &+\int_0^T \mu \Vert A \nabla_\tau \tilde \rho\Vert_{L^2}^2 +\Vert A \sqrt{\tfrac {\p_t M} M }\tilde  \rho\Vert_{ L^2}^2\dd\tau  \le (C_1+ 2cC_2+ C_{3,\rho}\delta) \eps^2.
    \end{align*}

\end{prop}
\begin{lemma}\label{lem:Invlem}
Let the bootstrap assumption hold on the interval $[0,T]$. Then we obtain 
    \begin{align*}
        \frac1 2 \p_t \Vert A\tilde  \rho\Vert_{L^2}^2 &+\mu \Vert A \nabla_t \tilde \rho\Vert_{L^2}^2 +\Vert A \sqrt{\tfrac {\p_tM} M }\tilde  \rho\Vert_{L^2}^2 \le L +NL_{\sim \to \rho}+NL_{\diamond  \to \rho}+NLP_\rho+ LNL.
    \end{align*}
\textbf{Linear terms:}
\begin{align*}
    L&=  c \mu^{\frac 13 }\Vert A\tilde \rho_{\neq \neq} \Vert_{L^2}^2+ \vert \langle A\tilde \rho_1 , A\tfrac 1 {\alpha \p_z } \p_x^2(\p_x^2-2(\p_y^t)^2)  \Delta_t^{-2} \tilde \rho_2\rangle\vert .
\end{align*}
\textbf{Nonlinear forcing by $\sim$:}
\begin{align*}
    NL_{\sim \to \rho}&=\sum_{\substack{i,j=1,2 \\ a\in\{v,b\}}}\vert \langle A\tilde \rho_i, \langle\p_x\rangle ^{-1}\Lambda_t A(  (a_\sim \cdot \nabla_t) \Lambda_t^{-1}\langle\p_x\rangle \rho_j)_\sim\rangle\vert.
\end{align*}
\textbf{Nonlinear forcing by  $\sim$ and $\diamond $:}
\begin{align*}
    NL_{\diamond ,\sim \to \rho}&=\sum_{\substack{i,j=1,2 \\ a\in\{v,b\}}}\vert \langle A\tilde \rho_i, \langle \p_x\rangle^{-1} \Lambda_t A ( (a_{\neq =} \cdot \nabla_t) \Lambda_t^{-1}\langle\p_x\rangle\rho_j )_\sim\rangle \vert \\
    &\quad +\sum_{\substack{i=1,2\\ a^1,a^2\in \{v,b\}}}\vert \langle A\tilde \rho_i, \langle \p_x\rangle^{-1} \Lambda_tA( (a_\sim^1  \cdot \nabla_t) a^{2,y}_{\neq =})_\sim\rangle\vert .
\end{align*}
\textbf{Nonlinear pressure:}
\begin{align*}
     NLP_\rho &= \vert \langle  A\tilde\rho, A\langle \p_x \rangle^{-1} \Lambda_t \p_y^t  \pi_\sim\rangle\vert .
\end{align*}
\textbf{Lower nonlinear terms:}
\begin{align*}
    LNL_\rho &= \vert  \langle A\tilde \rho_1 ,\tfrac {\p_x} {\alpha\p_z \langle\p_x\rangle} \p_y^t \Lambda_t^{-1}((b\cdot \nabla_t) v^y -(v\cdot \nabla_t) b^y)_\sim\rangle\vert .
\end{align*}
\end{lemma}
We obtain this Lemma by direct calculations and using that 
\begin{align*}
    ((v\cdot \nabla_t) v^y)&=((v_\sim \cdot \nabla_t) v_\sim^y )+((v_\sim \cdot \nabla_t) v_{\neq=}^y)+((v_{\neq=}\cdot \nabla_t) v_\sim^y )+((v_{\neq=}\cdot \nabla_t) v_{\neq=}^y)\\
    &=((v_\sim \cdot \nabla_t) \langle \p_x\rangle \Lambda_t^{-1} \rho_1 )+((v_\sim \cdot \nabla_t) v_{\neq=}^y)+((v_{\neq=}\cdot \nabla_t)  \langle \p_x\rangle \Lambda_t^{-1} \rho_1 )+((v_{\neq=}\cdot \nabla_t) v_{\neq=}^y).
\end{align*}
For $\rho$ the ${\neq=},{\neq=}$ interaction vanishes. I.e. since $\rho_{==}=\rho_{\neq=}=0$ we obtain 
\begin{align*}
    \langle A\tilde \rho_1 , \langle \p_x\rangle^{-1} \Lambda_t A ((v_{\neq=}\cdot \nabla_t) v_{\neq=}^y)_\sim \rangle =\langle A(\tilde \rho_{1,==}+\tilde \rho_{1,\neq=}) , \langle \p_x\rangle^{-1} \Lambda_t A ((v_{\neq=}\cdot \nabla_t) v_{\neq=}^y)_\sim \rangle=0
\end{align*}
and similar for the terms involving $b$ and $\rho_2$.

\subsection{Proof of Proposition \ref{prop:Invest}}
To prove Proposition \ref{prop:Invest}, we estimate all the terms of Lemma \ref{lem:Invlem}. Again, we only distinguish between different choices of $a_{\neq =}$, for $\rho_i$ and $a_\sim^i$ we drop the indices.  \\
\textbf{Linear estimates} we estimate 
\begin{align*}
    \vert \langle A\tilde \rho_1 , A\tfrac 1 {\alpha \p_z } \p_x^2(\p_x^2-2(\p_y^t)^2)  \Delta_t^{-2} \tilde \rho_2\rangle\vert &\le  c \Vert \sqrt{\tfrac {\p_t m } m} A\rho\Vert_{L^2}^2. 
\end{align*}
Furthermore, by Lemma \ref{lem:Mmu}  we have $$\mu^{\frac 1 3 } \Vert A\tilde \rho_{\neq \neq} \Vert_{L^2L^2}^2\le\Vert \sqrt{\tfrac{\p_t M_\mu}{M_\mu}}A\tilde \rho_{\neq \neq} \Vert_{L^2L^2}^2+\mu\Vert A\nabla_t \tilde \rho_{\neq \neq} \Vert_{L^2L^2}^2. $$
Therefore,
\begin{align*}
    \int_0^T L\dd\tau &\le  3 c C_1 \eps^2.
\end{align*}

\textbf{Bound on $NL_{\sim \to \rho}$,} we use that $ \rho_{\diamond}=0$ and $(a_{= } \cdot \nabla_t )\Lambda_t^{-1}\langle\p_x\rangle  \rho_{= \neq }=(a^{y,z} \cdot \nabla ) \Lambda^{-1} \rho_{= \neq }$ to infer 
\begin{align*}
     NL_{\sim \to \rho}&=\vert \langle A\tilde \rho, \langle\p_x\rangle ^{-1}\Lambda_t A( ( a_\sim \cdot \nabla_t) \Lambda_t^{-1}\langle\p_x\rangle  \rho)_{\sim}\rangle\vert  \\
     &=\vert \langle A\tilde \rho, \langle\p_x\rangle ^{-1}\Lambda_t A(  ((a_{\neq \neq }+a_{= }) \cdot \nabla_t )\Lambda_t^{-1}\langle\p_x\rangle  (\rho_{\neq \neq }+\rho_{= \neq }))\rangle\vert  \\
     &\lesssim \Vert \Lambda_t \tilde \rho \Vert_{L^2 } (\Vert A \rho \Vert_{L^2 } \Vert  Aa_{\neq \neq } \Vert_{L^2 } +\Vert A \rho_{\neq\neq  } \Vert_{L^2 } \Vert  Aa_\sim  \Vert_{L^2 }+\Vert A a_{= }^{y,z} \Vert_{L^2}\Vert A \rho_{=\neq } \Vert_{L^2})\\
     &\lesssim \Vert \Lambda_t \tilde \rho \Vert_{L^2 } (\Vert A \rho \Vert_{L^2 } \Vert  Aa_{\neq \neq } \Vert_{L^2 } +\Vert A \rho_{\neq\neq  } \Vert_{L^2 } \Vert  Aa_\sim  \Vert_{L^2 }+\Vert A  a_{= }^{y,z}  \Vert_{L^2}\Vert A \nabla \rho_{=\neq } \Vert_{L^2}).
\end{align*}
Integrating in time yields 
\begin{align*}
     \int_0^T NL_{\sim \to \rho}\dd\tau &\lesssim \mu^{-\frac 23 } \eps^3 +  \mu^{-1} (\delta \mu) \eps^2 \lesssim \delta \eps^2. 
\end{align*}
\textbf{Bound on $NL_{\diamond \to \rho}$,} we insert the different terms of $v_{\neq =}=(-\p_y^t \Lambda_t^{-2}f, \p_x \Lambda_t^{-2}f,h_1)_{\neq} $ and $v_{\neq =}=(-\p_y^t f, \p_x f,h_2)_{\neq}$ to split 
\begin{align*}
    NL_{ \diamond,\sim\to \rho}&\le \vert \langle A\tilde \rho, \langle\p_x\rangle^{-1}\Lambda_t A( (\tilde \nabla^\perp \tilde \Lambda_t^{-2} f_{\neq}  \cdot \tilde \nabla) \Lambda_t^{-1}  \langle\p_x\rangle \rho)\rangle\vert +\vert \langle A\tilde \rho, \langle\p_x\rangle^{-1}\Lambda_tA( (a_\sim  \cdot \nabla_t) \p_x \tilde \Lambda_t^{-2} f_{\neq}) \rangle\vert   \\
    &\quad +\vert \langle A\tilde \rho, \langle\p_x\rangle^{-1}\Lambda_t A ( (\tilde \nabla^\perp g_{\neq}  \cdot \tilde \nabla) \Lambda_t^{-1}  \langle\p_x\rangle\rho ) \rangle\vert +\vert \langle A\tilde \rho, \langle\p_x\rangle^{-1}\Lambda_t A( (a_\sim  \cdot \nabla_t) \p_x  g_{\neq}) \rangle\vert   \\
    &\quad +\vert \langle A\tilde \rho, \langle\p_x\rangle^{-1}\Lambda_tA ( h_{i,\neq }  \p_z \langle\p_x\rangle  \Lambda_t^{-1}  \rho) \rangle\vert  \\
    &= NL_{f\to \rho,1}+NL_{f\to \rho,2}+NL_{g\to \rho,1}+NL_{g\to \rho,2}+NL_{h\to \rho}
\end{align*}
and estimate the terms separately. \underline{To bound $NL_{f\to \rho,1}$} we estimate 
\begin{align*}
    NL_{f\to \rho,1}&=\vert \langle A\tilde \rho,  \langle\p_x\rangle^{-1}\Lambda_t A( (\tilde \nabla^\perp_t \tilde \Lambda_t^{-2} f_{\neq}  \cdot \tilde \nabla_t) \Lambda_t^{-1}   \langle\p_x\rangle\rho)\rangle\vert \lesssim e^{-c\mu^{-\frac 13 }t }\Vert \Lambda_t A\tilde \rho\Vert_{L^2 } \Vert \sqrt{\tfrac {\p_t m } m }Af\Vert_{L^2 } \Vert A \rho\Vert_{L^2 } . 
\end{align*}
Integrating in time yields 
\begin{align}
    \int_0^T NL_{f\to \rho,1} \dd\tau 
    &\lesssim \mu^{-\frac 12 } \eps^2 (\mu^{-(1-\gamma)} \eps ) .\label{eq:frho1}
\end{align}
 \underline{To bound $NL_{f\to \rho,2}$} we estimate
\begin{align*}
     NL_{f\to \rho,2}&=\vert \langle A\tilde \rho,\langle\p_x\rangle^{-1} \Lambda_tA( (a_\sim  \cdot \nabla_t) \p_x  \Lambda_t^{-2}f_{\neq}) \rangle\vert \lesssim \Vert A\Lambda_t \tilde \rho \Vert_{L^2} \Vert A a_\sim \Vert_{L^2} \Vert \sqrt{\tfrac {\p_t m } m }A f\Vert_{L^2 }.
\end{align*}
Integrating in time yields 
\begin{align}
    \int_0^T NL_{f\to \rho,2} \dd\tau &\lesssim \mu^{-\frac 12 } \eps^2 (\mu^{-(1-\gamma)} \eps ) .\label{eq:frho2}
\end{align}
 \underline{To bound $NL_{g\to \rho,1}$}, we split 
\begin{align*}
    NL_{g\to \rho,1}&= \vert\langle A\tilde \rho, \langle\p_x\rangle^{-1}\Lambda_t A((\tilde \nabla^\perp g_{\neq }  \cdot \tilde \nabla) \Lambda_t^{-1}  \langle\p_x\rangle\rho_{\neq\neq })\rangle\vert + \vert\langle A\tilde \rho, \langle\p_x\rangle^{-1}\Lambda_t A(\p_xg_{\neq }   \p_y \vert \p_y \vert^{-1}  \rho_{= \neq  })\rangle\vert.
\end{align*}
We estimate 
\begin{align*}
    \vert \langle A\tilde \rho, \langle\p_x\rangle^{-1}\Lambda_t A((\tilde \nabla^\perp g_{\neq }  \cdot \tilde \nabla) \Lambda_t^{-1}  \langle\p_x\rangle\rho_{\neq\neq })\rangle\vert 
    %&\lesssim \Vert \Lambda_t A \tilde \rho \Vert_{L^2 }\left(\Vert \nabla \Lambda_t^{-1}A \rho_{\neq \neq } \Vert_{L^2 } \Vert g_{\neq } \Vert_{H^6}+\Vert \nabla \Lambda_t^{-1} \rho_{\neq \neq } \Vert_{H^5} \Vert  A \nabla   g_{\neq } \Vert_{L^2 }
    %\right)\\
    &\lesssim  e^{-c\mu^{\frac 13}t}\Vert \Lambda_t A \tilde \rho \Vert_{L^2 }\Vert \vert \tilde \nabla\vert  \Lambda_t^{-1}A \rho_{\neq \neq } \Vert_{L^2 }\Vert A\nabla g_{\neq } \Vert_{L^2}\\
    &\lesssim  e^{-c\mu^{\frac 13}t}\langle t \mu^{\frac 13 }\rangle \Vert \Lambda_t A \tilde \rho \Vert_{L^2 }\Vert \vert \tilde \nabla\vert  \Lambda_t^{-1}A \rho_{\neq \neq } \Vert_{L^2 }\Vert A^gg_{\neq } \Vert_{L^2}.
\end{align*}
Then by using 
$$\Vert \vert \tilde \nabla\vert  \Lambda_t^{-1}A \tilde \rho_{\neq \neq } \Vert_{L^2 }\le \left(\Vert A \tilde  \rho_{\neq \neq } \Vert_{L^2 }+t\Vert \sqrt{\tfrac {\p_t m }m }A \tilde \rho_{\neq \neq } \Vert_{L^2 }  \right)$$
we infer 
\begin{align}\begin{split}
    \vert \langle A\tilde \rho, &\langle\p_x\rangle^{-1}\Lambda_t A((\tilde \nabla^\perp g_{\neq }  \cdot \tilde \nabla) \Lambda_t^{-1}  \langle\p_x\rangle\rho_{\neq\neq })\rangle\vert \\
    &\lesssim  e^{-c\mu^{\frac 13}t}\langle t \mu^{\frac 13 }\rangle \Vert \Lambda_t A \tilde \rho \Vert_{L^2 }\left(\Vert A \rho_{\neq \neq } \Vert_{L^2 }+t\Vert \sqrt{\tfrac {\p_t m }m }A \rho_{\neq \neq } \Vert_{L^2 }  \right)\Vert A^gg_{\neq } \Vert_{L^2}.\label{eq:gtor11}
\end{split}\end{align}
For the term, including the average, we estimate 
\begin{align}\begin{split}
    \vert\langle A\tilde \rho, \langle\p_x\rangle^{-1}\Lambda_t A(\p_xg_{\neq }   \p_y \vert \p_y \vert^{-1}  \rho_{= \neq  })\rangle \vert&\le \Vert \Lambda_t A \tilde \rho \Vert_{L^2 }\Vert A \rho_{=\neq } \Vert_{L^2} \Vert  A\p_x g_{\neq } \Vert_{L^2 }\\
    &\le \Vert \Lambda_t A \tilde \rho \Vert_{L^2 }\Vert A \rho_{=\neq } \Vert_{L^2} \Vert  A^gg_{\neq } \Vert_{L^2 }.\label{eq:gtor12}
\end{split}\end{align}
Combining estimates \eqref{eq:gtor11} and \eqref{eq:gtor12} we infer  
\begin{align*}
    NL_{g\to \rho,1}&\lesssim e^{-c\mu^{\frac 13}t}\langle t \mu^{\frac 13 }\rangle \Vert \Lambda_t A \tilde  \rho \Vert_{L^2 }\left(\Vert A \rho_{\neq \neq } \Vert_{L^2 }+t\Vert \sqrt{\tfrac {\p_t m }m }A \rho_{\neq \neq } \Vert_{L^2 }  \right)\Vert A^gg_{\neq } \Vert_{L^2}\\
    &\quad +\Vert \Lambda_t A\tilde  \rho \Vert_{L^2 }\Vert A \rho_{=\neq } \Vert_{L^2} \Vert  A^gg_{\neq } \Vert_{L^2 }. 
\end{align*}
Integrating in time yields 
\begin{align}
    \int_0^T  NL_{g\to \rho,1}\dd\tau &\lesssim \mu^{-\frac 56 } \eps^3 \lesssim \delta \eps^2 .\label{eq:grho1}
\end{align}
\underline{To bound  $NL_{g\to \rho,2}$}, we use that $g$ is independent of $z$ to split  
\begin{align*}
( (a_\sim  \cdot \nabla_t) \p_x g_{\neq })_{\ndiamond }&=( (a_{\sim }^{x,y}  \cdot \tilde \nabla_t) \p_x g_{\neq })_{\ndiamond }\\
&=( a_{\sim }^x  \p_x^2 g_{\neq })_{\ndiamond } +( a_{\neq \neq }^y  \p_y^t \p_x g_{\neq })_{\ndiamond } +( a_{= \neq }^y  \p_y^t \p_x g_{\neq })_{\ndiamond }+\underbrace{( (a_{== }^x  \p_x) \p_x g_{\neq })_{\ndiamond }}_{=0}.
\end{align*}
We estimate all terms separately 
\begin{align*}
    \vert \langle A\tilde \rho, \langle\p_x\rangle^{-1}\Lambda_t A( a_\sim^x   \p_x^2 g) \rangle\vert &\lesssim \Vert A\Lambda_t\tilde  \rho \Vert_{L^2}\Vert A a^x_\sim \Vert_{L^2} \Vert A^g g_{\neq} \Vert_{L^2 },\\
\vert \langle A\tilde \rho, \langle\p_x\rangle^{-1}\Lambda_t A( a_{\neq \neq }^y   \p_y^t \p_x g_{\neq}) \rangle\vert &\lesssim e^{-c\mu^{\frac 13 } t} \Vert A\Lambda_t \tilde \rho \Vert_{L^2}\Vert A a^y_{\neq \neq } \Vert_{L^2} \Vert A \p_y^t g_{\neq} \Vert_{L^2 }\\
&\lesssim t \langle t \mu^{\frac13 } \rangle  e^{-c\mu^{\frac 13 } t} \Vert A\Lambda_t \tilde \rho \Vert_{L^2}\Vert A a^y_{\neq \neq } \Vert_{L^2} \Vert A^g g_{\neq} \Vert_{L^2 },\\
\vert \langle A\tilde \rho, \langle\p_x\rangle^{-1}\Lambda_t A( a_{= \neq }^y   \p_y^t \p_x g_{\neq}) \rangle\vert &\le  \Vert A\Lambda_t \tilde \rho \Vert_{L^2}\Vert A a_{=    }^y \Vert_{L^2} \Vert A^g\p_y^t g_{\neq} \Vert_{L^2 }.
\end{align*}
Integrating in time, using the bootstrap assumption and Lemma \ref{lem:Inviscid}, we estimate
\begin{align}
    \int_0^T NL_{g\to \rho,2}\dd\tau &= \mu^{-\frac 56}\eps^3+\mu^{-1}(\delta \mu)\eps^2\le \delta \eps^2 .\label{eq:grho2}
\end{align}
\underline{To bound $NL_{h\to \rho}$}, we estimate 
\begin{align*}
    NL_{h\to \rho}&=\vert \langle A\tilde \rho, \langle\p_x\rangle^{-1}\Lambda_t A ( h_{i,\neq }  \p_z \p_x  \Lambda_t^{-1}  \rho) \rangle\vert \lesssim \Vert A\Lambda_t\tilde  \rho \Vert_{L^2}\Vert A h_{\neq }\Vert_{L^2}\Vert A\p_z\Lambda_t^{-1} \rho \Vert_{L^2}\\
    &\lesssim \Vert A\Lambda_t \tilde \rho \Vert_{L^2}\Vert A h_{\neq }\Vert_{L^2}\Vert A \rho \Vert_{L^2}.
\end{align*}
Integrating in time yields 
\begin{align}
    \int_0^T   NL_{h\to \rho} \dd\tau &\lesssim \mu^{-\frac 23} \eps^3\le \delta \eps^2.\label{eq:hrho}
\end{align}
Combining the estimates \eqref{eq:frho1}, \eqref{eq:frho2}, \eqref{eq:grho1}, \eqref{eq:grho2} and \eqref{eq:hrho} we obtain 
\begin{align*}
    \int_0^T NL_{ \diamond\to \rho}\dd\tau \lesssim  \delta \eps^2. 
\end{align*}
\textbf{Bound on $LNL_\rho$}, the $LNL_\rho$ term can be estimated by the same steps as the $NL$ terms. \\
\textbf{Bound on $NLP_\rho$}, we bound the term 
\begin{align*}
    NLP_\rho&=\langle  A\tilde\rho_1, A\langle \p_x \rangle^{-1} \Lambda_t \p_y^t  \pi_\sim\rangle. 
\end{align*}
We show first 
\begin{align}
    NLP_\rho&\le 2NLP_{\sim,1} +NLP_{\sim,2}+2 NLP_{f}+ 2 NLP_{g}+ 2 NLP_{h}\label{eq:NLPspli}
\end{align}
with 
\begin{align*}
    NLP_{\sim,1}&= \sum_{i\in\{x,y,z\}}\vert \langle\langle \p_x \rangle^{-1}  RA\tilde \rho_1 ,A(\p_j^t v^y_\sim  \p_y^t v^j_\sim-\p_j^t b^y_\sim  \p_y^t b^j_\sim)\rangle \vert, \\
    NLP_{\sim,2}&= \sum_{i,j\in\{x,z\}}\vert \langle \langle \p_x \rangle^{-1} RA\tilde \rho_1 ,A(\p_j v^i_\sim  \p_i v^j_\sim-\p_j b^i_\sim  \p_i b^j_\sim)\rangle \vert, \\
    NLP_{f}&= \sum_{i,j\in\{x,y\}}\vert \langle \langle \p_x \rangle^{-1} RA\tilde \rho_1 ,A(\p_j^t v^i_\sim  \p_i^tv_{\neq =}^j )\rangle \vert, \\
    NLP_{g}&= \sum_{i,j\in\{x,y\}}\vert \langle\langle \p_x \rangle^{-1}  RA\tilde \rho_1 ,A(\p_j^t b^i_\sim  \p_i^tb_{\neq =}^j) \rangle \vert, \\
    NLP_{h}&= \sum_{i\in\{x,y\}}\vert \langle \langle \p_x \rangle^{-1} RA\tilde \rho_1 ,A(\p_z v^i_\sim  \p_i^th_{1,\neq}+\p_z b^i_\sim  \p_i^th_{2,\neq}) \rangle \vert,
\end{align*}
Where we denote $R=  \p_y^t\Lambda_t^{-1}$ which is a bounded operator $R:L^2 \to L^2$. For the pressure we have 
\begin{align*}
    \Delta_t \pi_\sim  &= \sum_{i,j\in\{x,y,z\}}(\p_j^t v^i \p_i^t v^j-\p_j^t b^i \p_i^t b^j)_\sim\\
    &= \sum_{i,j\in\{x,y,z\}}(\p_j^t v^i_\sim  \p_i^t v^j_\sim -\p_j^t b^i_\sim  \p_i^t b^j_\sim )_\sim\\
    &\quad +2 \sum_{i,j\in\{x,y,z\}}(\p_j^t v^i_\sim  \p_i^t v^j_{\neq =} -\p_j^t b^i_\sim  \p_i^t b^j_{\neq =}  )_\sim\\
    &\quad + \sum_{i,j\in\{x,y,z\}}(\p_j^t v^i_{\neq =}  \p_i^t v^j_{\neq =} -\p_j^t b^i_{\neq =}  \p_i^t b^j_{\neq =}  )_{==}\\
    &= P_1+P_2 +P_3. 
\end{align*}
So we have 
\begin{align*}
    NLP&\le \vert \langle  \langle \p_x \rangle^{-1} R A\tilde\rho_1, A(P_1+P_2+P_3 )\rangle\vert . 
\end{align*}
We split the term 
\begin{align*}
    \p_j^t v^i_\sim  \p_i^t v^j_\sim&= (2-\textbf{1}_{j=y}) \p_j v^y_\sim  \p_y^t v^j_\sim + \textbf{1}_{i,j\neq y } \p_j^t v^i_\sim  \p_i^t v^j_\sim
\end{align*}
and so 
\begin{align*}
    \langle  \langle \p_x \rangle^{-1} R A\tilde\rho_1, AP_1\rangle\vert &\le 2NLP_{\sim,1} +NLP_{\sim,2}. 
\end{align*}
For the $P_2$ we have  $\vert \langle  \langle \p_x \rangle^{-1} R A\tilde\rho_1, AP_2 \rangle \le  NLP_{f}+ NLP_{g}+ NLP_{h}$ by distinguishing between  $v_{\neq=}$ and $b_{\neq=}$ and $j=z$ and $j\neq z$. For $P_3$ we have 
\begin{align*}
    \langle  \langle \p_x \rangle^{-1} R A\tilde\rho_1, AP_1\rangle =\langle  \langle \p_x \rangle^{-1} R A\tilde\rho_{1,==}, AP_1\rangle=0. 
\end{align*}
Therefore, \eqref{eq:NLPspli} holds. In the following, we estimate the different terms. 

\underline{Bound on $NLP_{\sim,1}$,} by $v^y_\sim= \langle \p_x\rangle \Lambda_t^{-1}\rho_{1,\neq \neq}+ v^{y}_{=}  $ and $b^y_\sim= \langle \p_x\rangle \Lambda_t^{-1}\rho_{2,\neq \neq}+ v^{y}_{=}  $ we infer
\begin{align*}
    NLP_{\sim,1}&\lesssim \Vert AR\tilde \rho_{1} \Vert_{L^2} (\Vert A \rho_{\neq \neq } \Vert_{L^2}+\Vert A \nabla  a_{= }^{y,z}\Vert_{L^2}) \Vert A\nabla_t (v ,b ) \Vert_{L^2}.
\end{align*}
Therefore, after integrating in time 
\begin{align*}
    \int_0^T NLP_{\sim,1} \dd\tau &\lesssim \mu^{-\frac 23} \eps^3 + \mu^{-1} \delta \mu \eps^2 \lesssim\delta  \eps^2 
\end{align*}

\underline{Bound on $NLP_{\sim,2}$,} since  $\tilde \rho_i= \tilde \rho_{i,\neq \neq }+ \rho_{i,=\neq }= \tilde \rho_{i,\neq \neq }+ \Lambda a^{y}_{=}$ for some $a\in\{v,b\}$ we infer 
\begin{align*}
    NLP_{\sim,2}&\lesssim (\Vert  AR\tilde \rho_{\neq\neq} \Vert_{L^2}+\Vert A\nabla  a^{y}_{=} \Vert_{L^2}) \Vert A (v ,b )\Vert_{L^2} \Vert A\nabla_t (v ,b ) \Vert_{L^2}.
\end{align*}
Thus,  after integrating in time 
\begin{align*}
    \int_0^T NLP_{\sim,2} \dd\tau &\lesssim \mu^{-\frac 23} \eps^3 + \mu^{-1} \delta \mu \eps^2 \lesssim \delta  \eps^2. 
\end{align*}

\underline{Bound on $NLP_{f}$,}  we use $\Vert A\nabla_t v_{\neq =}\Vert_{L^2}\approx \Vert Af_{\neq}\Vert_{L^2}$ to infer 
\begin{align*}
    NLP_f & \lesssim \Vert AR\tilde  \rho_1\Vert_{L^2} \Vert A\nabla_t v\Vert_{L^2} \Vert Af_{\neq} \Vert_{L^2}. 
\end{align*}
Integrating in time yields 
\begin{align*}
    \int_0^T NLP_{f} \dd\tau &\lesssim \mu^{-\frac 23} (\mu^{-(1-\gamma)} \eps)\eps^2\lesssim \delta  \eps^2. 
\end{align*}

\underline{Bound on $NLP_{g}$}, we split this term 
\begin{align*}
     NLP_{g}&=\vert \langle \langle \p_x \rangle^{-1} RA\tilde \rho_1 ,A(\p_x  b^x_\sim  \p_x\p_y^t  g_{\neq} ) \rangle \vert  +\vert \langle  \langle \p_x \rangle^{-1} RA\tilde \rho_1 ,A(\p_x  b^y_\sim  (\p_y^t)^2  g_{\neq} ) \rangle \vert\\
     &\quad +\vert \langle \langle \p_x \rangle^{-1} RA\tilde \rho_1 ,A(\p_y^t  b^y_\sim  \p_x^2  g_{\neq} ) \rangle \vert +\vert \langle \langle \p_x \rangle^{-1} RA\tilde \rho_1 ,A(\p_y^t  b^y_\sim  \p_x\p_y^t   g_{\neq} ) \rangle \vert\\
     &= NLP_{g,1}+NLP_{g,2}+NLP_{g,3}+NLP_{g,4}.
     %&\quad +\vert \langle RA\tilde \rho_1 ,A(\p_y^t  b^y_\sim  \p_x\p_y^t   g ) \rangle \vert
\end{align*}
Using  $\p_x  b^x_\sim= \p_x  b^x_{\neq \neq }$ we estimate 
\begin{align*}
     NLP_{g,1}&\lesssim e^{-c\mu^{\frac 13 } t }\Vert  A\tilde \rho_{1} \Vert_{L^2} \Vert A  b^x_{\neq \neq} \Vert _{L^2}  \Vert A \p_y^t  g_{\neq} \Vert _{L^2},\\
     NLP_{g,2}&\lesssim e^{-c\mu^{\frac 13 } t }\Vert  A\tilde \rho_{1} \Vert_{L^2} \Vert A  b^y_{\neq \neq} \Vert _{L^2}  \Vert A(\p_y^t)^2  g_{\neq} \Vert _{L^2}\\
     &\lesssim \langle t \rangle \langle t \mu^{\frac 1 3 } \rangle  e^{-c\mu^{\frac 13 } t } \Vert  A\tilde \rho_{1} \Vert_{L^2} \Vert A  b^y_{\neq \neq} \Vert _{L^2}  \Vert A^g \p_y^t  g_{\neq} \Vert _{L^2}.
\end{align*}
We  estimate directly 
\begin{align*}
    NLP_{g,3}&\lesssim \Vert  A\tilde \rho_{1} \Vert_{L^2} \Vert A \p_y^t  b^y_{\neq \neq} \Vert _{L^2}  \Vert A^g   g_{\neq} \Vert _{L^2}.
\end{align*}
With $\p_y^t b^y_\sim= \langle \p_x\rangle \p_y^t \Lambda^{-1}_t \tilde \rho_{2,\neq \neq } + \p_y b^{y}_{=}$ we infer 
\begin{align*}
    NLP_{g,4}&\lesssim \Vert  A\tilde \rho_{1} \Vert_{L^2} (\Vert A \tilde \rho_{\neq \neq }\Vert _{L^2}+\Vert A \p_y (v^{y,z}_{=}, b^{y,z}_{=})\Vert _{L^2})  \Vert A^g\p_y^t   g_{\neq} \Vert _{L^2}.
\end{align*}
Combining these estimates, integrating in time, using the Bootstrap assumption and Lemma \ref{lem:Inviscid} gives 
\begin{align*}
    \int_0^T NLP_{g} \dd\tau &\lesssim \delta  \eps^2. 
\end{align*}

The Bound on $NLP_{h}$ is done the same as $NLP_{\sim,1}$ and $NLP_{\sim,2}$. Therefore, we conclude the pressure estimate
\begin{align*}
    \int_0^T NLP_\rho \dd\tau &\lesssim \delta  \eps^2. 
\end{align*}
With this estimate, all terms of Lemma \ref{lem:Invlem} are estimated, and we obtain Proposition \ref{prop:Invest}. 

\section{Nonlinear Growth of Vorticity}\label{sec:nlgrow}
In this and the following section, we use initial data independent of $z$. Furthermore, we consider vanishing $z$ component $v^z, b^z=0$. In this section we dropp the $\sim$ for the $2d$ derivatives, i.e. $\nabla =(\p_x,\p_y)$. Therefore, the $3d$ MHD reduces to the $2d$ equations
\begin{align}\begin{split}\label{eq:nlfg}
    \p_t f   &=\mu \Delta_t f+ (\nabla^\perp \Lambda_t^{-2} f\cdot \nabla) f + (\nabla^\perp g\cdot \nabla) \Delta_t  g,   \\
    \p_t g   &=\mu \Delta_t g+ (\nabla^\perp \Lambda_t^{-2} f\cdot \nabla) g .
\end{split}\end{align}
 Before proceeding to the initial data that yield the norm inflation, we prove the following lemma, which bounds the size of the average velocity: 
\begin{lemma}\label{lem:faver}
    Let $(f,g)$ be a solution to \eqref{eq:nlfg} satisfying, $\Vert \vert\p_y\vert^{-1} f_{=,in }\Vert_{H^{N-2}}<\infty $, $\Vert f\Vert_{H^N}\le \eps_f$ and $\Vert g\Vert_{H^{N}}\le \eps_g$ for $\eps_g, \ \eps_f >0$ then it holds 
    \begin{align*}
        \Vert \vert\p_y\vert^{-1} f_=\Vert_{H^{N-2}}(t)&\lesssim \Vert \vert\p_y\vert^{-1} f_{=,in }\Vert_{H^{N-2}} + \eps_f^2 +\mu^{-\frac 23 } \eps_g^2
    \end{align*}
    for all $t\le \mu^{-\frac 13 }$. 
\end{lemma}

\begin{proof}
This is a consequence of  
\begin{align*}
    \p_t \Vert \p_y\vert^{-1} f_= \Vert_{H^{N-2}}^2&\le \langle  \vert\p_y\vert^{-1} f_= ,  \vert \p_y\vert^{-1} ( (\nabla^\perp \Lambda_t^{-2} f \cdot \nabla) f+(\nabla^\perp g \cdot \nabla) \Delta_t  g  )_{=}\rangle_{H^{N-2}} \\
    &=\langle  \vert\p_y\vert^{-1} f_= ,   \vert \p_y\vert^{-1}  (\nabla^\perp_t ( (\nabla^\perp \Lambda_t^{-2} f \cdot \nabla) \nabla_t^\perp \Delta_t^{-1} f +(\nabla^\perp g \cdot \nabla) \nabla^\perp_t  g))_=\rangle_{H^{N-2}} \\
    &=\langle  \vert\p_y\vert^{-1} f_= ,   \vert \p_y\vert^{-1}  \p_y ( (\nabla^\perp \Lambda_t^{-2} f \cdot \nabla) \p_y^t \Delta_t^{-1} f +(\nabla^\perp g \cdot \nabla) \p_y^t  g)_=\rangle _{H^{N-2}}\\
    &\lesssim \Vert \vert  \p_y\vert^{-1} f_=\Vert_{H^{N-2}} \left( \langle t \rangle^{-2} \Vert f_{\neq} \Vert_{H^N }^2 +\langle t \rangle \Vert g_{\neq} \Vert_{H^N}^2 \right) 
\end{align*}
by bootstrap and integrating in time, we obtain the Lemma. 
\end{proof}
To prove the norm inflation we use the following initial data 
\begin{align}\label{eq:fgin}\begin{split}
    f_{in}(x,y)&=0, \\
    g_{in}(x,y) &= c_N \delta\mu^{\gamma } \left( 2\cos(x)\int_{\vert \eta-10 \vert \le 2 } \cos (y \eta ) \dd \eta   + \cos(2x)\int_{\vert \eta \vert \le 1 } \cos (y \eta ) \dd \eta\right),
\end{split}\end{align}
where $c_N>0$ is a constant chosen such that 
\begin{align*}
    \Vert g_{in}\Vert_{H^{N+1}} &=\delta \mu^{\gamma}.
\end{align*}

\begin{prop}\label{prop:nlgrowth2} Let $N\ge7$ and $\frac 12 < \gamma\le 1$. Assume that for $\delta,\mu_0>0$ small enough and $\frac {\gamma} 2+\frac 14  \le\tilde \gamma \le \gamma$, the corresponding solution $(f,g)$ of \eqref{MHD2fg} the initial data \eqref{eq:fgin} satisfies
\begin{align*}
    \Vert g\Vert_{L^\infty_t H^{N+1}}+ \mu^{\frac 1 6} \Vert f \Vert_{L^\infty_t H^N }&\le C \delta \mu^{\tilde \gamma}
\end{align*}
uniformly in $0<\mu\le \mu_0$. Then the corresponding solution $f,g$ of \eqref{eq:nlfg} satisfies the nonlinear growth estimate
\begin{align}
     \Vert f\Vert_{L^2}(T)   &\gtrsim \delta^2 \mu^{2\gamma-1}\label{eq:lowf}
\end{align}
for $T=\frac 1 {10} \mu^{-\frac 13 }$. Furthermore, we obtain the lower bound
\begin{align}
    \Vert g\Vert_{H^1}(t)\gtrsim \eps \label{eq:lowglin}
\end{align}
for times $t\le \frac 1 {10}\mu^{-\frac 1 3}$.
\end{prop}

%The nonlinear growth in the $z$-average of Theorem~\ref{thm:main} and Corollary \ref{cor:opt} are a direct consequence of this proposition. 
Proposition \ref{prop:nlgrowth} is a direct consequence. 
\begin{proof}[Proof of Theorem~\ref{thm:below} and Corollary~\ref{cor:opt}]
It is sufficient to prove Theorem~\ref{thm:below} for  $\frac 5 6 \ge \gamma\ge\frac 23$. Let the stability assumption of Theorem~\ref{thm:below} hold. If $\beta \ge \frac 1 6$ we are done. Let $\beta \le \frac 1 6$, then since $\frac 5 6 \ge \gamma\ge\frac 23$ we have the assumption of Proposition \ref{prop:nlgrowth2} for  $\tilde \gamma =\gamma-\beta $ and obtain for the initial data \eqref{eq:fgin} the lower bound \eqref{eq:lowf}. Therefore, in the limit $\mu\to 0$, it holds that $\tilde \gamma-\frac 1 6 \le 2\gamma-1 $ and so $\beta \ge \frac 5 6 -\gamma$.
\end{proof}

\textbf{Proof of Proposition \ref{prop:nlgrowth2}:} It is sufficient to prove Proposition \ref{prop:nlgrowth2} for $N=7$ and $\delta$ small enough and we focous on proving \eqref{eq:lowf}. We write $\eps=\delta \mu^{\tilde \gamma} $. By assumption, we have the 
\begin{align}\label{eq:growthfgbound}
    \Vert g\Vert_{L^\infty_t H^N}+\mu^{\frac 16   } \Vert f\Vert_{L^\infty_t H^N}\le C   \delta \mu^{\tilde \gamma}.
\end{align} We define the functions $f_1$ and $g_1$ determined through 
\begin{align*}
    g_1&= e^{\mu \int_0^t\Delta_\tau \dd\tau } g_{in}, \\
    \p_t f_1&= \mu \Delta_t f_1 + (\nabla^\perp g_1 \cdot \nabla) \Delta_t g_1, \\
    f_{1,in}&=0. 
\end{align*}
Using these functions and times  $t\le\frac 1 {10} \mu^{-\frac 1 3 }$, we prove Proposition \ref{prop:nlgrowth2} in two steps: First we prove
\begin{align}
    \Vert f_1\Vert_{L^2} \gtrsim \delta^2 t^3 \mu^{2\gamma  }\label{eq:f1low}
\end{align}
and then we show
\begin{align}
    \Vert f-f_1\Vert_{L^2}  \lesssim\mu^{-\frac 7 6 }\eps^3+\mu^{-2 }  \eps^4  .\label{eq:f1diff}
\end{align}
From \eqref{eq:f1low} and \eqref{eq:f1diff} we infer that there exists a constant $C>0$ such that
\begin{align*}
    \Vert f\Vert_{L^2}  \gtrsim  \delta^2 t^3 \mu^{2\gamma}- C  (\mu^{-\frac 76  }\eps^3+\mu^{-2}  \eps^4 ).
\end{align*}
Since by the assumption $\tilde \gamma \ge \frac 23 \gamma +\frac 1{18} $ we obtain $$\mu^{-\frac 76} \eps^3=\delta^3 \mu^{2\gamma-1}\mu^{3\tilde \gamma -2\gamma -\frac 16}\le \delta^3 \mu^{2\gamma-1} $$
and by  $\tilde \gamma \ge \frac 12 \gamma +\frac 14 $ 
$$\mu^{-2 } \eps^4=\delta^4 \mu^{2\gamma-1}\mu^{4\tilde \gamma -2\gamma -1}\le \delta^3 \mu^{2\gamma-1}. $$
Thus for  for  $T=\frac 1 {10} \mu^{-\frac 13 }$ and $\mu\le \mu_0$ for $\mu_0$ small enough we obtain 
\begin{align*}
    \Vert f\Vert_{L^2}(T)  \gtrsim   \delta^2 \mu^{2\gamma-1}(1-C\delta)
\end{align*}
for $\delta$ small enough we infer Proposition \ref{prop:nlgrowth2}. So it is left to prove \eqref{eq:f1low} and \eqref{eq:f1diff}. We start with the proof of \eqref{eq:f1low}. By taking a Fourier transform we infer 
\begin{align*}
    \hat g_{in}(k,\eta ) &=\frac {c_N}{(2\pi)^{\frac 12 }} \delta \mu^{\gamma } \begin{cases}
        \textbf{1}_{\vert \eta \pm 10\vert \le 2 },& k =\pm 1, \\
         \textbf{1}_{\vert \eta \vert \le 1 },& k =\pm 2 ,
        \end{cases}
\end{align*}
and so we obtain 
\begin{align}
    \hat g_1(k,\eta ) &=\frac {c_N}{(2\pi)^{\frac 12 }}  \delta \mu^{\gamma } \begin{cases}
          \textbf{1}_{\vert \eta \pm 10\vert \le 2 }\exp(-\mu k^2 \int_0^t  \langle s-\frac \eta k\rangle ^2 \dd s),& k =\pm 1, \\
       \textbf{1}_{\vert \eta \vert \le 1 }\exp(-\mu k^2 \int_0^t  \langle s-\frac \eta k\rangle ^2 \dd s),& k =\pm 2.
        \end{cases}\label{eq:gheat}
\end{align}
Therefore, since $f_1(t,x,y) = \int_0^t \exp\left(\mu \int_\tau^t\Delta_s ds  \right) ((\nabla^\perp g_1 \cdot  \nabla)\Delta_\tau   g_1)\dd\tau  $ we obtain after a Fourier transform 
\begin{align*}
     \hat f_1(t,k,\eta ) &=\int^t_0\dd\tau \exp\left(-\mu \int_\tau^t(k^2+(\eta-ks)^2) \dd s  \right)  \calN[g_1,g_1](k,\eta).\\
     \calN[g_1,g_1](k,\eta):&=\sum_l \int (\eta l - k\xi )  g_1(k-l,\eta-\xi)  (l^2 +(\xi-l\tau )^2  )g_1(l,\xi) \  \dd\xi 
\end{align*}
We show a lower bound on $f_1(k)$ for $k=3$, $\eta\ge 0$ and $\vert \eta -10\vert \le 1 $. The $g_1(l)$ is supported on $l\in \{\pm1 ,\pm 2\}$ and therefore $g_1(3-l,\eta-\xi)  g_1(l,\xi)\neq 0 $ iff $ (3-l,l)\in\{(2,1),(1,2)\}$. Thus we obtain  
\begin{align*}
    \calN[g_1,g_1](3,\eta)&=\sum_l \int   (\eta l - 3\xi ) g_1(3-l,\eta-\xi)  (l^2 +(\xi-lt)^2  )g_1(l,\xi)\dd\xi \\
    &=\int  (\eta  - 3\xi ) g_1(2,\eta-\xi)  (1 +(\xi-t)^2  )g_1(1,\xi)  \dd\xi\\
    &\quad+\int (2\eta - 3\xi ) g_1(1,\eta-\xi)  (2^2 +(\xi-2t)^2  )g_1(2,\xi)  \dd\xi.
\end{align*}
By a change of variables $\tilde \xi = \eta -\xi $
\begin{align*}
    \int(\eta  - 3\xi ) g_1(2,\eta-\xi)  (1 +(\xi-t)^2  )g_1(1,\xi)\dd\xi&= \int (2\eta  - 3\tilde \xi ) g_1(2,\tilde \xi)  (1 +(\eta-\tilde \xi-t)^2  )g_1(1,\eta-\tilde \xi)\dd\tilde \xi.
\end{align*}
Therefore,
\begin{align*}
    F[g_1,g_1](3,\eta)
    &=\int  (2\eta - 3\xi ) [  (2^2 +(\xi-2t)^2  )+ (1 +(\eta-\xi-t)^2  )]     g_1(1,\eta-\xi)g_1(2,\xi)\dd\xi.
\end{align*}
For $\eta\ge 0$ with $\vert \eta -10\vert \le 1 $ and $\vert \xi\vert \le 1 $  we obtain 
\begin{align}\label{eq:lowt2}
    \langle t\rangle ^2 \lesssim  (2 \eta  - 3\xi )\left[    (1 +(\eta -\xi-t)^2  )+ (2^2 +(\xi-2t)^2  )\right],
\end{align}
for some $C>0$. For small enough $\mu_0>0$ such that $0<\mu\le \mu_0$, for times  $ \tau \le \frac 1{10} \mu^{-\frac 13 } $ with $0<\mu\le \mu_0$, it holds that 
\begin{align}\begin{split}\label{eq:lowdiss}
    \frac 12 \le \exp\left(-\mu \int_0^\tau (k^2 +(\xi- \tau_1 k )^2  ) \dd\tau_1\right)\le 1,&\qquad \text{ for } \vert k\vert \le 3, \vert \xi\vert \le 100 .
\end{split}\end{align}
Using that and the support of $g_1$ we obtain 
\begin{align}\label{eq:lowg}
    g_1(2,\eta-\xi)g_1(1,\xi)\ge \frac {c_N^2}{2\pi} \textbf{1}_{\vert \eta-10\vert \le 1} \textbf{1}_{\vert \xi\vert \le 1}  \delta^2 \mu^{2\gamma}.
\end{align}
Combining \eqref{eq:lowt2} and \eqref{eq:lowg} we infer for $\eta \ge 0$ and $\vert \eta-10\vert \le 1$, that 
\begin{align*}
    F[g_1,g_1](3,\eta)&\gtrsim   t^2   \delta^2 \mu^{2\gamma  }.
\end{align*}
Therefore, for $f_1$  by using this estimate and again \eqref{eq:lowdiss} we infer for $\eta \ge 0$ with $\vert \eta-10\vert \le 1$
\begin{align*}
     \hat f_1(3,\eta ) 
     &=\int^t_0 \exp\left(-\mu k^2 \int_\tau ^t  \langle s-\tfrac \eta k\rangle ^2 \dd s\right)F[g_1,g_1](3,\eta) \dd\tau\gtrsim   \delta^2 \mu^{2\gamma  } \int_0^t  \tau^2  \dd\tau\gtrsim   \delta^2 \mu^{2\gamma }t^3 ,
\end{align*}
which yields \eqref{eq:f1low}. To prove \eqref{eq:f1diff}, we fist apply Lemma \ref{lem:faver} to infer 
\begin{align}
    \Vert \p_y\vert^{-1} f_= \Vert_{L^\infty_T H ^5}^2 \lesssim \delta \mu^{-\frac 23 }\eps^2.\label{eq:avest}
\end{align}
We denote 
\begin{align*}
    f_2=f-f_1,\qquad \qquad 
    g_2=g-g_1.
\end{align*}
Then $g_2$ solves the equations 
\begin{align*}
    \p_t g_2 &= \mu \Delta_t g_2+ ((\nabla^\perp \Lambda_t^{-2} f\cdot \nabla) g),\\
    g_{2,in}&=0.
\end{align*}

First we bound $g_2$, we have 
\begin{align}\begin{split}    
    \Vert g_2\Vert_{L^\infty_T H^4} &\le \int_0^T \Vert  (\nabla^\perp \Lambda_\tau^{-2} f\cdot \nabla) g\Vert_{H^4} d\tau \\
    & \lesssim \int_0^T\Vert g\Vert_{H^5}(\langle \tau \rangle ^{-2}\Vert f\Vert_{H^7}+\Vert \vert \p_y\vert^{-1}f_=\Vert_{H^4})d\tau \\
    &\lesssim  C\mu^{-\frac 16 } \eps^2 (1 + \mu^{-\frac 56}\eps).  \label{eq:estg2}
\end{split}
\end{align}

Now we estimate $f_2$, which satisfies the equation 
\begin{align*}
    \p_t f_2   
    &=\mu \Delta_t f_2+  (\nabla^\perp g\cdot \nabla) \Delta_t g-(\nabla^\perp g_1\cdot \nabla)  \Delta_t g_1+(\nabla^\perp \Lambda_t^{-2}  f\cdot \nabla)   f,\\
    f_{2,in}&=0.
\end{align*}
 We split 
\begin{align*}
    (\nabla^\perp g\cdot \nabla) \Delta_t g-(\nabla^\perp g_1\cdot \nabla)  \Delta_t g_1&= (\nabla^\perp g_1\cdot \nabla) \Delta_t g_2+(\nabla^\perp g_2\cdot \nabla) \Delta_t g_1+(\nabla^\perp g_2\cdot \nabla) \Delta_t g_2
\end{align*}
and so we obtain
\begin{align*}
     \Vert f_2\Vert_{L^\infty_T L^2 }&\lesssim    \int_0^T  \Vert (\nabla^\perp \Lambda_t^{-2}  f\cdot \nabla )  f+ (\nabla^\perp g\cdot \nabla) \Delta_t g-(\nabla^\perp g_1\cdot \nabla)  \Delta_t g_1\Vert_{L^2} d\tau  \\
   &\lesssim   \int_0^T   \langle \tau\rangle^{-2}\Vert f\Vert_{H^3}^2+\Vert f\Vert_{H^2}\Vert \vert \p_y\vert^{-1}f_=\Vert_{H^2} + \Vert g_1\Vert_{H^2}  \Vert \Delta_t  g_2\Vert_{H^2} \\
   &\qquad \qquad +\Vert g_2\Vert_{H^2}  \Vert \Delta_t  g_1\Vert_{H^2}+ \Vert g_2\Vert_{H^2}  \Vert \Delta_t  g_2\Vert_{H^2} 
    d\tau \\
   &\lesssim  \int^T_0  \langle \tau \rangle^{-2}\Vert f\Vert_{H^2}^2+\Vert f\Vert_{H^2}\Vert \vert \p_y\vert^{-1}f_=\Vert_{H^2} + \langle \tau \rangle^2(\Vert g\Vert_{H^4}+\Vert  g_2\Vert_{H^4})\Vert  g_2\Vert_{H^4}
    d\tau  \\
    &\lesssim  \left(\Vert f\Vert_{L^\infty_T H^2}^2+t\Vert f\Vert_{L^\infty_T H^2} \Vert \vert \p_y\vert^{-1} f_=\Vert_{L^\infty_T H^2} + t^3 (\Vert g\Vert_{L^\infty_T H^4}+\Vert  g_2\Vert_{L^\infty_T H^4})\Vert  g_2\Vert_{L^\infty_T H^4}
   \right )\\
   &\lesssim  \eps^2+\mu^{-1}\eps^3+\mu^{-\frac 76 } \eps^3 (1 + \mu^{-\frac 56  }\eps) \\
   &\lesssim  \mu^{-\frac 7 6 }\eps^3+\mu^{-2 }  \eps^4. 
\end{align*}

   This yields \eqref{eq:f1diff} and therefore we obtain the nonlinear transient growth \eqref{eq:lowf}.  The linear transient growth \eqref{eq:lowglin} is a consequence of \eqref{eq:gheat} and  \eqref{eq:estg2}. Therefore, we conclude Proposition \ref{prop:nlgrowth2}.

With Proposition \ref{prop:nlgrowth2}, we have proved all necessary propositions to conclude the Theorems~\ref{thm:main} and \ref{thm:below}.

\section{Nonlinear Growth of Magnetic Potential}\label{sec:nlgrow2}
In this section we prove Theorem~\ref{thm:opt} by showing nonlinear growth of the magnetic potential. As in the previous section, we consider solutions of \eqref{eq:nlfg} and omit writing $\sim$ for derivatives. We define the initial data 
\begin{align}
    f_{in}(x,y)&= c_N \delta \mu^{\gamma-\beta_1 } \frac {1}{\langle k_0,\eta_0\rangle^N}\left(\cos(k_0x) \int_{\vert \eta-\eta_0\vert \le 1}  \cos(\eta y) \dd y +\sin(k_0x) \int_{\vert \eta-\eta_0\vert \le 1}  \sin(\eta y) \dd y\right),\label{eq:fin}\\
    g_{in}(x,y)&= -c_N  \delta \mu^{\gamma } \cos(x) \int_{\vert \eta \vert\le 1} \cos (\eta  y ) \dd \eta\label{eq:gin}.
\end{align}
The Fourier transform of the initial data satisfies 
\begin{align}
    \hat f_{in}(k,\eta)&=  \frac{c_N} {2(2\pi)^{\frac 1 2 }}\frac {\delta \mu^{\gamma-\beta_1}}{\langle k_0,\eta_0\rangle^N}\left(\textbf{1}_{k=k_0}\textbf{1}_{\vert \eta-\eta_0\vert\le 1 }+ \textbf{1}_{k=-k_0}\textbf{1}_{\vert \eta+\eta_0\vert\le 1 } \right), \label{eq:Ffin}\\
    \hat g_{in}(k,\eta) &= -\frac{c_N} {2(2\pi)^{\frac 1 2 }} \delta \mu^{\gamma }\textbf{1}_{k=\pm 1 } \textbf{1}_{\vert \eta\vert\le 1 }. \label{eq:Fgin}
\end{align}
Here the constant $c_N>0$ is chosen such that 
\begin{align*}
    \tfrac 12 \delta \mu^\gamma\le \mu^{\beta_1}\Vert f_{in}\Vert_{H^N}+ \Vert g_{in}\Vert_{H^{N+1}}\le \delta \mu^\gamma .
\end{align*}

\begin{prop}\label{prop:maggr}
Let  $N\ge 7$, $\gamma\ge \frac34$  and $ \beta_1 \ge 1-\gamma $ such that $\gamma \ge \beta_1+\frac 12  $. Assume that for $\delta,\mu_0>0$ small enough and $0<\mu\le \mu_0$ the initial data \eqref{eq:fin} and \eqref{eq:gin} with $k_0=10$ and $\eta_0= c\mu^{-\frac 13 }$ for a small constant $c>0$ it holds that 
    \begin{align}
        \mu^{\beta_1}\left \Vert f \right\Vert_{H^N}(t)+\left \Vert g \right\Vert_{H^{N+1}}(t)+\mu^{\frac 12 }\left \Vert \nabla_\tau g \right\Vert_{L^2_t H^{N+1}}\le \delta\mu^{\gamma}, \label{eq:FG}
    \end{align}
    for times $t\le \mu^{-\frac 1 3 }$. Then there exists a time $0\le t_1\le \mu^{-\frac 1 3 }$ such that 
    \begin{align}
        \left \Vert g \right\Vert_{H^{N+1}}(t_1) \gtrsim_\delta   \mu^{2 \gamma-\frac 23-\beta_1 }.\label{nlg}
    \end{align}

\end{prop}

Note that \eqref{nlg} contradicts \eqref{eq:FG} for $\gamma -\beta_1 -\frac 23 <0$ in the limit $\mu \downarrow0$, which we use to prove Theorem~\ref{thm:opt}:
\begin{proof}[Proof of Theorem~\ref{thm:opt}]
     It is sufficient to consider $\frac 34 \le \gamma\le \frac 5 6$. By Proposition \ref{prop:nlgrowth2} we obtain $\beta_1\ge \gamma-1$. By Proposition \ref{prop:maggr} we have a contradiction in the limit $\mu\downarrow 0 $  if $\gamma -\frac 23 -\beta_1=2(\gamma -\frac 56) - (\beta_1- \gamma+1)< 0$ which implies $\gamma =\frac 5 6$ and $\beta_1=\frac 1 6$. 
\end{proof}

\begin{proof}[Proof of Proposition \ref{prop:maggr}]
The proposition is clear if $\mu^{2\gamma-\frac 23 -\beta_1}\le 2 \delta^2 \mu^\gamma$ or $\delta^{6} \le\mu$. Therefore, we consider the case $\mu^{2\gamma-\frac 23 -\beta_1}\ge 2 \delta^2 \mu^\gamma$ and  $\delta^{6} \le\mu$. To prove the Proposition, we split the solution into
\begin{align*}
    f &= f_0+ f_2, \\
    g&= g_0+g_1 +g_2.
\end{align*}
We denote  $f_0=e^{\int_0^t\Delta_\tau \dd\tau }f_{in}$, $g_0=e^{\int_0^t\Delta_\tau \dd\tau }g_{in}$, 
\begin{align*}
    \p_t g_1&=\mu\Delta_t g_1+ (\nabla^\perp\Lambda_t^{-2} f_0\cdot\nabla) g_0,\\
    g_{1,in}&=0,
\end{align*}
and 
\begin{align*}
    \p_t f_2&=\mu\Delta_t f_2 + (\nabla^\perp\Lambda_t^{-2} f\cdot \nabla) f+ (\nabla^\perp g\cdot \nabla) \Delta_t g,  \\
    \p_t g_2&=\mu\Delta_t g_2+ (\nabla^\perp\Lambda_t^{-2} f\cdot \nabla) g- (\nabla^\perp\Lambda_t^{-2} f_0\cdot \nabla) g_0, \\
    f_{2,in}&=g_{2,in}=0.
\end{align*}
We prove that for times $\frac c{10}\mu^{-\frac 13 }\le t\le c\mu^{-\frac 13 }=\eta_0 $ the following bounds 
\begin{align}
    \Vert g_1|_{k=k_0-1}\Vert_{H^{N+1}} (t) &\gtrsim\delta^2  \mu^{2\gamma-\frac 23 -\beta_1},\label{eq:g1est}
\end{align}
and 
\begin{align}
    \Vert g_2\Vert_{H^{N+1}} (t) &\lesssim \delta^3  \mu^{2\gamma-\frac 23 -\beta_1}.\label{eq:g2}
\end{align}
From that, we infer 
\begin{align*}
    \Vert g\Vert_{H^{N+1}}\ge\Vert g|_{k=k_0-1}\Vert_{H^{N+1}}&\ge \Vert g_1|_{k=k_0-1}\Vert_{H^{N+1}} -\underbrace{\Vert g_0|_{k=k_0-1}\Vert_{H^{N+1}}}_{=0} -\Vert g_2\Vert_{H^{N+1}} \gtrsim \delta^2  \mu^{2\gamma-\frac 23 -\beta_1}
\end{align*}
for $\delta>0$ small enough, which yields the proposition. Therefore, it is left to prove \eqref{eq:g1est} and \eqref{eq:g2}. To prove \eqref{eq:g1est} we use that 
\begin{align*}
    \hat g_1(t,k,\eta) &=\int_0^t e^{\int_\tau ^t\Delta_{\tilde \tau}\dd \tilde \tau}\calF[(\nabla^\perp\Lambda_t^{-2} f_0\cdot\nabla) g_0](k,\eta) \dd\tau \\
    &= \int_0^t e^{\int_\tau^t\Delta_{\tilde \tau}\dd \tilde\tau}\calF [ (\nabla^\perp\Lambda_t^{-2} e^{\int_0^\tau\Delta_{\tilde \tau}\dd \tilde \tau}f_{in}\cdot\nabla) e^{\int_0^\tau \Delta_{\tilde \tau} \dd\tilde \tau}g_{in}](k,\eta) \dd\tau \\
    &=- \int_0^t \dd\tau \sum_{l}  \int \dd\xi\frac {\eta l -k\xi} {\vert k-l,\eta-\xi-(k-l)t \vert^2 }g_{in} (l,\xi) f_{in}(k-l,\eta-\xi) \\
    &\qquad \qquad \cdot \exp\left(-\mu \int_\tau ^t \vert k,\eta-  k\tilde\tau \vert^2  \dd\tilde\tau -\mu \int_0^\tau  \vert k-l,\eta-\xi- (k-l)\tilde\tau \vert^2+\vert l,\xi- l\tilde\tau \vert^2   \dd\tilde\tau
    \right) .
\end{align*}
When we consider the frequency $k=k_0-1$, then $f_{in}(k-l,\eta-\xi)g_{in}(l,\xi)\neq 0$ only if $k-l=k_0$ and $l=-1$. We insert $f_{in}$ and $g_{in}$ to infer 
\begin{align*}
    \hat g_1(t,k_0-1,\eta) &=\frac 1 {8\pi } \frac {\delta^2 \mu^{2\gamma-\beta_1}}{\langle k_0,\eta_0\rangle^N} \int_0^t \dd\tau  \int \textbf{1}_{\vert \xi\vert \le 1}\textbf{1}_{\vert \eta-\xi-\eta_0 \vert \le 1} \dd\xi\frac {\eta  -(k_0-1)\xi} {\vert k_0 ,\eta-\xi-k_0 t \vert^2 }a(t,\tau,\eta,\xi), \\
    a(t,\tau,\eta,\xi)&=\exp\left(-\mu \int_\tau ^t \vert k_0-1,\eta-  (k_0-1)\tilde\tau \vert^2  \dd\tilde\tau -\mu \int_0^\tau  \vert k_0,\eta-\xi- k_0\tilde\tau \vert^2+\vert 1,\xi+\tilde\tau \vert^2   \dd\tilde\tau
    \right)  .
\end{align*}
Since $\eta_0,t\le c \mu^{-\frac 1 3 }$ we have for $\xi, \ \eta$ such that  $\vert \xi\vert \le 1$, $\vert \eta-\xi-\eta_0 \vert \le 1$ the estimate 
\begin{align*}
    \mu \int_\tau ^t \vert k_0-1,\eta-  (k_0-1)\tilde\tau \vert^2  \dd\tilde\tau +\mu \int_0^\tau  \vert k_0,\eta-\xi- k_0\tilde\tau \vert^2+\vert 1,\xi+\tilde\tau \vert^2   \dd\tilde\tau\le c
\end{align*}
and so 
\begin{align*}
    1/2\le a(t,\tau,k_0,\eta,\xi) \le 1.
\end{align*}
Furthermore, on $\vert \eta-\xi -\eta_0 \vert \le 1$ and $\vert \xi \vert\le 1 $ we have the lower bound 
\begin{align*}
    \frac {\eta  -(k_0-1)\xi} {\vert k_0 ,\eta-\xi-k_0 t \vert^2 }\ge \frac 12 \frac {\eta_0}{k_0^2} \langle t-\tfrac {\eta-\xi}{k_0}\rangle^{-2} \gtrsim \mu^{-\frac 13 } \langle t-\tfrac {\eta_0}{k_0}\rangle^{-2}
\end{align*}
which yields 
\begin{align*}
    \hat g_1(t,k_0-1,\eta) &\gtrsim \mu^{-\frac 13 } \frac {\delta^2 \mu^{2\gamma-\beta_1}}{\langle k_0,\eta_0\rangle^N} \int^t_0  \langle \tau -\frac {\eta_0}{k_0}\rangle^{-2}  \dd\tau \underbrace{ \int \textbf{1}_{\vert \xi\vert \le 1} \textbf{1}_{\vert \eta-\xi -\eta_0 \vert \le 1}\dd\xi}_{\ge \textbf{1}_{\vert \eta-\eta_0\vert \ge 1 }}.
\end{align*}
For times $t\ge \frac {\eta_0}{k_0}=\frac c {10}\mu^{-\frac1 3 }$ it holds $\int^t_0  \langle \tau -\frac {\eta_0}{k_0}\rangle^{-2}  \dd\tau\gtrsim  1$ and so 
\begin{align*}
    \hat g_1(t,k_0-1,\eta) 
      &\gtrsim \mu^{-\frac 13 } \frac {\delta^2 \mu^{2\gamma-\beta_1}}{\langle k_0,\eta_0\rangle^N}  \textbf{1}_{\vert \eta -\eta_0 \vert \le 1}.
\end{align*}
Taking the Sobolev norm and using that $\eta_0\approx  \mu^{-\frac 13 }$ yields
\begin{align*}
    \Vert  g_1\Vert_{H^{N+1}}\gtrsim \delta^2 \mu^{2\gamma-\beta_1-\frac 23 }.
\end{align*}
To prove  \eqref{eq:g2} we show by bootstrap
\begin{align}
    \Vert  m^{-1}f_2 \Vert_{L^\infty_t H^N}^2+ \int_0^t \mu \Vert m^{-1}\nabla_\tau  f_2 \Vert_{H^N}^2+ \Vert m^{-1} \sqrt {\tfrac{\p_t m}m}    f_2 \Vert_{H^N}^2 \dd\tau &\le  (C_1 \delta^2 \mu^{\gamma-\beta_1 })^2,\label{eq:f2est2}\\
    \Vert m^{-1} g_2 \Vert_{L^\infty_t H^{N+1}}^2+ \int_0^t \mu \Vert m^{-1} \nabla_\tau  g_2 \Vert_{H^{N+1}}^2+ \Vert m^{-1} \sqrt {\tfrac{\p_t m}m}    g_2 \Vert_{H^{N+1}}^2 \dd\tau &\le (C_2
    \delta^3  \mu^{2\gamma-\frac 23 -\beta_1})^2,\label{eq:g2est}
\end{align}
for times $t\le \mu^{-\frac 13 }$ and $C_2>C_1>0$ to be chosen later. We use Lemma \ref{lem:faver} to obtain $ \Vert \vert \p_y\vert^{-1} f_=\Vert_{L^\infty_t L^2} \lesssim \delta \mu^{2(\gamma-\beta_1)}+\delta \mu^{2(\gamma-\frac13 )}$ and so since $\gamma \ge \max( \beta_1, \frac 1 3 )$ we infer by \eqref{eq:FG} that
\begin{align}
    \Vert \vert \p_y\vert^{-1} f_=\Vert_{L^\infty_t  H^{N+1}} &\lesssim \delta \mu^{\gamma-\beta_1}.\label{eq:f=g}
\end{align}
In the following estimates, we write $\lesssim$ for a constant which is also independent of the constants $C_1$ and $C_2$. For sake of contratdiction we assume that $T<\mu^{-\frac 1 3}$ is the maximal time such that \eqref{eq:f2est2} and \eqref{eq:g2est} hold. \textbf{Bound on $f_2$},  we split
\begin{align*}
    \frac 12  \Vert m^{-1}  f_2\Vert_{H^N}^2(T)&+\int_0^{T} \mu \Vert  m^{-1}\nabla_\tau  f_2\Vert_{H^N}^2+ \Vert  m^{-1}\sqrt{\tfrac {\p_t m}m } f_2\Vert_{H^N}^2\dd \tau\le \int_0^{T} 
    E_1^f+E_2^f+E_3^f \dd\tau,  \\
    E_1^f&=\vert \langle   m^{-2}f_2,   ((\nabla^\perp_t\Lambda_t^{-2} f_{\neq}\cdot \nabla_t) f\rangle_{H^N}\vert,\\
    E_2^f&=\vert \langle   m^{-2}f_2,   ((\p_y\p_y^{-2} f_=\p_x  f\rangle_{H^N}\vert,\\
    E_3^f&=\vert \langle   m^{-2}f_2,    (\nabla^\perp g\cdot \nabla) \Delta_t) g\rangle_{H^N}\vert,
\end{align*}
which we bound separately. For $E_1^f$ we have 
\begin{align}
    E_1^f&\lesssim  \Vert  f_2\Vert_{H^N}\Vert\sqrt{\tfrac{\p_t m}m } f\Vert_{H^N}\Vert  \nabla_t f\Vert_{H^N}.\label{eq:Ef1}
\end{align}
Integrating in time and using $\gamma-\beta -\frac 1 2 \ge 0 $ yields
\begin{align*}
\int_0^{T} E_1^f \dd\tau&\le C C_1  \mu^{-\frac 1 2 }(\delta\mu^{\gamma-\beta_1})^2(\delta^2 \mu^{\gamma-\beta_1 })\le \frac C{C_1}(C_1\delta^2 \mu^{\gamma-\beta_1 })^2.
\end{align*}
For $E_2^f$ we estimate 
\begin{align*}
    E_2^f=\vert \langle   m^{-2}f_2,   \p_y\p_y^{-2} f_= \p_x f\rangle_{H^N}\vert   &\le\sum_{k\neq 0 }\iint \vert k \vert \vert m^{-2} f_2(k,\eta)\vert \vert \tfrac 1{\vert \eta-\xi\vert} f_=(\eta-\xi)\vert \vert f (k,\xi) \vert \dd (\eta,\xi). 
\end{align*}
We use that 
\begin{align}
\begin{split}\label{eq:kbound}
   \vert k\vert &\le \frac {\vert k,\xi-kt\vert}{\langle t-\frac \xi k \rangle }\textbf{1}_{\vert \eta-\xi\vert \le \vert k,\xi\vert }+\textbf{1}_{\vert \eta-\xi\vert > \vert k,\xi\vert }\vert  k \vert \\
   &\le  {\vert k,\xi-kt\vert}\sqrt {\tfrac {\p_t m} m  (k,\eta )} \langle  \eta-\xi \rangle \textbf{1}_{\vert \eta-\xi\vert \le \vert k,\xi\vert }+\textbf{1}_{\vert \eta-\xi\vert > \vert k,\xi\vert }\vert  k \vert 
\end{split}
\end{align}
to infer 
\begin{align*}
    E_2^f&\le C \Vert \sqrt{\tfrac {\p_t m}m }  f _2\Vert_{H^N }\Vert \vert \p_y\vert^{-1} f_=\Vert_{H^N }\Vert \Lambda_t  f \Vert_{H^N } + \Vert  f_2 \Vert_{H^N }\Vert \vert \p_y\vert^{-1} f_=\Vert_{H^N }\Vert   f \Vert_{H^N }.
\end{align*}
Integrating in time  and using $\gamma-\beta -\frac 1 2 \ge 0 $ gives
\begin{align}
    \int_0^{T} E_2^f\dd\tau&\le \frac C{C_1}(C_1\delta^2 \mu^{\gamma-\beta_1 })^2.\label{eq:Ef2}
\end{align}
To bound $E_3^f$ we use $(\nabla^\perp g\cdot \nabla) \Delta_t g= \nabla_t^\perp\cdot ((\nabla^\perp g\cdot \nabla) \nabla^\perp_t g)$ to estimate 
\begin{align*}
    E_3^f &\le C \Vert \nabla_t f_2 \Vert_{H^N} \Vert  g\Vert_{H^{N+1}} \Vert \nabla_t  g\Vert_{H^{N+1}}. 
\end{align*}
Integrating in time yields
\begin{align}
    \int_0^{T} E_3^f \dd\tau  &\le C \delta^2 \mu^{2\gamma-1} (C_1 \delta^2\mu^{\gamma-\beta_1})\lesssim \frac 1{C_1}(C_1 \delta^2\mu^{\gamma-\beta_1})^2. \label{eq:Ef3}
\end{align}
Combining the estimates \eqref{eq:Ef1}, \eqref{eq:Ef2} and \eqref{eq:Ef3} we obtain 
\begin{align*}
    \int_0^{T} E_1^f+ E_2^f+ E_3^f \dd\tau & \lesssim \frac 1{C_1}(C_1 \delta^2\mu^{\gamma-\beta_1})^2
\end{align*}
Thus, we improve estimate \eqref{eq:f2est2} for  $C_1>0$ large enough. 

\textbf{Bound on $g_2$,} we split 
\begin{align*}
    \frac 12  \Vert m^{-1}g_2\Vert_{L^\infty_T H^{N+1}}^2 &+\int^{T}_0\mu\Vert m^{-1}\nabla_\tau g_2 \Vert_{H^{N+1}}^2+\Vert m^{-1}\sqrt{\tfrac {\p_t m}m } g_2 \Vert_{H^{N+1}}^2\dd\tau\le \int^{T}_0  E_1^g+E_2^g+E_3^g\dd\tau,  \\
    E_1^g&= \vert \langle m^{-2}g_2 , (\nabla^\perp \Lambda_t^{-2}f_0\cdot \nabla )(g-g_0)\rangle_{H^{N+1}}\vert ,\\
    E_2^g&=\vert \langle m^{-2}g_2 ,(\nabla^\perp_t \Lambda_t^{-2}f_{2,\neq}\cdot \nabla_t )g\rangle_{H^{N+1}}\vert , \\
    E_3^g&=\vert \langle m^{-2}g_2 ,(\nabla^\perp_t \Lambda_t^{-2}f_{2,=}\cdot \nabla_t )g\rangle_{H^{N+1}}\vert ,
\end{align*}
which we estimate separately. To bound $E_1^g$ we estimate 
\begin{align*}
    E_1^g&\lesssim  \Vert g_2 \Vert_{H^{N+1}}\left(\Vert \Lambda_t^{-2} f_0\Vert_{H^{N+2}}\Vert \nabla (g-g_0) \Vert_{L^\infty }+\Vert \Lambda_t^{-1} f_0\Vert_{L^\infty}\Vert \nabla_t (g-g_0) \Vert_{H^{N+1} }\right). 
    %&\lesssim  \Vert g_2 \Vert_{H^N}\left(\mu^{-\frac 23 } \langle t-\tfrac {\eta_0}{k_0} \rangle^{-2}\Vert f_0\Vert_{H^{N}}\Vert  (g-g_0) \Vert_{H^N  }+t \mu \Vert  f_0\Vert_{H^N }\Vert \nabla_t (g-g_0) \Vert_{H^{N+1} }\right) 
\end{align*}
We use that $f_0$ is supported on $k=k_0$ and $\vert \eta-\eta_0\vert\le 1$, where $\eta_0 \approx \mu^{-\frac 13 }$ and so for all $j_1,j_2\ge 0 $ it holds
\begin{align*}
    \Vert \Lambda^{-j_2}_t  f_0\Vert_{H^{j_1}}&\lesssim  \langle t-\tfrac {\eta_0}{k_0} \rangle^{-j_2} \Vert   f_0\Vert_{H^{j_1}},\\
    \Vert   f_0\Vert_{H^{j_1}}&\lesssim \mu^{\frac {j_2-j_1}3}\Vert   f_0\Vert_{H^{j_2}}.
\end{align*}

Therefore,
\begin{align*}
    E_1^g&\lesssim   \left(\mu^{-\frac 23 }  \langle t-\tfrac {\eta_0}{k_0} \rangle^{-2}\Vert \nabla  (g -  g_0 )\Vert_{L^\infty} + \mu  \langle t-\tfrac {\eta_0}{k_0} \rangle^{-1}(\Vert \nabla_t g \Vert_{H^{N+1} }+\Vert \nabla_t g_0 \Vert_{H^{N+1} }) \right)\Vert g_2 \Vert_{H^N}\Vert f_0\Vert_{H^{N}}.
\end{align*}
Furthermore, we have the low-frequency bound 
\begin{align*}
    \Vert \nabla (g-g_0)\Vert_{L^\infty}\lesssim \Vert g-g_0\Vert_{H^{\frac 52 }} &\le \int_0^{T} \Vert (\nabla^\perp \Lambda_t^{-2}  f\cdot \nabla) g\Vert_{H^{\frac 52 }}\dd \tau  \lesssim \Vert f\Vert_{ L^\infty_T H^5}\Vert g\Vert_{L^\infty_T H^{\frac 72 }}\lesssim \delta^2 \mu^{2\gamma-\beta_1 }.
\end{align*}
Using the explicit representation of $g_0$, \eqref{eq:FG}, and the bootstrap assumption, we obtain after integrating in time 
\begin{align}
    \int_0^{T}  E_1^g \dd\tau&\lesssim  (\delta^2 \mu^{2\gamma-\beta_1 -\frac 23 }+\delta \mu^{\frac 12 }  )\delta \mu^{\gamma }(\delta^3\mu^{2\gamma-\beta_1-\frac 23  }) \lesssim( \tfrac 1 {C_2}+\delta^{-1}\mu^{\frac 76+\beta_1-\gamma  } ) (C_2\delta^3\mu^{2\gamma-\beta_1-\frac 23  })^2. \label{eq:Eg1}
\end{align}
To bound $E_2^g$ we use Plancherel's identity
\begin{align*}
     E_2^g&=\vert \langle m^{-2}g_2 ,(\nabla^\perp_t \Lambda_t^{-2}f_{2,\neq}\cdot \nabla_t )g\rangle_{H^{N+1}}\vert \\
     &\lesssim  \sum_{\substack{k,\tilde k\\ k\neq \tilde k }} \iint \dd(\xi,\eta) \frac {\vert \eta \tilde k-k\xi\vert }{\vert k-\tilde k,\eta-\xi-(k-\tilde k)t \vert^2 }\vert g_2(k,\eta)\vert  \vert f_{2}(k-\tilde k,\eta-\xi)\vert \vert g(\tilde k,\xi)\vert. 
\end{align*}
For $k-\tilde k\neq 0$ we distinguish between high and low frequencies and use the weight $m$ to estimate 
\begin{align*}
     \frac {\vert \eta \tilde k-k\xi\vert }{\vert k-\tilde k,\eta-\xi-(k-\tilde k)t \vert^2 }&\lesssim \textbf{1}_{\vert k-\tilde k,\eta-\xi\vert\ge \vert \tilde k,\xi\vert} \left(1+\frac {t^2 }{\langle t-\frac {\eta-\xi}{k-\tilde k}\rangle^2 } \right)\frac{\vert \tilde k,\xi\vert}{\langle k-\tilde k,\eta-\xi\rangle}\\
     &\quad + \textbf{1}_{\vert k-\tilde k,\eta-\xi\vert\le \vert \tilde k,\xi\vert}   \frac {\langle t \rangle \vert \tilde k,\xi-\tilde kt\vert}{\vert k-\tilde k,\eta-\xi-(k-\tilde k)t \vert^2 }  \vert k-\tilde k,\eta-\xi\vert \\
     &\lesssim \textbf{1}_{\vert k-\tilde k,\eta-\xi\vert\ge \vert \tilde k,\xi\vert}\left (1+t^2  \sqrt{\tfrac{\p_tm }m(k-\tilde k,\eta-\xi) } \sqrt{\tfrac{\p_tm }m(\tilde k,\xi) } \right)\frac {\vert \tilde k,\xi\vert^3}{\langle k-\tilde k,\eta-\xi\rangle}\\
     &\quad + \textbf{1}_{\vert k-\tilde k,\eta-\xi\vert\le \vert \tilde k,\xi\vert}  \langle t \rangle^{-1} \vert \tilde k,\xi-\tilde kt\vert \vert k-\tilde k,\eta-\xi\vert^3.
\end{align*}
This yields the estimate 
\begin{align*}
    E_2^g &\lesssim \left( \Vert  g_2\Vert_{H^{N+1}} \Vert  f_{2,\neq} \Vert_{H^N}+\mu^{-\frac 23 } \Vert \sqrt{\tfrac {\p_t m}m } g_2\Vert_{H^{N+1}}  \Vert \sqrt{\tfrac {\p_t m}m } f_{2,\neq} \Vert_{H^N}\right) \Vert \Lambda^3  g \Vert_{L^\infty} \\
    &\quad + \langle t \rangle^{-1} \Vert  g_2\Vert_{H^{N+1}} \Vert \Lambda^2  f_{2,\neq}\Vert_{L^\infty } \Vert \nabla_t g \Vert_{H^{N+1}}.
\end{align*}
After integrating in time, using \eqref{eq:FG}, the bootstrap assumption and $\gamma-\beta -\frac 1 2 \ge 0 $ we infer 
\begin{align}
    \int_0^T  E_2^g \dd\tau &\le \frac {C C_1}{C_2}(C_2\delta^2\mu^{2\gamma-\beta_1-\frac 23  })^2.\label{eq:Eg2}
\end{align}
To estimate $E_3^g$ we use Plancherel's identity 
\begin{align*}
    E_3^g=\vert \langle m^{-2}g_2 ,\p_y\p_y^{-2}f_{2,=}\p_x g\rangle_{H^{N+1}}\vert 
    &\le \sum_{k \neq \tilde k } \int \dd (\eta,\xi) \langle k,\eta\rangle^{2N} g_2(k,\eta) \frac{\vert f_2(0,\eta-\xi)\vert }{\vert \eta-\xi\vert} \vert k\vert \vert g(k,\xi)\vert 
\end{align*}
We use again \eqref{eq:kbound} to bound 
\begin{align*}
    E_3^g&\lesssim \left(\Vert\sqrt{\tfrac{\p_t m }{m}} g_2 \Vert_{H^{N+1}}\Vert\Lambda_t  g_2 \Vert_{H^{N+1}}+\Vert g_2 \Vert_{H^{N+1}}\Vert  g \Vert_{H^{N+1}}    
    \right)\Vert \vert \p_y\vert^{-1} f_{2,=}\Vert_{H^{N+1}}.
\end{align*}
Integrating in time yields 
\begin{align}
    \int_0^T E^g_3 \dd \tau &\le C \mu^{-\frac 12 }\delta \mu^{\gamma}  \delta \mu^{\gamma-\beta_1}(C_2
    \delta^3  \mu^{2\gamma-\frac 23 -\beta_1}) \lesssim \frac 1{C_2} \delta^{-1} \mu^{\frac 1 6 } (C_2
    \delta^3  \mu^{2\gamma-\frac 23 -\beta_1})^2,
    \label{eq:Eg3}.
\end{align}

Combining the estimates \eqref{eq:Eg1}, \eqref{eq:Eg2} and \eqref{eq:Eg3} we infer 
\begin{align*}
    \int_0^T E^g_1+E^g_2+E^g_3 \dd \tau &\le   \left(\frac C{C_2} (C_1+\delta^{-1} \mu^{\frac 1 6 })+\delta^{-1}\mu^{\frac 76+\beta_1-\gamma  } \right)  (C_2
    \delta^3  \mu^{2\gamma-\frac 23 -\beta_1})^2,
    \label{eq:Eg3}.
\end{align*}

Thus for $\mu_0, \delta>0$ small enough such that $\mu\le \delta^{6}$ and $C_2>0$ large enough we improve the estimate \eqref{eq:g2est}, by local existence this contradicts that that the maximal time $T<\mu^{-\frac 13}$. Therefore, both \eqref{eq:g1est} and \eqref{eq:g2} hold for all times $t\le \mu^{-\frac 1 3} $. This yields \eqref{eq:g2} and together with \eqref{eq:g1est} we optain Proposition \ref{prop:maggr}. 
\end{proof}

\subsection*{Data availability}
No data was used for the research described in the article.

\subsection*{Acknowledgments}

The author would like to thank Michele Coti Zelati, Rajendra Beekie, Víctor Navarro-Fernández, and Yupei Huang for valuable discussions, and in particular Michele Coti Zelati for his insightful comments that helped reorganize the manuscript and improve the introduction.

N. Knobel is supported by ERC/EPSRC Horizon Europe Guarantee EP/X020886/1. The author declares that he has no conflict of interest.

\appendix
\section{Properties of the Weights }
In the following, we summarize the properties of the weights we use. The proofs are standard and can be obtained by simple calculations (see for example \cite{liss2020sobolev,knobel2025suppression})

\begin{lemma}[Properties of $M_\mu$]\label{lem:Mmu}
    The weight $M_\mu$ is bounded and for $k\neq 0$ it satisfies 
    \begin{align*}
    \mu^{\frac 1 3 } &\le  c\tfrac {\p_t M_\mu} {M_\mu}(k,\eta, l  )  + 2c \mu(1+\vert t-\tfrac \eta k \vert^2).
\end{align*}
\end{lemma}

\begin{lemma}[Properties of $m$]\label{lem:m} The weight $m$ is bounded and for $k \neq 0$ it satisfies 
         \begin{align*}
              \frac 1 {\langle\frac{\eta}{k}- t\rangle^{2} }\le  {\frac {\p_t   m} {  m} (\tilde k,\xi, \tilde l )}\langle k-\tilde k ,\eta-\xi \rangle^5. 
         \end{align*}
 \end{lemma}

\bibliographystyle{alpha} %alpha IEEEtran alphaabbr

\bibliography{library}

\end{document}